\documentclass[12pt,a4paper]{amsart}
\usepackage{color,graphicx,fullpage,amssymb}
\usepackage{url}
\usepackage[colorlinks=true]{hyperref}
\usepackage{doi}
\usepackage[colorinlistoftodos]{todonotes}

\newtheorem{theorem}{Theorem}[section]
\newtheorem{lemma}[theorem]{Lemma}
\newtheorem{proposition}[theorem]{Proposition}
\newtheorem{corollary}[theorem]{Corollary}

\theoremstyle{definition}
\newtheorem{definition}[theorem]{Definition}
\newtheorem{remark}[theorem]{Remark}

\newcommand{\Q}{\mathbb{Q}}
\newcommand{\F}{\mathbb{F}}

\newcommand{\Qp}{\mathbb{Q}_p}
\newcommand{\Cp}{\mathbb{C}_p}
\newcommand{\Zp}{\mathbb{Z}_p}
\newcommand{\N}{\mathbb{N}}
\newcommand{\R}{\mathbb{R}}

\newcommand{\Z}{\mathbb{Z}}
\newcommand{\C}{\mathbb{C}}

\newcommand{\dd}{\mathrm{d}}
\newcommand{\X}{\mathfrak{X}}
\newcommand{\g}{\mathfrak{g}}
\newcommand{\e}{\mathrm{e}}
\newcommand{\ii}{\mathrm{i}}
\newcommand{\orm}{\mathrm{O}}
\newcommand{\letpprime}{Let $p$ be a prime number}
\newcommand{\ocal}{\mathcal{O}}
\renewcommand{\le}{\leqslant}
\renewcommand{\ge}{\geqslant}

\DeclareMathOperator{\rank}{rank}
\DeclareMathOperator{\ord}{ord}
\numberwithin{equation}{section}
\title{The $p$-adic Jaynes-Cummings model in symplectic geometry}
\author[Luis Crespo, \'Alvaro Pelayo]{Luis Crespo\,\,\,\,\,\, \'Alvaro Pelayo}

\address{Luis Crespo,
	Departamento de Matem\'{a}ticas, Estad\'{i}stica y Computaci\'{o}n, Universidad de Cantabria, Av.~de Los Castros 48, 39005 Santander, Spain}
\email{luis.cresporuiz@unican.es}

\address{\'Alvaro Pelayo,
	Facultad de Ciencias Matem\'aticas,
	Universidad Complutense de Madrid, 28040 Madrid, Spain, and Real Academia de Ciencias Exactas, F\'isicas y Naturales de Espa\~na}
\email{alvpel01@ucm.es}

\begin{document}

\begin{abstract}
	The notion of classical $p$-adic integrable system on a $p$-adic symplectic manifold was proposed by Voevodsky, Warren and the second author a decade ago in analogy with the real case.
	In the present paper we introduce and study, from the viewpoint of symplectic geometry and topology, the basic properties of the $p$-adic version of the classical Jaynes-Cummings model. The Jaynes-Cummings model is a fundamental example of integrable system going back to the work of Jaynes and Cummings in the 1960s, and which applies to many physical situations, for instance in quantum optics and quantum information theory. Several of our results depend on the value of $p$: the structure of the model depends on the class of the prime $p$ modulo $4$ and $p=2$ requires special treatment.
\end{abstract}

\maketitle

\section{Introduction}
Symplectic geometry is concerned with the study of symplectic manifolds, that is, smooth manifolds $M$ endowed with a closed non-degenerate $2$-form. The subject has its roots in the study of planetary motions in the XVII and XVIII centuries, and has undergone extensive developments since then. We refer to \cite{Eliashberg,Pelayo,Schlenk,Weinstein} for surveys on symplectic geometry which also mention historical developments.

Usually in symplectic geometry one works with smooth manifolds and differential forms over the real numbers, and there is a large body of work for this case. On the other hand, in this paper we are going to work with the case in which the field of real numbers $\R$ is replaced by the field of $p$-adic numbers $\Qp$, where $p$ is a fixed prime number. Recall that $\Qp$ can be defined as a completion of $\Q$ with respect to a non-archimedean absolute value (as we will recall in Appendix \ref{app:prelim}). In \cite{PVW} V. Voevodsky, M. Warren and the second author gave a construction of $\Qp$ from the point of view of homotopy type theory and Voevodsky's Univalent Foundations. Their long-term goal was constructing the theory of $p$-adic integrable systems on $p$-adic symplectic manifolds, with an eye on eventually formalizing this theory using proof assistants. In \cite[Section 7]{PVW} the authors propose a notion of $p$-adic integrable system and sketch a possible research program concerning these systems. We recommend \cite{APW,PelWar,PelWar2} for a concise introduction to homotopy type theory and Voevodsky's Univalent Axiom.

The present paper continues the work sketched by Pelayo, Voevodsky and Warren in \cite[Section 7]{PVW}. More precisely, we start to develop the basics of $p$-adic integrable systems proposed therein, by focusing on working out the $p$-adic analog of an integrable system of great importance: the \textit{classical Jaynes-Cummings model} (Figure \ref{fig:real-JC}), also known in the mathematics community as the \textit{coupled spin-oscillator}, because it comes from coupling an oscillator and a spin in a non-trivial way. It is a fundamental example of integrable system with two degrees of freedom. The Jaynes\--Cummings model was originally introduced \cite{JayCum}  to give a description of  the interaction between an atom prepared in a mixed state with a quantum particle in an optical cavity.  The Jaynes\--Cummings model has been studied or applied in different contexts, and from several viewpoints, and
is relevant across a number of areas within physics, chemistry and mathematics,  the reason being that it  represents the simplest possible way  to have a finite dimensional state, like a spin, be in interaction with an oscillator. For instance, the model has been found to apply to physical situations in the contexts of quantum physics, quantum optics or quantum information theory~\cite{GuAg, RaBrHa, ShKn}. It is also of high interest in mathematical physics, see for
 instance~\cite{BaCaDo, BaDo12, BaDo15}, and symplectic geometry~\cite{AlDuHo,LFPeVN16, PelVuN}.

Unlike in \cite{PVW}, in the present paper we use the more familiar non-constructive definition of $\Qp$, since the purpose of the current paper is to start developing the theory of $p$-adic integrable systems, and it does not matter which definition of $\Qp$ one uses. However, for problems involving the use of algorithms or constructive proofs one should use the construction of $\Qp$ which is given in \cite{PVW}.

We believe that it is crucial to have at least an important example of $p$-adic integrable system worked out in detail before trying to develop a general theory, hence this paper. The example of the Jaynes-Cummings model shows that a general theory will be subtle, yet worthwhile pursuing in our opinion. In particular, understanding this example requires combining techniques from symplectic geometry and $p$-adic analysis and some of the conclusions were surprising to us. For instance, the structure of the system varies significantly depending on the class of the prime $p$ modulo $4$, and the case of $p=2$ often requires a special treatment.

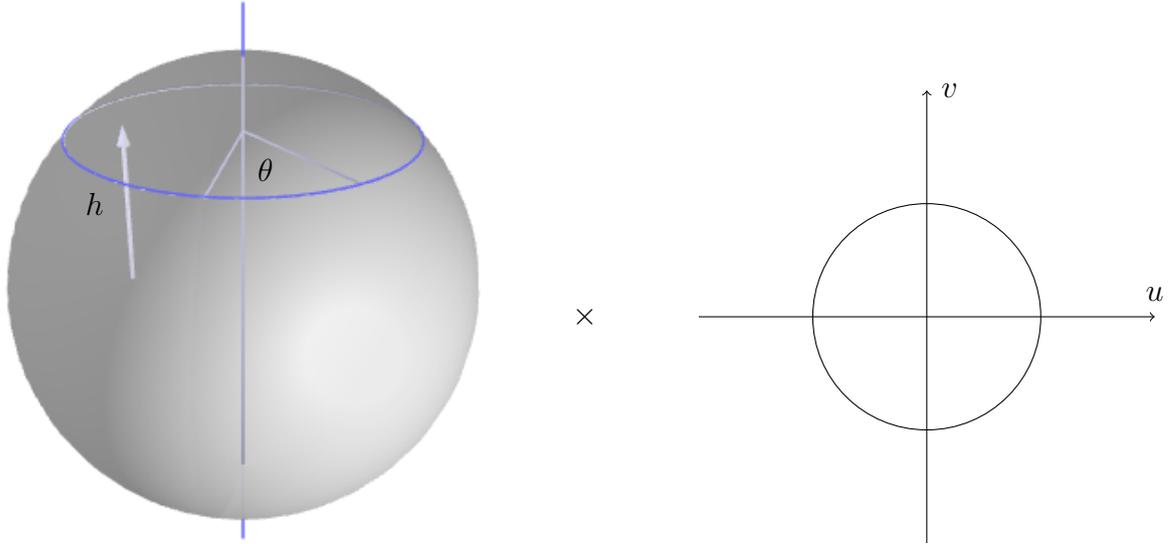
\begin{figure}
	\begin{tikzpicture}[scale=1.5]
		\node (esfera) at (0,0) {\includegraphics[trim=5cm 1.5cm 5cm 1.5cm,height=7.5cm]{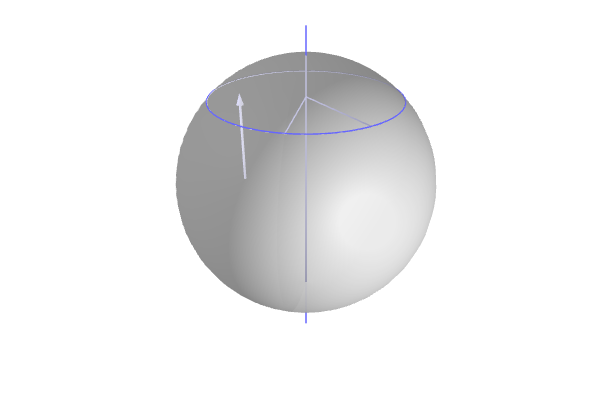}};
		\node (theta) at (0.2,1.3) {$\theta$};
		\node (h) at (-1.3,1) {$h$};
		\node (x) at (3,0) {$\times$};
		\draw[->] (4,0) -- (8,0);
		\draw[->] (6,-2) -- (6,2);
		\draw (6,0) circle (1);
		\node (u) at (8,0.2) {$u$};
		\node (v) at (6.2,2) {$v$};
	\end{tikzpicture}
	\caption{The real Jaynes-Cummings model on $\mathrm{S}^2\times\R^2$.}
	\label{fig:real-JC}
\end{figure}

The following result, from the paper \cite{PelVuN} by V\~u Ng\d oc and the second author, describes the basic properties of the classical ``real'' Jaynes-Cummings model.

\begin{proposition}[{Real Jaynes Cummings model \cite[Prop. 2.1]{PelVuN}}]\label{prop:JC-real}
	The coupled spin-oscillator $(\mathrm{S}^2\times\R^2,\omega_{\mathrm{S}^2}\oplus\omega_0,(J,H))$, defined by
	\[
	\left\{
	\begin{aligned}
	J(x,y,z,u,v)&=\frac{u^2+v^2}{2}+z; \\
	H(x,y,z,u,v)&=\frac{ux+vy}{2},
	\end{aligned}
	\right.
	\]
	with coordinates $(\theta,h)=(x,y,z)$ on $\mathrm{S}^2$ and $(u,v)$ on $\R^2$ and symplectic form $\omega=\dd\theta\wedge\dd h+\dd u\wedge\dd v$, is an integrable system, meaning that the Poisson bracket $\{J,H\}$ vanishes everywhere.
	
	In addition, the map $J$ is the momentum map for the Hamiltonian circle action of $\mathrm{S}^1$ on $\mathrm{S}^2\times\R^2$ that rotates simultaneously horizontally about the vertical axis on $\mathrm{S}^2$, and about the origin on $\R^2$.
	
	The critical points of the coupled spin-oscillator are non-degenerate and of elliptic-elliptic, transversally-elliptic or focus-focus type. It has exactly one focus-focus singularity at the ``North Pole'' $(0,0,1,0,0)\in\mathrm{S}^2\times\R^2$ and one elliptic-elliptic singularity at the ``South Pole'' $(0,0,-1,0,0)$.
\end{proposition}

In the paper \cite{PelVuN} the following information about the Jaynes-Cummings model was either also determined or it can be easily deduced from it (either from considerations in the paper or by general theory, which is not available to us in the $p$-adic case):
\begin{enumerate}
	\item \textit{Classical spectrum:} the image $F(\mathrm{S}^2\times\R^2)$ of $F$ (Figure \ref{fig:real-JC-image});
	\item \textit{Fibers:} the (singular and regular) fibers $F^{-1}({c})$, $c\in F(\mathrm{S}^2\times\R^2)$ (Figure \ref{fig:real-JC-image});
	\item \textit{Critical sets/values:} the sets of critical points and critical values of $F$;
	\item \textit{Ranks of critical points:} the rank of all critical points of $F$.
\end{enumerate}

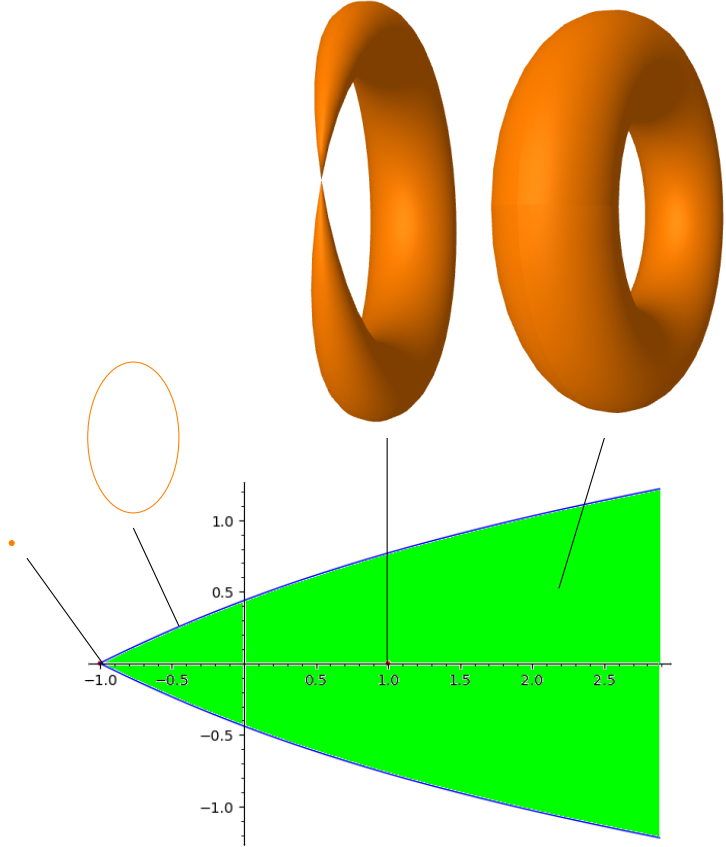
\begin{figure}
	\begin{tikzpicture}[scale=2]
		\node (im) at (0.8,0) {\includegraphics[width=8cm]{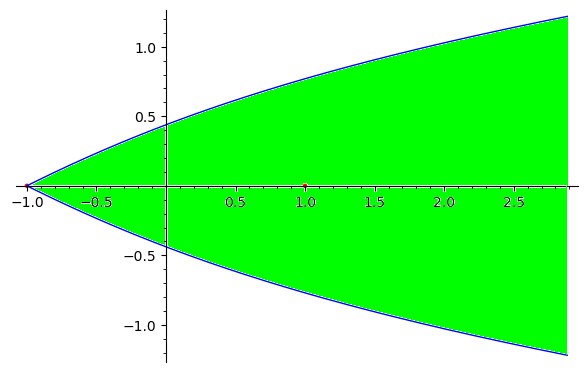}};
		\draw (-1,0)--(-1.5,0.7);
		\fill[orange] (-1.6,0.8) circle (0.02);
		\draw (-0.5,0.25) -- (-0.8,0.9);
		\draw[orange] (-0.8,1.5) ellipse (0.3 and 0.5);
		\draw (0.87,0) -- (0.87,1.5);
		\node (foco) at (0.87,3) {\includegraphics[height=6cm,trim=20cm 0 20cm 0,clip]{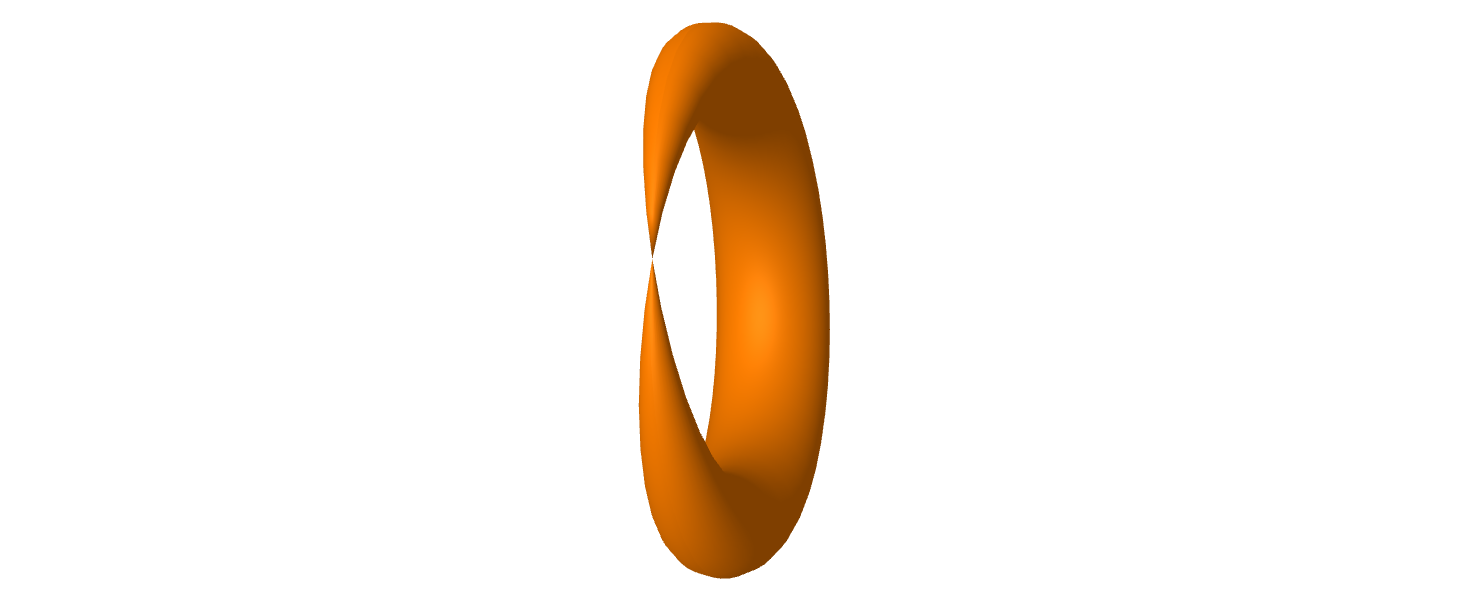}};
		\draw (2,0.5) -- (2.3,1.5);
		\node (toro) at (2.5,3) {\includegraphics[height=6cm,trim=19cm 0 19cm 0,clip]{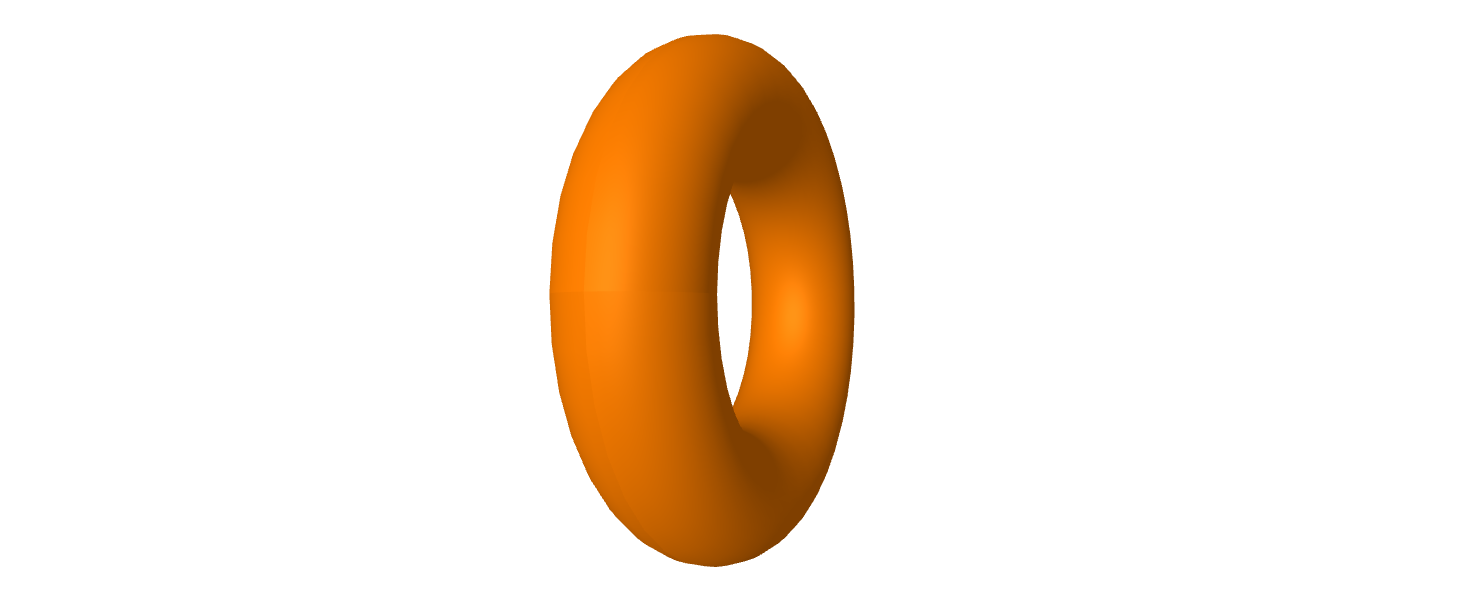}};
	\end{tikzpicture}
	\caption{Image and fibers of the Jaynes-Cummings model of Proposition \ref{prop:JC-real}. The blue curve consists of rank $1$ critical points, and the two red points are rank $0$. The Jaynes\--Cummings model is an example of a class of integrable systems called \emph{semitoric systems}, which were classified in the so called simple case in \cite{PelVuN-semitoric,PelVuN-construct} by the second author and V\~u Ng\d oc and in the non simple case by Palmer, the second author and Tang in \cite{PPT}. Here being simple means that there is at most one pinched point of rank 0 per singular fiber of $J$, a condition which is satisfied by the Jaynes\--Cummings model. All fibers of this system are connected: points, circles, $2$-tori (generic fiber) or a pinched torus. There is exactly one fiber which is not a manifold: the fiber over $(1,0)$, which is the pinched torus.}
	\label{fig:real-JC-image}
\end{figure}

All of these results have analogs if one replaces $\mathrm{S}^2$ by the $p$-adic sphere $\mathrm{S}^2_p$ and $\R^2$ by the $p$-adic plane $(\Qp)^2$, while keeping the same formulas for $\omega$ and $F$. These analogs are more complicated to formulate than in the real case and depend on the value of $p$.

We refer to \cite{DjoDra,Dragovich-quantum,Dragovich-harmonic,DKKV} for works which deal with different geometric aspects of $p$-adic geometry and mathematical physics, and to \cite{Gouvea,Koblitz,Schneider} for books on $p$-adic numbers, $p$-adic analysis, $p$-adic manifolds, and related concepts.

\subsection*{Structure of the paper}
The paper is organized as follows.
\begin{itemize}
	\item In Section \ref{sec:main} we state the main results of the paper.
	\item In Section \ref{sec:integrable} we recall the notions of $p$-adic symplectic manifold and $p$-adic integrable system.
	\item In Section \ref{sec:circ-oscillator} we analyze the structure of $p$-adic circles, in order to calculate the image and fibers of the $p$-adic oscillator.
	\item In Section \ref{sec:spin} we study the $p$-adic spin system, and also calculate its image and fibers.
	\item In Section \ref{sec:JC} we start studying the coupling of both models (oscillator and spin).
	\item In Section \ref{sec:nondeg} we deduce the normal forms of the critical points of this coupling.
\end{itemize}

The paper concludes with three appendices in which we recall some facts we need and prove some technical results that are used in the main part of the paper. Concretely:
\begin{itemize}
	\item In Appendix \ref{app:prelim} we give all necessary preliminaries about the $p$-adic numbers, focusing at $p$-adic analysis and trigonometric functions.
	\item In Appendix \ref{app:vector} we give some definitions about vector fields and forms on $p$-adic analytic manifolds.
	\item In Appendix \ref{app:actions} we review $p$-adic Hamiltonian actions.
\end{itemize}

\subsection*{Acknowledgements}
The first  author is funded by grants PID2019-106188GB-I00 and PID2022-137283NB-C21 of
MCIN/AEI/10.13039/501100011033 / FEDER, UE and by project CLaPPo (21.SI03.64658) of Universidad
de Cantabria and Banco Santander.

The second author is funded by a BBVA (Bank Bilbao Vizcaya Argentaria) Foundation Grant for Scientific Research Projects with title \textit{From Integrability
	to Randomness in Symplectic and Quantum Geometry}. He thanks the Dean of the
School of Mathematics Antonio Br\'u and the Chair
of the Department of Algebra, Geometry and Topology at the Complutense University of Madrid, Rutwig Campoamor, for their support
and excellent resources he is being provided with to carry out the BBVA project.

We are very thankful to Mar\'ia In\'es de Frutos, Antonio D\'iaz-Cano and Juan Ferrera for many helpful discussions and feedback on a preliminary version of this paper.

\section{Main results}\label{sec:main}

The following results describe the basic properties of the $p$-adic Jaynes-Cummings model. In contrast with the real case and for reasons which will become clear later, we develop the theory of this system viewed as a $p$-adic analytic map on a $p$-adic analytic manifold, which is the situation for which we later give general definitions. Below the notation $\ord(x)$, where $x\in\Qp$, refers to \emph{$p$-adic order}, that is, an integer such that
\[|x|_p=p^{-\ord_p(x)}.\]
The concepts of \emph{critical point}, \emph{rank of a critical point} and \emph{non-degeneracy for a critical point} which appear in the statement below are defined in the $p$-adic case, via direct analogy with the real case. The concept of \emph{action of a $p$-adic Lie group} is defined in Appendix \ref{app:actions}. Finally, the $p$-adic analog of the circle $\mathrm{S}^1$, and in general the $n$-sphere $\mathrm{S}^n$ for $n\in\N$, is the \textit{$p$-adic $n$-sphere}, which is defined as
\[\mathrm{S}^n_p=\left\{(x_1,\ldots,x_{n+1})\in(\Qp)^{n+1}\,\,\middle|\,\, \sum_{i=1}^{n+1}x_i^2=1\right\}.\]
We refer to Figure \ref{fig:padic-JC-image} for an illustration of the following result when $p=5$.

\begin{theorem}[Classical spectrum and critical points of $p$-adic Jaynes-Cummings model]\label{thm:JC1}
	Let $p$ be a prime number. Let $F=(J,H):\mathrm{S}_p^2\times(\Qp)^2\to(\Qp)^2$ be the $p$-adic Jaynes-Cummings model, that is, the $p$-adic analytic map given by
	\[\left\{\begin{aligned}
		J & = \frac{u^2+v^2}{2}+z; \\
		H & = \frac{ux+vy}{2},
	\end{aligned}\right.\]
	where $(x,y,z)\in\mathrm{S}_p^2$, $u,v\in(\Qp)^2$ and $\mathrm{S}_p^2\times(\Qp)^2$ is endowed with the $p$-adic analytic symplectic form $\dd\theta\wedge\dd h+\dd u\wedge\dd v$, where $(\theta,h)$ are angle-height coordinates on $\mathrm{S}_p^2\times(\Qp)^2$.
	Then the following statements hold.
	\begin{enumerate}
		\item The map $F=(J,H):\mathrm{S}_p^2\times(\Qp)^2\to(\Qp)^2$ is a $p$-adic analytic integrable system, that is, $\{J,H\}=0$ (Theorem \ref{thm:JC-general}).
		\item The map $J:\mathrm{S}_p^2\times(\Qp)^2\to\Qp$ is the momentum map of the Hamiltonian action of $\mathrm{S}^1_p$ that rotates simultaneously horizontally about the vertical axis on $\mathrm{S}^2_p$, and about the origin on $(\Qp)^2$ (Theorem \ref{thm:JC-general}(1)).
		\item The image of $F$, that is, the classical spectrum of $F$, is given as follows.
		\begin{enumerate}
			\item If $p\ne 2$, the classical spectrum is $F(\mathrm{S}_p^2\times(\Qp)^2)=(\Qp)^2$ (Corollary \ref{cor:JC-surj}).
			\item \label{item:JC-image2}If $p=2$, the classical spectrum $F(\mathrm{S}_p^2\times(\Qp)^2)$ contains all points $(j,h)$ with $\ord(j)\ge 1$ and $\ord(h)\ge 0$ and some points with $\ord(j)=0$ and $\ord(h)\ge -1$, $\ord(j)<0$ odd and $\ord(h)=\ord(j)/2-1$, and $\ord(j)<0$ even, and $\ord(h)\ge (\ord(j)-1)/2$ (Propositions \ref{prop:JC-image-nec} to \ref{prop:JC-image-suf}).
		\end{enumerate}
		\item \label{item:JC-critical}The set of critical points of $F:\mathrm{S}_p^2\times(\Qp)^2\to(\Qp)^2$ is given as follows (Theorem \ref{thm:JC-general}(2)).
		\begin{enumerate}
			\item the set of rank $0$ points is $\{(0,0,-1,0,0),(0,0,1,0,0)\}$.
			\item the set of rank $1$ points is
			\[\Big\{(au,av,-a^2,u,v)\,\,\Big|\,\, a,u,v\in\Qp, (u,v)\ne(0,0),a^2(u^2+v^2)+a^4=1\Big\}.\]
		\end{enumerate}
	\end{enumerate}
\end{theorem}

\begin{remark}
	In part (\ref{item:JC-image2}) of the previous theorem, we use ``contains'' because we do not have a complete description of the image of the system for $p=2$. Deciding whether some points are in the image seems more complicated than for other primes, partly because $\mathrm{S}_2^2$ is compact while $\mathrm{S}_p^2$ is not compact for any other $p$.
\end{remark}

The following is the most interesting result of the paper, and the one for which the calculations are more involved. A depiction of this result is given in Figure \ref{fig:padic-JC-image}.

\pagebreak
\begin{theorem}[Fibers of $p$-adic Jaynes-Cummings model]\label{thm:JC2}
	Let $p$ be a prime number. Let $F=(J,H):\mathrm{S}_p^2\times(\Qp)^2\to(\Qp)^2$ be the $p$-adic Jaynes-Cummings model, that is, the $p$-adic analytic map given by
	\[\left\{\begin{aligned}
	J(x,y,z,u,v) & = \frac{u^2+v^2}{2}+z; \\
	H(x,y,z,u,v) & = \frac{ux+vy}{2},
	\end{aligned}\right.\]
	where $(x,y,z)\in\mathrm{S}_p^2$, $u,v\in(\Qp)^2$ and $\mathrm{S}_p^2\times(\Qp)^2$ is endowed with the $p$-adic analytic symplectic form $\dd\theta\wedge\dd h+\dd u\wedge\dd v$, where $(\theta,h)$ are angle-height coordinates on $\mathrm{S}_p^2\times(\Qp)^2$.
		The fibers of $F$ are given as follows.
		\begin{enumerate}
			\item Suppose that $p\not\equiv 1\mod 4$ (Theorems \ref{thm:JC-topology} and \ref{thm:JC-topology2}).
			\begin{enumerate}
				\item\label{item:JC-fiber-elliptic} If $(j,h)=(-1,0)$, then the fiber $F^{-1}(\{(j,h)\})$ is the disjoint union of a $2$-dimensional $p$-adic analytic submanifold, which may be empty depending on the value of $p$, and an isolated point at $(0,0,-1,0,0)$.
				\item If $(j,h)=(1,0)$, then the fiber $F^{-1}(\{(j,h)\})$ has dimension $2$ and a singularity at $(0,0,1,0,0)$. (By this we mean that $F^{-1}(\{(1,0)\})$ minus the critical point is a $p$-adic analytic submanifold of dimension $2$, but as a whole it is not a manifold because it has a singularity at the critical point, as happens in the real case for the same point, as in Figure \ref{fig:real-JC-image}).
				\item If $(j,h)$ is a rank $1$ critical value, that is, $j=(1-3a^4)/2a^2$ and $h=(1-a^4)/2a$ for some $a\in\Qp$ such that $1-a^4$ is the sum of two squares, then the following statements hold.
				\begin{enumerate}
					\item If $-3a^4-1$ is a non-square modulo $p$, the fiber $F^{-1}(\{(j,h)\})$ is the disjoint union of a $2$-dimensional $p$-adic analytic submanifold and the $1$-dimensional $p$-adic analytic submanifold homeomorphic to $\mathrm{S}^1_p$ which consists exactly of the critical points whose image is $(j,h)$.
					\item If $-3a^4-1$ is a square modulo $p$, the fiber $F^{-1}(\{(j,h)\})$ has dimension $2$ and singularities at the critical points. (By this we mean, as in case (b), that $F^{-1}(\{(j,h)\})$ minus the set of critical points is a $p$-adic analytic submanifold of dimension $2$, but as a whole it has singularities at the critical points contained in $F^{-1}(\{(j,h)\})$. In the real case this only happens for the pinched torus in Figure \ref{fig:real-JC-image}.)
				\end{enumerate} 
				\item For the rest of values of $(j,h)\in F(\mathrm{S}_p^2\times(\Qp)^2)$, the fiber $F^{-1}(\{(j,h)\})$ is a $2$-dimensional $p$-adic analytic submanifold.
			\end{enumerate}
			\item Suppose that $p\equiv 1\mod 4$ (Theorem \ref{thm:JC-topology1}).
			\begin{enumerate}
				\item If $(j,h)=(\pm 1,0)$ is a rank $0$ critical value, the fiber $F^{-1}(\{(j,h)\})$ has dimension $2$ and a singularity at every point of $L_j$, where
				\[L_{-1}=\Big\{\left(\delta u,\delta\epsilon \ii u,-1,u,\epsilon \ii u\right)\Big| u\in\Qp,\epsilon=\pm 1,\delta=\pm 1\Big\},\]
				\[L_1=\Big\{\left(-\delta\epsilon \ii u,\delta u,1,u,\epsilon \ii u\right)\Big| u\in\Qp,\epsilon=\pm 1,\delta=\pm 1\Big\}.\]
				\item Otherwise, $F^{-1}(\{(j,h)\})$ has the same form as in part (1).
			\end{enumerate}
		\end{enumerate}
	\end{theorem}

\begin{remark}
	In Theorem \ref{thm:JC2}, part (\ref{item:JC-fiber-elliptic}), ``isolated'' is used in the topological sense, that is, there is a ball around this point that does not contain any other point in the fiber.
\end{remark}

We conclude this introduction discussing the local models of the $p$-adic Jaynes-Cummings model at the critical points. In this case the computations are analogous to the real case, and roughly speaking so are the conclusions, although some simplifications of the expressions below can be given in the real case, as we see later (Corollary \ref{cor:JC-simplified}).
\begin{proposition}[Normal forms at critical points of $p$-adic Jaynes-Cummings model]\label{prop:JC3}
		All critical points in part (\ref{item:JC-critical}) of Theorem \ref{thm:JC1} are non-degenerate and their local normal forms are given as follows.
		\begin{enumerate}
			\item At $(0,0,-1,0,0)$, there are local coordinates $(x,\xi,y,\eta)$ such that the $p$-adic symplectic form is given by $\omega=(\dd x\wedge\dd\xi+\dd y\wedge\dd\eta)/2$ and \[\tilde{F}(x,\xi,y,\eta)=\frac{1}{2}(x^2+\xi^2,y^2+\eta^2)+\ocal((x,\xi,y,\eta)^3).\] Here $\tilde{F}=B\circ(F-F(0,0,-1,0,0))$ with
			\[B=\begin{pmatrix}
			1 & 2 \\
			1 & -2
			\end{pmatrix}.\]
			We say that $(0,0,-1,0,0)$ is a point of ``elliptic-elliptic'' type (Proposition \ref{prop:nf-elliptic}).
			\item At $(0,0,1,0,0)$, there are local coordinates $(x,\xi,y,\eta)$ such that the $p$-adic symplectic form is given by $\omega=(\dd x\wedge\dd\xi+\dd y\wedge\dd\eta)/2$ and \[\tilde{F}(x,\xi,y,\eta)=(x\eta-y\xi,x\xi+y\eta)+\ocal((x,\xi,y,\eta)^3).\] Here $\tilde{F}=B\circ(F-F(0,0,1,0,0))$ with
			\[B=\begin{pmatrix}
			2 & 0 \\
			0 & 4
			\end{pmatrix}.\]
			We say that $(0,0,1,0,0)$ is a point of ``focus-focus'' type (Proposition \ref{prop:nf-focus}).
			\item Let $\alpha,\beta,\gamma:\Qp\to\Qp$ be the $p$-adic analytic functions given by $\alpha(a)=a^2(a^2+1)^2(3a^4+1)/2,\beta(a)=(3a^4+1)^2/2,\gamma(a)=a^2(1-a^4)(a^2+1)^2/2$. Then, at any rank $1$ point $(au,av,-a^2,u,v)$, there are local coordinates $(x,\xi,y,\eta)$ such that the $p$-adic symplectic form is given by $\omega=(1-a^4)(a^2+1)a^{-3}(\dd x\wedge\dd\xi+\dd y\wedge\dd\eta)$ and \[\tilde{F}(x,\xi,y,\eta)=(\eta+\ocal(\eta^2),\alpha(a)x^2+\beta(a)\xi^2+\gamma(a)\eta^2+\ocal((x,\xi,\eta)^3)).\] Here $\tilde{F}=B\circ(F-F(au,av,-a^2,u,v))$ with
			\[B=\begin{pmatrix}
			a & 0 \\
			\frac{a^6(3a^4+1)}{1-a^4} & \frac{-2a^5(3a^4+1)}{1-a^4}
			\end{pmatrix}.\]
			We say that any of the points $(au,av,-a^2,u,v)$ is a point of ``transversally elliptic'' type (Proposition \ref{prop:nf-rank1}). Recall that these points were defined in Theorem \ref{thm:JC1}(4).
		\end{enumerate}
	\end{proposition}
\begin{remark}
	In principle the local models in Proposition \ref{prop:JC3} do not depend on the value of $p$, but of course some further simplifications are possible for some values of $p$. However the ``type'' of the point (elliptic-elliptic, etc.) does not change.
\end{remark}

\begin{table}
	\begin{tabular}{p{.18\linewidth}|p{.18\linewidth}|p{.18\linewidth}|p{.18\linewidth}|p{.18\linewidth}|}
		& & \multicolumn{3}{l|}{$p$-adic} \\\cline{3-5}
		& Real & $p=2$ & $p\equiv 1\mod 4$ & $p\equiv 3\mod 4$ \\ \hline
		Image of Hamiltonians & green region in Figure \ref{fig:real-JC-image} & no easy description & all $(\Qp)^2$ & all $(\Qp)^2$ \\ \hline
		Fiber of regular value & dimension $2$ analytic manifold (isomorphic to a torus) & \multicolumn{3}{p{.58\linewidth}|}{dimension $2$ analytic manifold (not isomorphic to a torus)} \\ \hline
		Fiber of rank $1$ value ($-3a^4-1$ not square) & circle & \multicolumn{3}{p{.58\linewidth}|}{circle + dimension $2$ analytic manifold} \\ \hline
		Fiber of rank $1$ value ($-3a^4-1$ square) & never happens & never happens & \multicolumn{2}{p{.38\linewidth}|}{dimension $2$, singular at a circle} \\ \hline
		Fiber of $(-1,0)$ & point & point & dimension $2$, singular at four lines & point + dimension $2$ analytic manifold \\ \hline
		Fiber of $(1,0)$ & dimension $2$, singular at a point & dimension $2$, singular at a point & dimension $2$, singular at four lines & dimension $2$, singular at a point \\ \hline
	\end{tabular}
	\medskip
	\caption{Comparison of the real and $p$-adic Jaynes-Cummings models.}
	\label{table:JC}
\end{table}

Table \ref{table:JC} summarizes the results of Theorems \ref{thm:JC1} and \ref{thm:JC2}. Figure \ref{fig:padic-JC-image} gives an idea of what the critical points look like in the $p$-adic case. The richness which this simple example exhibits is an indication that a general $p$-adic theory of integrable systems will include many intricacies.

\begin{figure}
	\begin{tikzpicture}
		\node (im) at (4,4.5) {\includegraphics[width=10cm,trim=0 5cm 0 0]{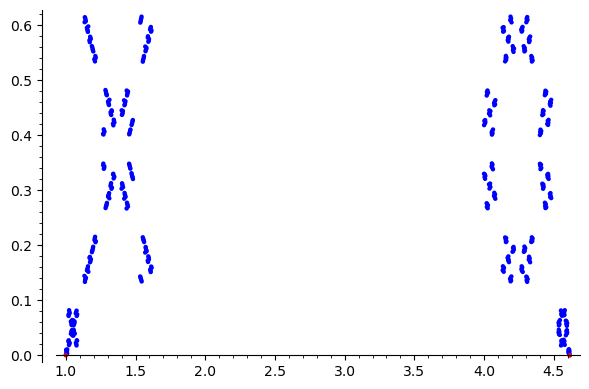}};
		\draw (0.1,0) -- (1,-3) node {\includegraphics[width=6cm]{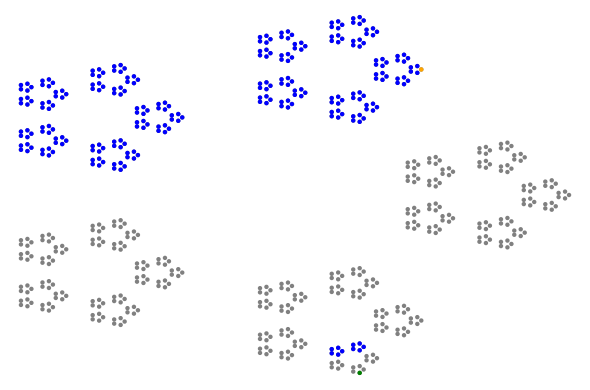}};
		\draw (8.6,0) -- (7.5,-3) node {\includegraphics[width=6cm]{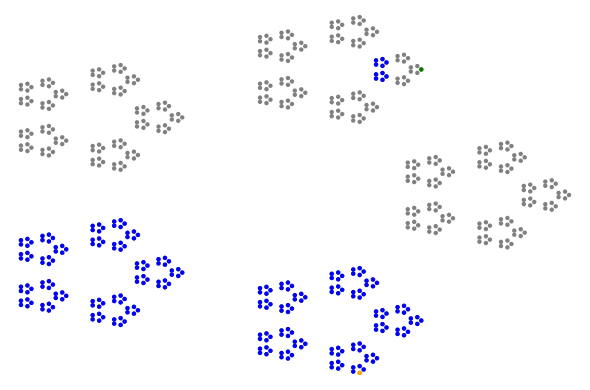}};
		\draw (1.1,4.4) -- (1,9) node {\includegraphics[width=6cm]{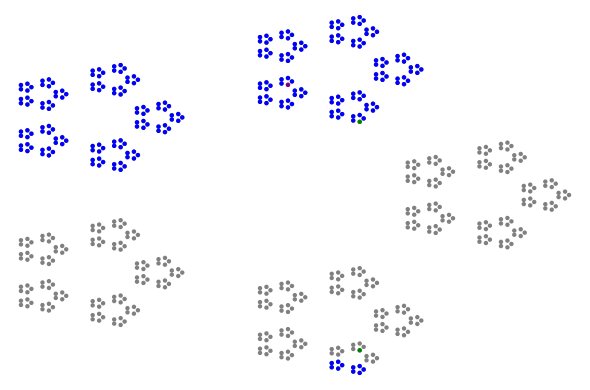}};
		\draw (6,1.3) -- (7.5,9) node {\includegraphics[width=6cm]{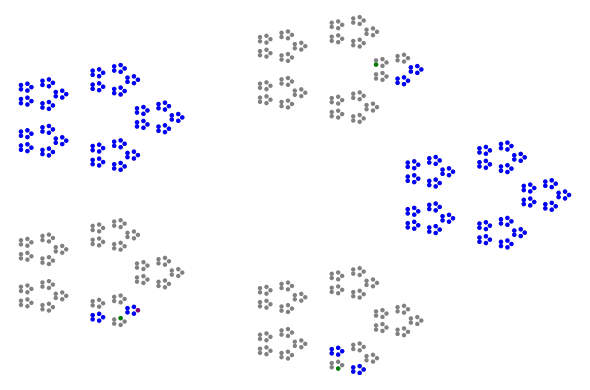}};
	\end{tikzpicture}
	\caption{Abstract representation of image and fibers of the $p$-adic Jaynes-Cummings model $F:\mathrm{S}_p^2\times(\Qp)^2\to(\Qp)^2$ for $p=5$. For other values of $p$ the results are not exactly the same, but there are many common aspects. As in Figure \ref{fig:real-JC-image}, the two red points correspond to rank $0$ critical points and the blue ones correspond to rank $1$ critical points. In the fibers, only the $z$ coordinate is represented; at the blue points the other coordinates form two ($p$-adic) circles, at the green points one circle, at the purple points they also have dimension $1$, and at the yellow points they form two $2$-planes. The grey points are values of $z$ that are not in the fiber.}
	\label{fig:padic-JC-image}
\end{figure}

In this article when we discuss images and fibers of functions, we will frequently graphically represent our findings. This is a problem because we are working with the $p$-adic field, and representations are usually done with real values. The solution is to ``translate'' the $p$-adic numbers into real ones.

We will use two types of representations, which we will call $1$-dimensional and $2$-dimensional. In the $1$-dimensional ones, only one variable is to be represented: suppose it is $x$. We choose a map \[R:\Zp\to\C\] from the values of $x$ to complex numbers and represent them as such (in the usual way where $\C$ is represented as $\R^2$). The mapping is chosen so that the distance between $R(x_1)$ and $R(x_2)$ represents the distance between $x_1$ and $x_2$ in the most accurate possible way. The solution we have found, which is not the only one possible, is as follows:

\begin{itemize}
	\item If $p=2$,
	\[R\left(\sum_{n=\ord(x)}^\infty x_n2^n\right)=\sum_{n=\ord(x)}^\infty x_n\left(\frac{3\mathrm{i}}{5}\right)^n.\]
	\item If $p=3$,
	\[R\left(\sum_{n=\ord(x)}^\infty x_n3^n\right)=\sum_{n=\ord(x)}^\infty \mathrm{e}^{\frac{2\pi \mathrm{i}x_n}{3}}\frac{1}{2^n}.\]
	\item If $p=5$,
	\[R\left(\sum_{n=\ord(x)}^\infty x_n5^n\right)=\sum_{n=\ord(x)}^\infty \mathrm{e}^{\frac{2\pi \mathrm{i}x_n}{5}}\left(\frac{3}{10}\right)^n.\]
\end{itemize}
Examples of these representations can be seen in Figures \ref{fig:fibers3}, \ref{fig:fibers1} and \ref{fig:fibers2}, for the different primes. For example, in Figure \ref{fig:fibers1} we can see a pentagon divided in five pentagons. Each one of these five pentagons represents a remainder modulo $5$ (hence, a $5$-adic ball of radius $1/5$). They are in turn subdivided in five pentagons, which represents balls of radius $1/25$, and so on.

The $2$-dimensional representations are recognized by their coordinate axes. Here two variables are being represented, say $x$ and $y$. A point $(x,y)\in\Zp^2$ is mapped to $\R^2$ by applying componentwise the correspondence
\[\sum_{i=\ord(x)}^\infty x_ip^i\mapsto\sum_{i=\ord(x)}^\infty x_ic^i\]
where $c=2/5$ for $p=2$, $c=2/9$ for $p=3$, and $c=2/15$ for $p=5$. For example, in Figure \ref{fig:exp} of the Appendix \ref{app:prelim}, we are representing a function from $5\Z_5$ to $1+5\Z_5$; $x$ is the independent variable and $y$ is the dependent one. The five clusters in which the points are divided correspond to $5$-adic balls of radius $1/25$ when projected to the $x$-axis. They can be seen, in turn, subdivided into five clusters, which correspond to $x$ being in a ball of radius $1/125$, and so on. The same clusters occur in the $y$-axis. The values in the axes do not represent $p$-adic numbers, but the real numbers resulting from the mapping.

All figures in this paper, including the $p$-adic representations, have been done using computer code developed in Sage.

\section{$p$-adic analytic integrable systems}\label{sec:integrable}

In this section we review the basic notions of $p$-adic symplectic geometry when the field $\mathbb{R}$ of coefficients
is replaced by the field of $p$\--adic numbers $\mathbb{Q}_p$. These extensions were proposed in an earlier paper by Pelayo, Voevodsky and Warren~\cite[Section 7]{PVW} in 2015.

\subsection{$p$\--adic analytic manifolds}
First we review some concepts for $p$-adic differential geometry, which can be found in the literature (see for example \cite{Schneider}), starting with the concept of a $p$-adic manifold. The following definitions are straightforward extensions of the real case. Following \cite[Sections 7-8]{Schneider}, given a Hausdorff topological space $M$ and an integer $n$, an \emph{$n$-dimensional $p$-adic analytic atlas} is a set of functions $A=\{\phi:U_\phi\to V_\phi\}$, where $U_\phi\subset M$ and $V_\phi\subset (\Qp)^n$ are open subsets, such that
	\begin{itemize}
		\item $\phi$ is a homeomorphism between $U_\phi$ and $V_\phi$;
		\item for any $\phi,\psi\in A$, the change of charts
		$\psi\circ \phi^{-1}:\phi(U_\phi\cap U_\psi)\to \psi(U_\phi\cap U_\psi)$
		is bi-analytic, i.e. it is analytic with analytic inverse.
	\end{itemize}
Such an $M$ together with such an atlas is called an \textit{$n$-dimensional $p$-adic analytic manifold}. A \emph{maximal atlas} for $M$ has a chart for each open set. The integer $n$ is called the \emph{dimension} of $M$.

Now let $M$ and $N$ be $p$-adic analytic manifolds of dimensions $m$ and $n$ respectively, a map $F:M\to N$ is \textit{analytic} if, for any $u\in M$, there are neighborhoods $U_\phi$ of $u$ and $U_\psi$ of $F(u)$ such that $\psi\circ F\circ\phi^{-1}$ is analytic (as a function from a subset of $(\Qp)^m$ to a subset of $(\Qp)^n$). $F$ is \textit{bi-analytic}, or an \textit{isomorphism of $p$-adic analytic manifolds}, if it is bijective and $F$ and $F^{-1}$ are analytic.

\begin{theorem}[{\cite[Proposition 8.6]{Schneider}}]\label{thm:paracompact}
	\letpprime. For a $p$-adic analytic manifold $M$ the following conditions are equivalent.
	\begin{enumerate}
		\item $M$ is paracompact (any open covering can be refined to a locally finite one).
		\item $M$ is strictly paracompact (any open covering can be refined to one consisting in pairwise disjoint sets).
		\item $M$ is an ultrametric space (its topology can be defined by a metric that satisfies the strict triangle inequality).
	\end{enumerate}
\end{theorem}

\begin{corollary}\label{cor:disjoint-union}
	\letpprime. Any paracompact $p$-adic analytic manifold is isomorphic to a disjoint union of $p$-adic analytic balls. Hence, a compact $p$-adic analytic manifold is isomorphic to a finite disjoint union of $p$-adic analytic balls.
\end{corollary}
\begin{proof}
	This is a consequence of Theorem \ref{thm:paracompact} and Corollary \ref{cor:disjoint}.
\end{proof}

Corollary \ref{cor:disjoint-union} implies that, when defining an atlas for a manifold, we can take the open sets in the atlas as disjoint, and the charts sending them to balls in $(\Qp)^n$.

The last part of Corollary \ref{cor:disjoint-union} was strengthened by Serre \cite{Serre}:
two finite disjoint unions of balls are isomorphic if and only if the corresponding numbers of balls differ by a multiple of $p-1$. That is, there are exactly $p-1$ compact $p$-adic manifolds, modulo isomorphism.

The notions of \emph{$p$-adic analytic function}, \emph{$p$-adic vector field} and \emph{$p$-adic differential form} are analogous to the ones in the real case. In order to make the paper as accessible as possible and make some elementary comparisons with the real case, we review these notions in Appendix \ref{app:vector}.

\subsection{$p$-adic integrable systems}

Let $p$ be a prime number. A \textit{$p$-adic analytic symplectic manifold} is a pair $(M,\omega)$ where $M$ is a $p$-adic analytic manifold and $\omega$ is a closed non-degenerate analytic $2$-form in $M$. For example,
if $S$ is a $p$-adic analytic manifold, then the canonical symplectic form on $M=\mathrm{T}^*S$ is also analytic by construction.

Given a $p$-adic analytic symplectic manifold $(M,\omega)$ and a $p$-adic analytic function $H:M\to\Qp$, there is a unique $p$-adic analytic vector field that satisfies
	\begin{equation}\label{eq:hamilton}
		\imath(X_H)\omega=\dd H.
	\end{equation}
As in the real case, $X_H$ is called the \textit{Hamiltonian vector field} associated to $H$.
We recall the proof of this fact, which is the same as in the real case. Let $q\in M$. We may assume that $\omega_q$ has the form
	\[\dd x_1\wedge\dd y_1+\ldots+\dd x_n\wedge\dd y_n\]
	in coordinates $(x_1,y_1,\ldots,x_n,y_n)$ near $q$, and
	$\dd H(q)=\sum_{i=1}^n\left(\frac{\partial H}{\partial x_i}(q)\dd x_i+\frac{\partial H}{\partial y_i}(q)\dd y_i\right)$. Hence
	$X_H(q)=\sum_{i=1}^n\left(\frac{\partial H}{\partial y_i}(q)\frac{\partial}{\partial x_i}-\frac{\partial H}{\partial x_i}(q)\frac{\partial}{\partial y_i}\right)$.

Also as in the real case, the \textit{Poisson bracket} $\{\cdot,\cdot\}$ of two $p$-adic analytic functions $f,g:M\to\Qp$ is defined by
	\[\{f,g\}=\omega(X_f,X_g).\]

\begin{definition}[{Pelayo-Voevodsky-Warren \cite[Definition 7.1]{PVW}}]\label{def:integrable}
	\letpprime\ and let $(M,\omega)$ be a $p$-adic analytic symplectic manifold. We say that a $p$-adic analytic map \[F:=(f_1,\ldots,f_n):(M,\omega)\to(\Qp)^n\] is a \emph{$p$-adic analytic integrable system} if two conditions hold:
	\begin{enumerate}
		\item The functions $f_1,\ldots,f_n$ satisfy $\{f_i,f_j\}=0$ for all $1\le i\le j\le n$;
		\item The set where the $n$ differential $1$-forms
		$\dd f_1,\ldots,\dd f_n$
		are linearly dependent has $p$-adic measure zero.
	\end{enumerate}
\end{definition}

In item (2) of Definition \ref{def:integrable} the $p$-adic measure is with respect to the $p$-adic volume form $\Omega=\omega^n$.

Throughout the paper, whenever we speak of analytic maps we always mean $p$-adic analytic maps, and similarly for manifolds.

\section{The $p$-adic analytic oscillator on $(\Qp)^2$}\label{sec:circ-oscillator}

This section starts the main part of the paper. Recall that our goal is to start developing the theory of $p$-adic integrable systems, so we will start with the simplest example, which is the oscillator, that is, the system \[f:(\Qp)^2\to\Qp\] with Hamiltonian \[f(x,y)=x^2+y^2\] on $(\Qp)^2$ endowed with the standard $p$-adic symplectic form $\dd x\wedge\dd y$. Since this section we believe is interesting in its own right, independently on the upcoming sections, we use $(x,y)$ for the variables on $(\Qp)^2$, instead of $(u,v)$, because it is more common.

As it is well known, in the real case the trajectory in the phase space of this system coincides with the fiber of $f$, which is a circle. \textit{In the $p$-adic case, we will see that the trajectory is part of the fiber of $f$.} Actually, $f$ is the momentum map of a $p$-adic circle action (see Definition \ref{def:momentum-map}). This will be useful later because the oscillator is used in the construction of the Jaynes-Cummings model.

As the fiber of $f$ is a circle, we will first find some structure in $\mathrm{S}^1_p$ and other circles in the $p$-adic plane $(\Qp)^2$, to which we dedicate Section \ref{sec:circles}. Then in Section \ref{sec:oscillator} we study the $p$-adic harmonic oscillator.

We refer to Appendix \ref{app:prelim} for the basic definitions concerning the $p$-adic numbers, which we use below.

\subsection{The structure of $p$-adic analytic circles $\mathrm{S}^1_p$}\label{sec:circles}

 We can understand a point $(a,b)$ in $\mathrm{S}^1_p$ as \textit{acting} on $(\Qp)^2$ (in the sense of Appendix \ref{app:actions}) as multiplication by the matrix
\[\begin{pmatrix}
	a & -b \\
	b & a
\end{pmatrix}.\]
These matrices are called \textit{unitary} in real symplectic algebra, due to its identification with complex numbers of absolute value $1$ that associates to this matrix the number $a+b\mathrm{i}$. In $p$-adic algebra, this identification does not make much sense because not all complex $p$-adics can be expressed as $a+b\mathrm{i}$, for $a,b\in\Qp$; actually some of them are transcendental over $\Qp$ (the field $\Cp$ is not defined as the algebraic closure of $\Qp$, but as the metric completion of that closure). This means that the identification will not be surjective.

Despite of this, we will still call these matrices \textit{unitary} because it is not clear what this term should mean in the complex $p$-adic context. (A unitary matrix in $\C$ is a matrix $A$ such that $\overline{A}^TA=I$. This definition uses the notion of complex conjugate, which has no canonical equivalent in $\Cp$ because the Galois group of the extension $\Cp/\Qp$ is infinite.)

Let $C_k=f^{-1}(\{k\})$, where $k\in\Qp$. In the real case, given two points $(x,y),(x',y')\in C_k$, for $k>0$, there is a unitary matrix that sends one point to the other:
\begin{equation}\label{eq:unitmatrix}
	\frac{1}{k}\begin{pmatrix}
		xx'+yy' & x'y-xy' \\
		xy'-x'y & xx'+yy'
	\end{pmatrix}.
\end{equation}
Also, all unitary matrices have the form
\begin{equation}
	\label{eq:rotation}
	\begin{pmatrix}
		\cos t & -\sin t \\
		\sin t & \cos t
	\end{pmatrix}.
\end{equation}
In terms of Lie groups, this is to say that the group of unitary matrices $\mathrm{S}^1$ is the same as the group of rotation matrices, which is essentially $\R/2\pi\Z$, because the domain of the functions is $\R$ and their values repeat with period $2\pi$.

Now we turn to the $p$-adic case. The group of rotation matrices can now be identified with $p^d\Zp$, where $d=2$ if $p=2$ and otherwise $d=1$, because this is the domain of the cosine and the sine, and the latter is injective as a function from $p^d\Zp$ to $p^d\Zp$ (see Appendix \ref{app:prelim}). We will see now that this group does not coincide with $\mathrm{S}^1_p$, but instead it is a proper subgroup. Equivalently, the set of points of the form $(\cos t,\sin t)$ for $t\in p^d\Zp$, does not give all the points in $\mathrm{S}^1_p$.

\begin{proposition}\label{prop:hamplane}
	Let $p$ be a prime number. Let $f:(\Qp)^2\to \Qp$ be given by $f(x,y)=x^2+y^2$. Let $C_k=f^{-1}(\{k\})$ and $C_k^*=C_k\setminus\{(0,0)\}$ (concretely, $C_k^*=C_k$ if $k\ne 0$). Then the following statements hold.
	\begin{enumerate}
		\renewcommand{\theenumi}{\roman{enumi}}
		\item Any two points in $C_k^*$ are related by the action of $\mathrm{S}^1_p$, except if $k=0$, in which case only proportional points are related.
		\item Two points $(x,y),(x',y')\in C_k^*$ are related by a rotation matrix if and only if
		\begin{equation}
		\left\{\begin{aligned}
		x&\equiv x'\mod p^{r+d} \label{eq:cond} \\
		y&\equiv y'\mod p^{r+d}
		\end{aligned}\right.
		\end{equation}
		where
		\[r=\min\{\ord(x),\ord(y),\ord(x'),\ord(y')\},\]
		and $d=2$ if $p=2$ and otherwise $d=1$.
	\end{enumerate}
\end{proposition}
This $r$, that is constant in an orbit of $p^d\Zp$, will be called the \textit{order} of the orbit.
\begin{proof}
	We start with the ``degenerate case'' $k=0$. If $p=2$ or $p\equiv 3\mod 4$, the only solution to $x^2+y^2=0$ is $(0,0)$ and $C_0^*=\varnothing$, so suppose $p\equiv 1\mod 4$.
	
	Now \[C_0=\{(x,\mathrm{i}x)\mid x\in\Qp\},\] where $\mathrm{i}$ is an element of $\Qp$ such that $\mathrm{i}^2=-1$. Let $(x,\mathrm{i}x)$ and $(x',\mathrm{i}'x')$ be two elements in $C_0$ with $x,x'\ne 0$. If there is a unitary matrix which sends one to the other, we have
	\[x'=(a-\mathrm{i}b)x\]
	and
	\[\mathrm{i}'x'=(b+\mathrm{i}a)x=\mathrm{i}(a-\mathrm{i}b)x=\mathrm{i}x',\]
	which implies $\mathrm{i}=\mathrm{i}'$. For the other direction, if $\mathrm{i}=\mathrm{i}'$, the matrix
	\[\frac{1}{2xx'}\begin{pmatrix}
		x^2+x'^2 & \mathrm{i}(x^2-x'^2) \\
		\mathrm{i}(x'^2-x^2) & x^2+x'^2
	\end{pmatrix}\]
	is unitary and sends $(x,\mathrm{i}x)$ to $(x',\mathrm{i}x')$. This shows part (i).
	
	In order to prove part (ii), we need the matrix to be a rotation matrix. This happens when
	\[x'=x\cos t-\mathrm{i}x\sin t=x\exp(-\mathrm{i}t).\]
	This needs that $x'/x\in 1+p\Zp$, which implies \[\ord(x'-x)\ge \ord(x)+1\] as we wanted.
	
	Now we turn to $k\ne 0$. The unitary matrix for part (i) is the same matrix (\ref{eq:unitmatrix}) as in the real case. For part (ii), this matrix is a rotation matrix if and only if
	\begin{equation}
	\left\{\begin{aligned}
	xx'+yy'&=k\cos t; \label{eq:cos-sin} \\
	xy'-x'y&=k\sin t.
	\end{aligned}\right.
	\end{equation}
	for some $t$.
	
	By changing $(x,y,x',y',k)$ to \[(p^{-r}x,p^{-r}y,p^{-r}x',p^{-r}y',p^{-2r}k),\] we may assume that $r=0$. Now, the first four numbers are in $\Zp$, but not all of them in $p\Zp$. Without loss of generality, suppose $x\notin p\Zp$. Let $s=\ord(k)$. The conditions \eqref{eq:cond} are now written as
	\begin{align}
		x&\equiv x' \mod p^d \label{eq:xdif} \\
		y&\equiv y' \mod p^d \label{eq:ydif}
	\end{align}
	
	Suppose first that we have a solution of \eqref{eq:cos-sin}. Proposition \ref{prop:functions} gives
	\[
	\left\{\begin{aligned}
	xx'+yy'&\equiv k\mod p^{s+d}; \\
	xy'-x'y&\equiv 0\mod p^{s+d}.
	\end{aligned}\right.
	\]
	
	Now make the change $y=ux$, where $u\in\Zp$ (because $\ord(y)\ge 0=\ord(x)$):
	\begin{align}
		x(x'+uy')&\equiv k\mod p^{s+d} \label{eq:cos3} \\
		x(y'-ux')&\equiv 0\mod p^{s+d} \label{eq:sin3}
	\end{align}
	
	As $\ord(x)=0$, \eqref{eq:sin3} solves as $y'\equiv ux'\mod p^{s+d}$. Substituting in \eqref{eq:cos3}, we get
	\[xx'(1+u^2)\equiv k\mod p^{s+d}.\]
	But we also know that
	\[x^2(1+u^2)=x^2+y^2=k\Longrightarrow xx'(1+u^2)=kx^{-1}x'\]
	and together with the previous equation
	\[kx^{-1}x'\equiv k\mod p^{s+d}\Longrightarrow x^{-1}x'\equiv 1\mod p^d \Longrightarrow x\equiv x'\mod p^d\]
	and
	\[y=ux\equiv ux'\equiv y'\mod p^d\]
	as we wanted.
	
	Conversely, suppose that \eqref{eq:xdif} and \eqref{eq:ydif} hold. We know that $x\notin p\Zp$, and by \eqref{eq:xdif}, also $x'\notin p\Zp$. Let $u=x^{-1}y\in\Zp$ and $u'=x'^{-1}y'\in\Zp$. We have that
	\begin{align*}
	\ord(1+u^2) & =\ord(x^2(1+u^2)) \\
	& =\ord(x^2+y^2) \\
	& =\ord(k)=s,
	\end{align*}
	and the same for $u'$. Using \eqref{eq:xdif} we get
	\begin{align*}
	x^2(1+u^2) & =k=x'^2(1+u'^2) \\
	& \equiv x^2(1+u'^2)\mod p^{s+2d-1},
	\end{align*}
	which implies $u^2\equiv u'^2\mod p^{s+2d-1}$, that is
	\[u\equiv\pm u'\mod p^{s+d}\]
	(the exponent of $p$ is always the same unless $p=2$, in which case it goes one up when squaring and one down when canceling the squares, as in Corollary \ref{cor:hensel}). We claim that the plus sign holds.
	
	By \eqref{eq:ydif},
	\[ux=y\equiv y'\equiv u'x\mod p^d,\]
	so we must have $u\equiv u'\mod p^d$. If $s=0$, the claim is proved. Otherwise, $1+u^2$ is a multiple of $p$, which implies $u$ is not and $u'\equiv u\not\equiv -u\mod p^d$. So the minus sign cannot hold, and the claim is proved.
	
	Now
	\[xy'-x'y=xu'x'-x'ux\equiv 0\mod p^{s+d}.\]
	It follows from Proposition \ref{prop:functions} that $(xy'-x'y)/k$ is in the image of the sine series, so there is $t$ such that $xy'-x'y=k\sin t$. This implies
	\begin{align*}
	(xx'+yy')^2 & =(x^2+y^2)(x'^2+y'^2)-(xy'-x'y)^2 \\
	& =k^2-k^2\sin^2 t
	& =k^2\cos^2 t
	\end{align*}
	and
	\[xx'+yy'=\pm k\cos t.\]
	Moreover,
	\begin{align*}
	xx'+yy' & \equiv xx'(1+u^2) \\
	& \equiv x^2(1+u^2)=k\mod p^{s+d};
	\end{align*}
	\[\pm k\cos t\equiv \pm k\mod p^{s+d},\]
	so the plus sign must hold, and $(x,y)$ and $(x',y')$ are related by a rotation.
\end{proof}

\begin{figure}
	\includegraphics[scale=0.5]{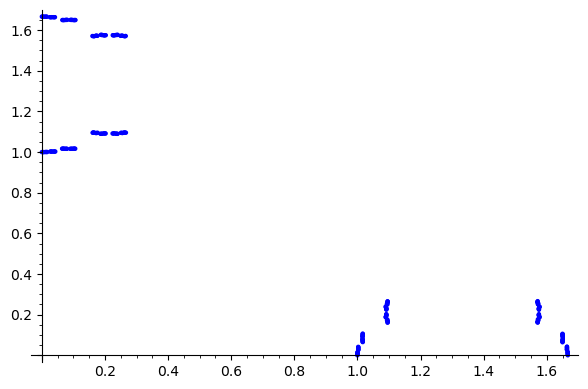}
	\includegraphics[scale=0.5]{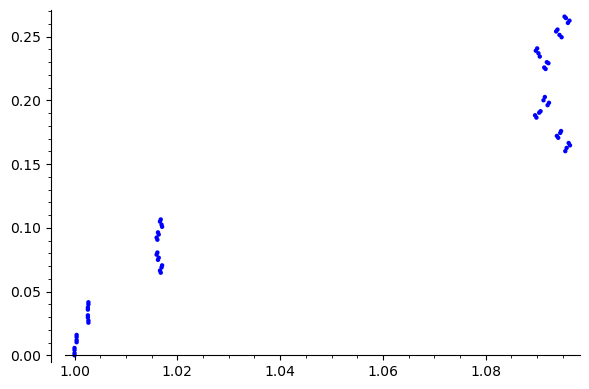}
	\caption{Left: The circle $\mathrm{S}^1_2$, consisting on four ``sectors''. The points in each sector are related by rotation. Right: a close-up on one sector.}
	\label{fig:S1-2}
\end{figure}

\begin{proposition}\label{prop:numorbits}
	\letpprime. Let $f:(\Qp)^2\to \Qp$ be given by $f(x,y)=x^2+y^2$. Let $k\in\Qp$ and let $C_k=f^{-1}(\{k\})$ and $C_k^*=C_k\setminus\{(0,0)\}$ (concretely, $C_k^*=C_k$ if $k\ne 0$). Given $r\in\Z$, the number $\orm(r,k)$ of order $r$ orbits of the rotation group $p^d\Zp$ in $C_k^*$ is given as follows.
	\begin{itemize}
		\item If $p\equiv 1\mod 4$:
		\[\orm(r,k)=\begin{cases}
		2p-2 & \text{if }\ord(k)>2r. \\
		p-1 & \text{if }\ord(k)=2r. \\
		0  & \text{if }\ord(k)<2r.
		\end{cases}\]
		\item If $p\equiv 3\mod 4$:
		\[\orm(r,k)=\begin{cases}
		p+1 & \text{if }\ord(k)=2r. \\
		0  & \text{otherwise.}
		\end{cases}\]
		\item If $p=2$:
		\[\orm(r,k)=\begin{cases}
		4 & \text{if }\ord(k)=2r\text{ and }\ord(k-2^{2r})\ge 2r+2. \\
		4 & \text{if }\ord(k)=2r+1\text{ and }\ord(k-2^{2r})\ge 2r+3. \\
		0 & \text{otherwise.}
		\end{cases}\]
	\end{itemize}
\end{proposition}
\begin{proof}
	As before, we start with the case $k=0$ (and $\ord(k)=\infty$). If $p=2$ or $p\equiv 3\mod 4$, $C_0^*$ is empty and the number of orbits is zero. Otherwise, an orbit is
	\[\{(x',\mathrm{i}x')\in\Qp\mid \ord(x'-x)\ge \ord(x)+1\}\]
	for fixed $x$ and $\mathrm{i}$ with $\mathrm{i}^2=-1$, with the order determined by the order of $x$. Two choices of $x$ give different orbits if and only if their leading digits differ, which gives $2p-2$ orbits ($2$ values of $i$ times $p-1$ leading digits of $x$).
	
	Now we turn to $k\ne 0$. Without loss of generality, we take $r=0$ (otherwise change $k$ to $p^{-2r}k$).
	
	We start with $p>2$. Let $\overline{C}_k$ be the set of orbits and $S_k$ be the set of pairs $(a,b)\in(\F_p)^2\setminus\{(0,0)\}$ such that $a^2+b^2=k$, where $\F_p$ is the finite field of order $p$. Of course, $\overline{C}_k$ and $S_k$ are empty if $\ord(k)<0$.
	
	By Proposition \ref{prop:hamplane}, there is an injective correspondence
	\[g:\overline{C}_k\to S_k\] that assigns to the orbit of $(x,y)$ the pair $(x\mod p,y\mod p)$. For the case $p>2$, $g$ is surjective, because if we have $a,b\in\F_p$ with $a^2+b^2=k$ we can lift them to $\Zp$ (apply Theorem \ref{thm:hensel}).
	
	Suppose that $\ord(k)>0$. $(a,b)$ must satisfy $a^2+b^2=0$, and $a$ and $b$ are not both $0$. This means $(ab^{-1})^2=-1$, which has no solution if $p\equiv 3\mod 4$ and leaves two possible values for $ab^{-1}$ otherwise. For each of these solutions, $a$ and $b$ can take $p-1$ possible values, and we are done.
	
	Now let us assume that $\ord(k)=0$. This means that $k\ne 0\in\F_p$, and $(a,b)$ must satisfy $a^2+b^2=k$.
	\begin{itemize}
		\item If $p\equiv 1\mod 4$, there is $u\in\F_p$ such that $u^2=-1$ and
	\begin{equation}
		\label{eq:a+ub}
		k=a^2+b^2=(a+ub)(a-ub)
	\end{equation}
	Substituting $c=a-ub$ in \eqref{eq:a+ub},
	\[k=c(c+2ub)\]
	which, for a fixed value of $c\ne 0$, has one solution for $b$, and in turn one solution for $a=c+ub$. Moreover, the same $(a,b)$ cannot be obtained for two values of $c$, hence there are as many solutions as values for $c$, which are $p-1$.
	
	\item If $p\equiv 3\mod 4$, $x^2+1$ is irreducible modulo $p$, so it has a root $u$ in $\F_{p^2}$ such that $\{1,u\}$ generates $\F_{p^2}$ as a $\F_p$-vector space. Defining as usual the conjugate of $x+uy\in\F_{p^2}$ as $x-uy$ for $x,y\in\F_p$, we can write
	\[k=a^2+b^2=(a+ub)(a-ub)=c\overline{c}\]
	and the problem reduces to count the number of $c\in\F_{p^2}$ such that $c\overline{c}=k$. Let $R_k$ be the set of such $c$. We have that:
	\begin{itemize}
		\item $\F_{p^2}\setminus\{0\}=R_1\cup R_2\cup\ldots\cup R_{p-1}$;
		\item If $x\in R_i$ and $y\in R_j$,
		\[xy\cdot\overline{xy}=x\overline{x}\cdot y\overline{y}=ij\]
		and $xy\in R_{ij}$;
		\item $R_i$ is not empty: 
		if $i$ is a square modulo $p$, this is obvious. Otherwise, let $a$ be the smallest non-square modulo $p$ and $i=ab^2$. By our choice of $a$, $a-1=c^2$ for some $c$, and
		\[i=ab^2=(c^2+1)b^2=(bc)^2+b^2,\]
		which implies $bc+rb\in R_i$.
	\end{itemize}
	This together implies that, for $s\in R_{i^{-1}j}$, $x\mapsto sx$ is a bijection between $R_i$ and $R_j$, which implies that all the $R_i$ have the same size and
	\[|R_i|=\frac{p^2-1}{p-1}=p+1\]
	as we wanted.
	
	\item If $p=2$, we suppose without loss of generality that $x$ is odd. Consider first the case where $y$ is even. In this case $k\equiv 1\mod 4$, so the other case ($k\equiv 3\mod 4$) has no solution. We need to choose $x\mod 4$ and $y\mod 4$. $x^2$ will always be $1\mod 8$, and $y^2$ can be $0$ or $4$. Hence $k\mod 8$ (which is $1$ or $5$) determines $y\mod 4$, and $x\mod 4$ can be chosen freely between $1$ and $3$. Once we have $x\mod 4$ and $y\mod 4$ so that $x^2+y^2\equiv k\mod 8$, we can fix $y$ and lift $x$ using Corollary \ref{cor:hensel}. This leaves two orbits, and the other two come from swapping $x$ and $y$.
	
	In the other case, $x$ and $y$ are both odd and $x^2+y^2\equiv 2\mod 8$. This means that the cases $k\equiv 6\mod 8$ and $\ord(k)>1$ have no solution. We can choose four possibilities for $x$ and $y$ modulo $4$. Again, Corollary \ref{cor:hensel} allows us to lift this to four orbits for $x$ and $y$.
\end{itemize}
\end{proof}

\begin{corollary}\label{cor:orbits}
	Let $f:(\Qp)^2\to \Qp$ be given by $f(x,y)=x^2+y^2$. Let $k$ in $\Qp$ and let $C_k=f^{-1}(\{k\})$.
	\begin{itemize}
		\item If $p\equiv 1\mod 4$ and $k\ne 0$, $C_k$ consists of $2p-2$ orbits for the rotation group with each order $r<\ord(k)/2$, together with $p-1$ orbits with order $\ord(k)/2$ if $\ord(k)$ is even. The set $C_0$ consists of two lines, that form $2p-2$ orbits with each integer order and one with their intersection point $(0,0)$.
		\item If $p\equiv 3\mod 4$ and $k\ne 0$, $C_k$ is empty if $\ord(k)$ is odd, and otherwise consists of $p+1$ orbits, all with order $\ord(k)/2$. The set $C_0$ is a single point.
		\item If $p=2$ and $k\ne 0$, $C_k$ consists of $4$ orbits with order $\lfloor \ord(k)/2\rfloor$, if it is not empty. The set $C_0$ is again a single point.
	\end{itemize}
	In any case, each orbit is homeomorphic to $p\Zp$ by definition. See Table \ref{table:oscillator} for a comparison to the real case, and Figures \ref{fig:S1-2}, \ref{fig:S1-3} and \ref{fig:S1-5} for representations of $S^1_p$ for $p=2,3,5$.
\end{corollary}

\begin{table}
	\begin{tabular}{p{.15\linewidth}|p{.18\linewidth}|p{.19\linewidth}|p{.2\linewidth}|p{.18\linewidth}|}
		& & \multicolumn{3}{l|}{$p$-adic} \\\cline{3-5}
		& Real & $p=2$ & $p\equiv 1\mod 4$ & $p\equiv 3\mod 4$ \\ \hline
		Uniqueness of flow & Unique & \multicolumn{3}{l|}{Not unique, but any two solutions coincide near $0$} \\ \hline
		Image of Hamiltonian & $[0,\infty)$ & ending in $01$ & all $\Qp$ & even order \\ \hline
		Fiber of nonzero & circle (1 sector) & circle (4 sectors) & circle ($\infty$ sectors) & circle ($p+1$ sectors) \\ \hline
		Fiber of $0$ & point & point & two lines & point \\ \hline
	\end{tabular}
	\medskip
	\caption{Comparison of the real and $p$-adic oscillators. ``Flow'' here refers to the flow of the Hamiltonian $x^2+y^2$, and ``image'' and ``fiber'' to those of this Hamiltonian.}
	\label{table:oscillator}
\end{table}

\begin{figure}
	\includegraphics[scale=0.5]{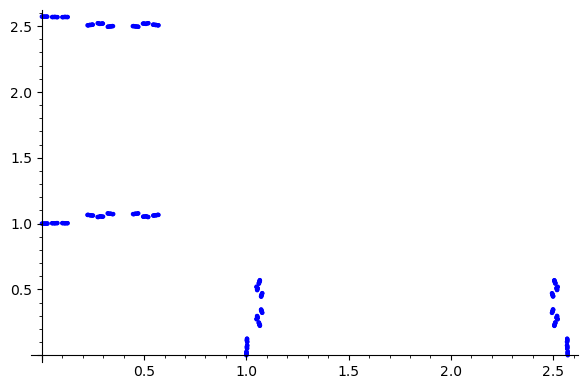}
	\includegraphics[scale=0.5]{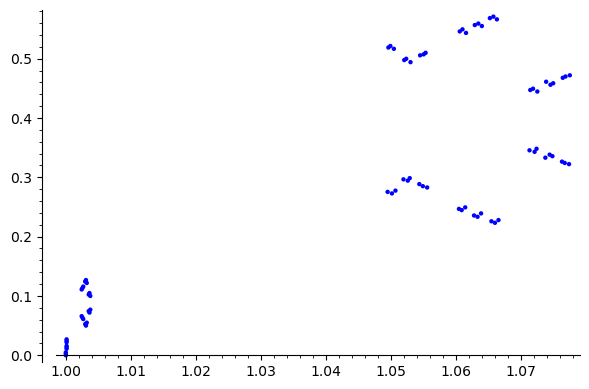}
	\caption{Left: The circle $\mathrm{S}^1_3$, consisting on four ``sectors''. The points in each sector are related by rotation. Right: a close-up on one sector.}
	\label{fig:S1-3}
\end{figure}

In view of the above we can now compute the image of $f$:
\begin{corollary}\label{cor:image}
	Let $f:(\Qp)^2\to \Qp$ be given by $f(x,y)=x^2+y^2$.
	\begin{itemize}
		\item If $p\equiv 1\mod 4$, $f$ is surjective.
		\item If $p\equiv 3\mod 4$, the image of $f$ consists of all even-order $p$-adics and zero.
		\item If $p=2$, the image of $f$ consists of the $p$-adics $x$ of order $r$ such that \[\ord(x-2^r)\ge r+2,\] and zero.
	\end{itemize}
\end{corollary}

The relations between points in the same orbit of the rotation group can be used to deduce the relation between $\mathrm{S}^1_p$ and $p^d\Zp$:
\begin{corollary}\label{cor:groups} \letpprime. The following statements hold.
	\begin{itemize}
		\item If $p\equiv 1\mod 4$, $\mathrm{S}^1_p/p\Zp$ is isomorphic to $\Z\times\F_p^*$.
		\item If $p\equiv 3\mod 4$, the quotient group $\mathrm{S}^1_p/p\Zp$ is isomorphic to $\F_{p^2}^*/\F_p^*$.
		\item If $p=2$, $\mathrm{S}^1_2/4\Z_2$ is isomorphic to $\Z/4\Z$.
	\end{itemize}
\end{corollary}

\begin{proof}
	The last two parts are direct consequences of the proof of Proposition \ref{prop:numorbits}. In the second the quotient is given by the pairs $(a,b)\in \F_p^2$ such that $a^2+b^2=1$, or equivalently the unitary matrices over $\F_p$. In the third it is given by
	\[\left\{\begin{pmatrix}
	1 & 0 \\
	0 & 1
	\end{pmatrix},
	\begin{pmatrix}
	3 & 0 \\
	0 & 3
	\end{pmatrix},
	\begin{pmatrix}
	0 & 3 \\
	1 & 0
	\end{pmatrix},
	\begin{pmatrix}
	0 & 1 \\
	3 & 0
	\end{pmatrix}\right\}\]
	where the matrices are taken with entries in $\Z/4\Z$ (note that, in this case, it is a proper subgroup of the unitary matrices modulo $4$).
	
	For the first part, consider $\ii\in\Qp$ such that $\ii^2=-1$ and $\phi:\mathrm{S}^1_p\to\Qp^*$ defined as $\phi(a,b)=a+\ii b$. We see that $\phi$ is a group morphism because
	\[\begin{pmatrix}
	a & -b \\
	b & a
	\end{pmatrix}
	\begin{pmatrix}
	a' & -b' \\
	b' & a'
	\end{pmatrix}=
	\begin{pmatrix}
	aa'-bb' & -ab'-a'b \\
	ab'+a'b & aa'-bb'
	\end{pmatrix}\]
	and
	\[aa'-bb'+\ii(ab'+a'b)=(a+\ii b)(a'+\ii b').\]
	
	Actually, it is an isomorphism, because
	\[(a+\ii b)(a-\ii b)=1\]
	so $a+\ii b$ determines its inverse $a-\ii b$, and they determine uniquely $a$ and $b$. The image by $\phi$ of $t\in p\Zp$ is \[\phi(\cos t,\sin t)=\exp(\ii t)\in 1+p\Zp.\] By Proposition \ref{prop:functions}, this implies $\phi(p\Zp)=1+p\Zp$ and
	\[\mathrm{S}^1_p/p\Zp=\Qp^*/(1+p\Zp)\]
	The class of a number in this quotient is given by its order in $\Z$, which is additive by multiplication, and its leading digit in $\F_p^*$, which is multiplicative. Hence, the quotient group is $\Z\times\F_p^*$.
\end{proof}

In any case, the three groups have the same Lie algebra, $\Qp$.

\begin{figure}
	\includegraphics[width=0.32\linewidth]{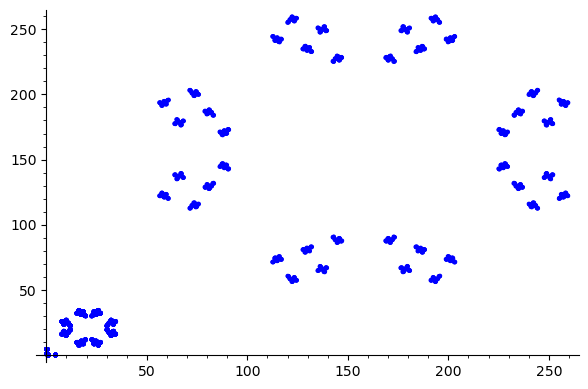}
	\includegraphics[width=0.32\linewidth]{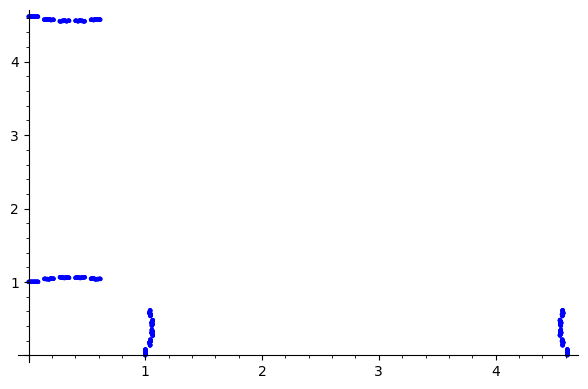}
	\includegraphics[width=0.32\linewidth]{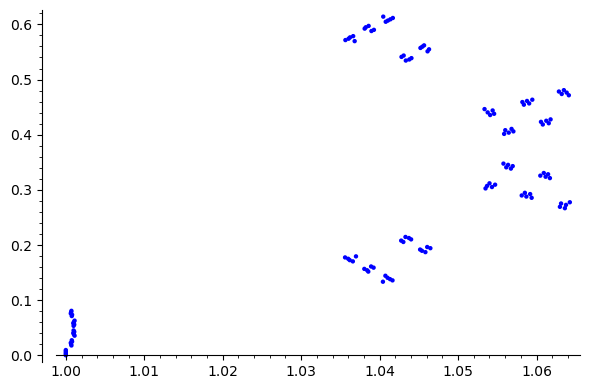}
	\caption{Left: Part of the circle $\mathrm{S}^1_5$, consisting on infinitely many ``sectors''. The points in each sector are related by rotation. The circle extends indefinitely, repeating the same pattern: we are seeing here four sectors of order $0$ (bottom left), eight of order $-1$ (small ones near them) and eight of order $-2$ (big ones). Middle: a close-up in the four order $0$ sectors. Right: a close-up on one sector.}
	\label{fig:S1-5}
\end{figure}

\subsection{Formulas and results for the $p$-adic analytic oscillator on $(\Qp)^2$}\label{sec:oscillator}

Consider the classical Hamiltonian $f(x,y)=x^2+y^2$ on the plane $\R^2$ with the standard symplectic form $\dd x\wedge\dd y$. By Hamilton's equations
\[\imath(X_f)\omega=\dd f\]
we have
\[\imath(X_f)(\dd x\wedge \dd y)=2x\dd x+2y\dd y\]
and taking $X_f=f_1\partial/\partial x+f_2\partial/\partial y$ we get $f_1\dd y-f_2\dd x=2x\dd x+2y\dd y$, that is, $f_1=2y$ and $f_2=-2x$.

To find the flow, we need to solve the differential equation
\[\left\{
\begin{aligned}
\frac{\partial\psi_t(x_0,y_0)}{\partial t}&=X_f(\psi_t(x_0,y_0));\\ \psi_0(x_0,y_0)&=(x_0,y_0).
\end{aligned}\right.
\]
that is, taking $\psi_t(x_0,y_0)=(x(t),y(t))$,
\begin{equation}\label{eq:diff}
\left\{
\begin{aligned}
\frac{\partial x(t)}{\partial t}&=2y(t);\\
\frac{\partial y(t)}{\partial t}&=-2x(t);\\
x(0)&=x_0;\\
y(0)&=y_0.
\end{aligned}\right.
\end{equation}
As we know, the solution to this problem in the real case is
\begin{equation}
	\label{eq:solution}
	\begin{pmatrix}
		x(t) \\ y(t)
	\end{pmatrix}
	=\begin{pmatrix}
		\cos 2t & \sin 2t \\
		-\sin 2t & \cos 2t
	\end{pmatrix}
	\begin{pmatrix}
		x_0 \\ y_0
	\end{pmatrix}.
\end{equation}

In the $p$-adic case, the equations \eqref{eq:solution} have infinitely many solutions and they do not even coincide near the origin. However, if we restrict to analytic functions, then there are still infinitely many solutions, but any two of them coincide near the origin (Proposition \ref{prop:initial}). So we can look for an analytic solution for this problem which is given as a power series around the initial point $t=0$:
\[x(t)=\sum_{i=0}^{\infty}a_it^i,\quad y(t)=\sum_{i=0}^{\infty}b_it^i,\]
and the equations \eqref{eq:diff} become
\[(i+1)a_{i+1}=2b_i,\quad (i+1)b_{i+1}=-2a_i,\quad a_0=x_0,\quad b_0=y_0.\]
Solving the recurrence
\[\left\{
\begin{aligned}
a_{2i}&=\frac{(-1)^i2^{2i}}{(2i)!}x_0, & a_{2i+1}&=\frac{(-1)^i2^{2i+1}}{(2i+1)!}y_0, \\
b_{2i}&=\frac{(-1)^i2^{2i}}{(2i)!}y_0, & b_{2i+1}&=\frac{(-1)^{i+1}2^{2i+1}}{(2i+1)!}x_0,
\end{aligned}
\right.\]
and substituting in the expressions for $x(t)$ and $y(t)$, we obtain the same solution \eqref{eq:solution} as in the real case. Despite this similarity, there are important differences:
\begin{itemize}
	\item In the $p$-adic case, by Proposition \ref{prop:initial}, the solution of the initial value problem is not unique, so the matrix \eqref{eq:rotation} is not the unique flow. Any other flow, however, will coincide near $0$.
	\item The image of the Hamiltonian in the real case is $[0,\infty)$. In the $p$-adic case, the image is given by Corollary \ref{cor:image}.
	\item In both cases, the fiber of each point different from $0$ is a ``circle''. But the structure of a circle is much more complicated in the $p$-adic case, as seen in Section \ref{sec:circles}.
	\item The fiber of $0$ is a point in the real case. In the $p$-adic case, by Corollary \ref{cor:orbits}, the same happens for some values of $p$ ($p=2$ and $p\equiv 3\mod 4$) but not for the $p\equiv 1\mod 4$: in this case the fiber consists of two lines.
\end{itemize}
See Table \ref{table:oscillator} for a summary of these differences.

\section{The $p$\--adic analytic spin system on $\mathrm{S}^2_p$}\label{sec:spin}

Once we have studied the oscillator, the next step is to study the spin system, which is the other system involved in the construction of the Jaynes-Cummings model. This will be relatively easy now because some results previously known for the oscillator can be used again.

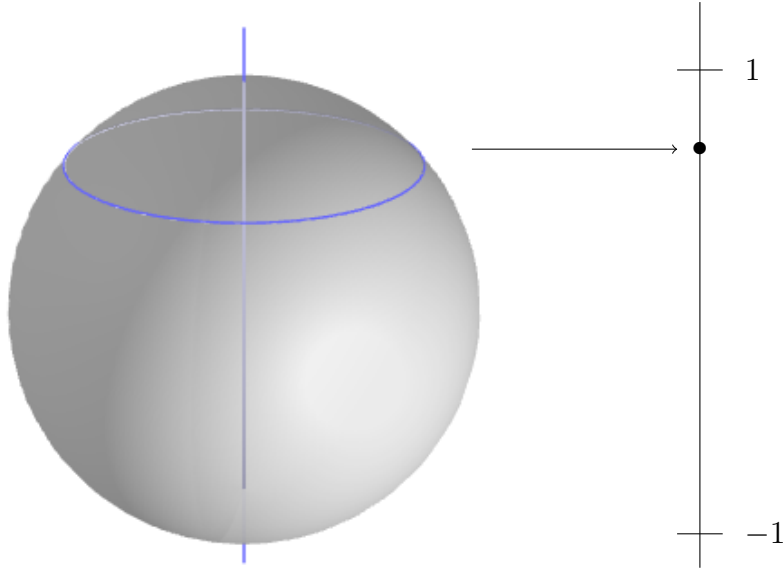
\begin{figure}
	\begin{tikzpicture}[scale=1.5]
		\node (esfera) at (0,0) {\includegraphics[trim=5cm 1.5cm 5cm 1.5cm,height=7.5cm]{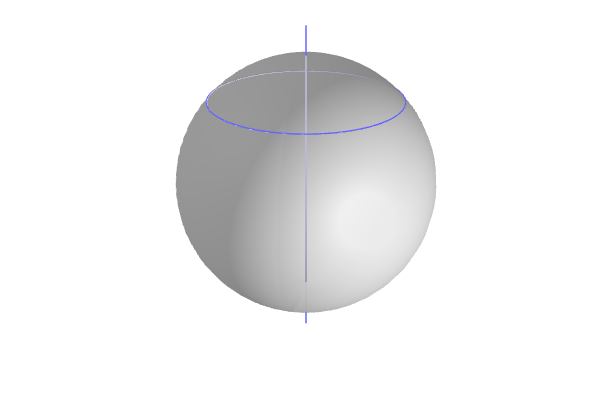}};
		\draw[->] (2,1.7) -- (3.8,1.7); \node at (4,1.7) {$\bullet$};
		\draw (4,-2) -- (4,3);
		\draw (3.8,2.4) -- (4.2,2.4) node[label=0:$1$] {};
		\draw (3.8,-1.7) -- (4.2,-1.7) node[label=0:$-1$] {};
	\end{tikzpicture}
	\caption{Image and a fiber of the real spin system}
	\label{fig:spin-real}
\end{figure}

In the real case, this system is simple: the momentum map is defined on the sphere $\mathrm{S}^2$ and it sends a point to its $z$ coordinate (see Figure \ref{fig:spin-real}). This gives the rotational action of $\mathrm{S}^1$ on $\mathrm{S}^2$. The image of the system is the interval $[-1,1]$, and the fiber of each point is a circle except for $-1$ and $1$ themselves. Now we turn to the $p$-adic case: \textit{the fibers are still circles, but the image is more complicated.}

We can give $\mathrm{S}^2_p$ (and in general $\mathrm{S}^n_p$) a structure of a $p$-adic analytic manifold: for fixed values of $x$ and $y$, the coordinate $z$ takes two possible values if $1-x^2-y^2$ is a square different from zero, one if it is zero, and no possible values otherwise, so we define
\[\phi_{\pm}:\{(x,y)\mid 1-x^2-y^2\text{ is a square}\ne 0\}\to\Qp\]
\[(x,y)\mapsto(x,y,\pm\sqrt{1-x^2-y^2})\]
and similarly one defines additional charts with the same formula but changing the order of the coordinates. From the formula, we deduce that
\[\dd z=\frac{\partial z}{\partial x}\dd x+\frac{\partial z}{\partial y}\dd y=-\frac{x\dd x+y\dd y}{z}.\]
We can also define the $1$-form $\dd\theta$ by
\[\dd \theta=\frac{y\dd x-x\dd y}{x^2+y^2}.\]

Now we consider on $\mathrm{S}^2_p$ the symplectic form given by
\[\omega=\dd\theta\wedge\dd z=-\frac{1}{z}\dd x\wedge\dd y\]
and the actions of the groups $\mathrm{S}^1_p$ (as the unitary matrices) and $p^d\Zp$ (as the rotation matrices) on the coordinates $(x,y)$ (see Appendix \ref{app:actions} for the definition of Hamiltonian actions). Substituting the expression of the induced vector field we get
\[\imath\left(y\frac{\partial}{\partial x}-x\frac{\partial}{\partial y}\right)\omega=-\frac{x\dd x+y\dd y}{z}=\dd z,\]
which implies, by $\omega=\dd\theta\wedge\dd z$, that this vector field is exactly $\partial/\partial \theta$. So $\theta$ represents the rotation angle around the $z$ axis in the clockwise direction. (Normally this angle is taken in counter-clockwise direction, but here we take it the other way to achieve consistency with the case of the oscillator.)

Of course, $\dd\theta$ and $\dd z$, and hence $\omega$, are invariant by these actions. So the actions are symplectic, as in the real case. The induced vector field is in the direction of $\theta$, and this is a Hamiltonian vector field with Hamiltonian function $f(x,y,z)=z$
This makes the actions Hamiltonian, because in this case $f(x,y,z)=\mu(x,y,z)$, that is,
\begin{equation}\label{eq:muspin}
	\mu(x,y,z)=z.
\end{equation}
A flow for this vector field can be calculated as in the case of the oscillator, resulting in the same solution \eqref{eq:solution}.

Now we characterize the fibers and image of $\mu$. As with our previous example, the results will depend on the value of $p$. So we start with $p\equiv 3\mod 4$.

\begin{proposition}
	\letpprime\ such that $p\equiv 3\mod 4$. Let $\mu:\mathrm{S}_p^2\to\Qp$ be the momentum map of the $p$-adic spin system given by \eqref{eq:muspin}. Given $z\in\Qp$, $\mu^{-1}(z)$ is equal to:
	\begin{enumerate}
		\item $\{(0,0,z)\}$ if $z=\pm 1$;
		\item empty if $z$ equals $t+1$ or $t-1$, and $t$ has odd positive order;
		\item a circle (that is, a set homeomorphic to $\mathrm{S}^1_p$) otherwise.
	\end{enumerate}
\end{proposition}

\begin{proof}
	First suppose that $\ord(z)\ne 0$. Then $\ord(1-z^2)$ is even, which implies by Corollary \ref{cor:image} that the $(x,y)$ such that $x^2+y^2=1-z^2$ form a circle.
	
	Now take $z$ with $\ord(z)=0$. We have
	\[\ord(1-z^2)=\ord(1+z)+\ord(1-z).\]
	If both are zero, we are in the previous case. Otherwise, suppose without loss of generality that $\ord(1-z)>0$. Then $\ord(1+z)=0$ and $\ord(1-z^2)=\ord(1-z)$. If this is even, we have again a circle, if it is odd there is no preimage, and if it is $\infty$ (i.e. $z=1$), the preimage is a single point.
\end{proof}

The case $p\equiv 1\mod 4$ is similar:
\begin{proposition}
	\letpprime\ such that $p\equiv 1\mod 4$. Let $\mu:\mathrm{S}_p^2\to\Qp$ be the momentum map of the $p$-adic spin system given by \eqref{eq:muspin}. Given $z\in\Qp$, $\mu^{-1}(z)$ is equal to
	\begin{enumerate}
		\item two lines intersecting at $\{(0,0,z)\}$, if $z=\pm 1$;
		\item a circle (that is, a set homeomorphic to $\mathrm{S}^1_p$), otherwise.
	\end{enumerate}
\end{proposition}
\begin{proof}
	The cases are the same as in the previous proof, but now Corollary \ref{cor:image} gives a different result.
\end{proof}

If $p=2$, the calculations are more involved. Coincidentally, this is the only case in which the sphere is compact.
\begin{proposition}
	Let $p=2$. Let $\mu:\mathrm{S}_2^2\to\Q_2$ be the momentum map of the $p$-adic spin system given by \eqref{eq:muspin}. Given $z\in\Q_2$, $\mu^{-1}(z)$ equals
	\begin{enumerate}
		\item $\{(0,0,z)\}$, if $z=\pm 1$;
		\item a circle (that is, a set homeomorphic to $\mathrm{S}^1_2$), if
		\[z\in 2\Z_2\cup (5+16\Z_2)\cup (11+16\Z_2)\cup \Big\{1+2^m(3+4u)\Big| m\in\N, m\ge 3, u\in\Z_2\Big\}\]\[\cup \Big\{-1+2^m(1+4u)\Big| m\in\N, m\ge 3, u\in\Z_2\Big\};\]
		\item the empty set, otherwise.
	\end{enumerate}
\end{proposition}

\begin{proof}
	Let $r=\ord(z)$ and $s=\ord(1-z^2)$. Corollary \ref{cor:image} implies that we have one point when $1-z^2=0$, that is, when $z=\pm 1$, a circle if $\ord(1-z^2-2^s)\ge s+2$ and nothing if this is $s+1$.
	
	If $r>0$, $s=\ord(1-z^2)=0$ and $\ord(1-z^2-1)=2\ord(z)\ge 2$, so we have a circle in this case.
	
	If $r<0$, $s=2r$. Let $z=2^r+z'$, with $\ord(z')>r$.
	\[1-z^2=1-2^{2r}+2^{r+1}z'-z'^2\]
	\begin{align*}
	\ord(1-z^2-2^{2r}) & =\ord(1-2^{2r+1}-2^{r+1}z'-z'^2) \\
	& =2r+1
	\end{align*}
	because $2^{2r+1}$ is the term with smallest order. So this case has no solution.
	
	Now suppose $r=0$ (and $z\ne\pm 1$, where we already know the solution). In this case
	\[s=\ord(1-z^2)=\ord(1+z)+\ord(1-z).\]
	One of these two summands is $1$ and the other is greater than $1$, depending on $z\mod 4$. Suppose without loss of generality that $z\equiv 1\mod 4$, the other case will follow by changing $z$ by $-z$. Then $s=1+\ord(z-1)$.
	
	Let $m=\ord(z-1)\ge 2$, so that $s=m+1$. We write $z-1=2^m+z'$, with $\ord(z')>m$.
	\begin{align*}
		\ord(1-z^2-2^{m+1}) & =\ord(z^2+2^{m+1}-1) \\
		& =\ord((1+2^m+z')^2+2^{m+1}-1) \\
		& =\ord(2^{2m}+z'^2+2^{m+2}+2z'+2^{m+1}z')
	\end{align*}
	This is at least $m+2$, because each term has at least this order. We want to know when it is at least $s+2=m+3$. Let $z'=2^{m+1}t$, and the sum becomes
	\[2^{2m}+2^{2m+2}t^2+2^{m+2}+2^{m+2}t+2^{2m+2}t=2^{m+2}(2^{m-2}+2^mt^2+1+t+2^mt)\]
	This has order at least $m+3$ if and only if the parenthesis is even. If $m=2$, this happens when $t$ is even, and if $m>2$, when $t$ is odd.
	
	It remains only to plug this back into $z$:
	\begin{align*}
		z & =1+2^m+z' \\
		& =1+2^m+2^{m+1}t \\
		& =1+2^m(1+2t)
	\end{align*}
	If $m=2$, putting $t=2u$ results in
	\[z=1+4(1+4u)=5+16u.\]
	If $m>2$, putting $t=2u+1$ results in $z=1+2^m(3+4u)$. The other two cases in the statement correspond to changing $z$ by $-z$.
\end{proof}

\begin{corollary}\label{cor:image-spin}
	Let $\mu:\mathrm{S}_p^2\to\Qp$ be the momentum map of the $p$-adic spin system given by \eqref{eq:muspin}. The image of $\mu$ is given by:
	\begin{enumerate}
		\item $\Qp$, if $p\equiv 1\mod 4$;
		\item the set $(\Qp\setminus\Zp)\cup A_p\cup B_p\cup\{1,-1\}$ where \[A_p=\{x\in\Zp\mid x\not\equiv \pm 1\mod p\}\] and \[B_p=\{\pm 1+p^{2m}u\mid m\in\N, m\ge 1, u\in\Zp\setminus p\Zp\},\] if $p\equiv 3\mod 4$;
		\item the set $2\Z_2\cup (5+16\Z_2)\cup (11+16\Z_2)\cup A_2\cup B_2\cup\{1,-1\}$ where \[A_2=\{1+2^m(3+4u)\mid m\in\N, m\ge 3, u\in\Z_2\}\] and \[B_2=\{-1+2^m(1+4u)\mid m\in\N, m\ge 3, u\in\Z_2\},\] if $p=2$.
	\end{enumerate}
\end{corollary}

\begin{table}
	\begin{tabular}{p{.2\linewidth}|p{.09\linewidth}|p{.26\linewidth}|p{.11\linewidth}|p{.24\linewidth}|}
		& & \multicolumn{3}{l|}{$p$-adic} \\\cline{3-5}
		& Real & $p=2$ & $p\equiv 1\mod 4$ & $p\equiv 3\mod 4$ \\ \hline
		Uniqueness of flow & Unique & \multicolumn{3}{l|}{Not unique, but any two solutions coincide near $0$} \\ \hline
		Image of Hamiltonian & $[-1,1]$ & $2\Z_2\cup (5+16\Z_2)\cup (11+16\Z_2)\cup \{1+2^m(3+4u)\mid m\in\N, m\ge 3, u\in\Z_2\}\cup \{-1+2^m(1+4u)\mid m\in\N, m\ge 3, u\in\Z_2\}\cup\{1,-1\}$ & all $\Qp$ & $(\Qp\setminus\Zp)\cup\{x\in\Zp\mid x\not\equiv \pm 1\mod p\}\cup\{\pm 1+p^{2m}u\mid m\in\N, m\ge 1, u\in\Zp\setminus p\Zp\}\cup\{1,-1\}$ \\ \hline
		Fiber of $-1$ and $1$ & point & point & two lines & point \\ \hline
		Fiber of other points & circle (1 sector) & circle (4 sectors) & circle ($\infty$ sectors) & circle ($p+1$ sectors) \\ \hline
	\end{tabular}
	\medskip
	\caption{Comparison of the real and $p$-adic spin systems.}
	\label{table:spin}
\end{table}

The results of Corollary \ref{cor:image-spin} are in strong contrast with the real case, where the image of the momentum map is $[-1,1]$. See Table \ref{table:spin} for a summary of the results.

\section{The $p$-adic Jaynes-Cummings model}\label{sec:JC}

Now we turn our attention to the coupling of the models: oscillator (Section \ref{sec:circ-oscillator}) and spin (Section \ref{sec:spin}). In the real case, this is known as the Jaynes-Cummings model \cite{JayCum,PelVuN}.

First we need two definitions which will help us write the statements later on.

\begin{definition}\label{def:UV}
	\letpprime. For $j,h\in\Qp$, let $V_{j,h}$ be the set of points $(z,b)\in(\Qp)^2$ such that $2(j-z)$ is the sum of two squares (the numbers that satisfy this are characterized in Corollary \ref{cor:image}) and
	\[b^2=2(j-z)(1-z^2)-4h^2.\]
	Also, for $(z,b)\in V_{j,h}$ and $z\ne j$, we define
	\[U_{j,h}(z,b)=\left\{\left(\frac{2hu-bv}{u^2+v^2},\frac{2hv+bu}{u^2+v^2},z,u,v\right)\Big| u,v\in\Qp,u^2+v^2=2(j-z)\right\}\]
	and for $z=j$,
	\[U_{j,h}(z,b)=\Big\{(x,y,j,u,v)\Big| x^2+y^2=1-j^2,u^2+v^2=0,ux+vy=2h,uy-vx=b\Big\}.\]
\end{definition}

\subsection{General results}

\begin{figure}
	\includegraphics[width=0.32\linewidth]{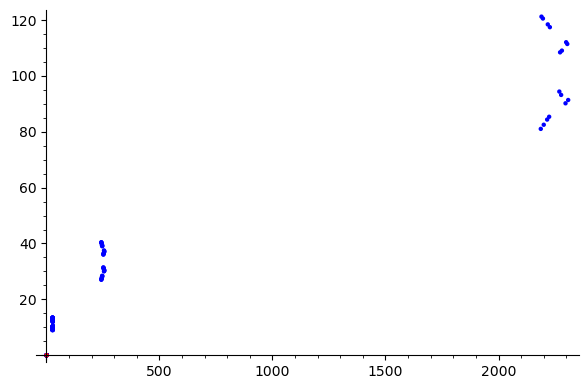}
	\includegraphics[width=0.32\linewidth]{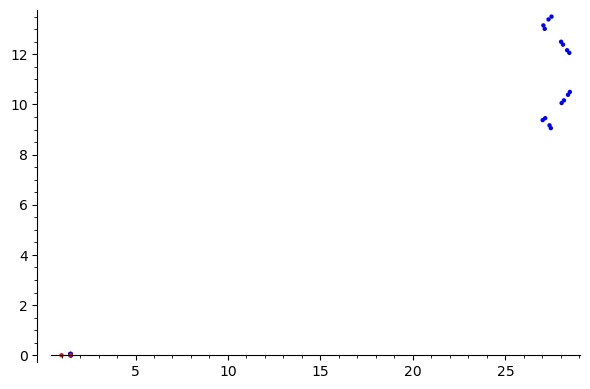}
	\includegraphics[width=0.32\linewidth]{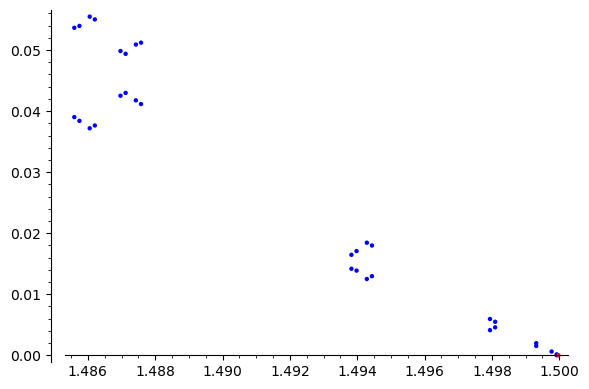}
	\caption{Three views of the critical points of the Jaynes-Cummings model for $p=2$. The blue curve (it is really a $p$-adic curve) corresponds to rank $1$ critical points. The two red dots correspond to rank $0$ critical points: the leftmost one, isolated, is the focus-focus point, the other the elliptic point. The rest of the figure corresponds to regular points. The blue curve extends indefinitely repeating the pattern in the first figure.}
	\label{fig:critical2}
\end{figure}

\begin{theorem}\label{thm:JC-general}
	\letpprime. The map
	\[F=(J,H):\mathrm{S}^2_p\times (\Qp)^2\to(\Qp)^2\]
	defined by
	\begin{equation}\label{eq:JC}
		\left\{\begin{aligned}
			J(x,y,z,u,v) & = \frac{u^2+v^2}{2}+z; \\
			H(x,y,z,u,v) & = \frac{ux+vy}{2},
		\end{aligned}\right.
	\end{equation}
	is a $p$-adic analytic integrable system on the $p$-adic analytic manifold $\mathrm{S}^2_p\times(\Qp)^2$ endowed with the $p$-adic analytic symplectic form $\dd\theta\wedge\dd z+\dd u\wedge\dd v$, where $(\theta,z)$ are angle-height coordinates on $\mathrm{S}^2_p\times(\Qp)^2$. In addition, this integrable system has the following properties.
	\begin{enumerate}
		\item The map $J:\mathrm{S}^2_p\times(\Qp)^2\to\Qp$ is the momentum map for the Hamiltonian circle action of $\mathrm{S}^1_p$ on $\mathrm{S}^2_p\times (\Qp)^2$ that rotates simultaneously horizontally about the vertical axis of $\mathrm{S}^2_p$ and about the origin of $(\Qp)^2$ (recall that the notion of $p$-adic Hamiltonian action is given in Appendix \ref{app:actions}).
		\item The set of rank $1$ critical points of $F$ is given by
		\[\Big\{(au,av,-a^2,u,v)\Big| a,u,v\in\Qp, (u,v)\ne(0,0),a^2(u^2+v^2)+a^4=1\Big\}.\]
		The set of rank $0$ critical points of $F$ is
		\[\{(0,0,-1,0,0),(0,0,1,0,0)\}.\]
		The set of critical values of $F$ is
		\[\left\{\left(\frac{1-3a^4}{2a^2},\frac{1-a^4}{2a}\right)\Big| a\in\Qp,1-a^4\text{ is sum of two squares}\right\}\cup\{(-1,0),(1,0)\}.\]
		\item The fiber of $F$ corresponding to $(j,h)$ is given by
		\[F^{-1}(\{(j,h)\})=\bigcup_{(z,b)\in V_{j,h}}U_{j,h}(z,b)\]
		where $V_{j,h}$ and $U_{j,h}(z,b)$ are given in Definition \ref{def:UV}.
	\end{enumerate}
\end{theorem}

\begin{figure}
	\includegraphics[width=0.32\linewidth]{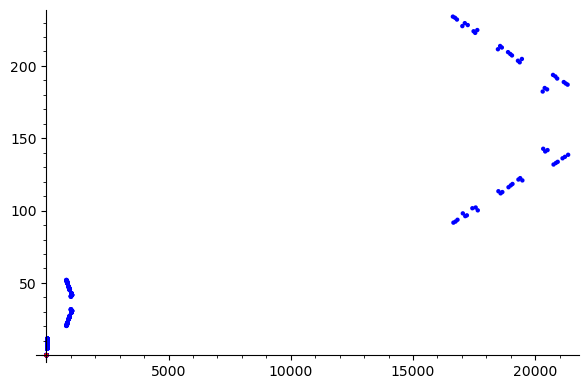}
	\includegraphics[width=0.32\linewidth]{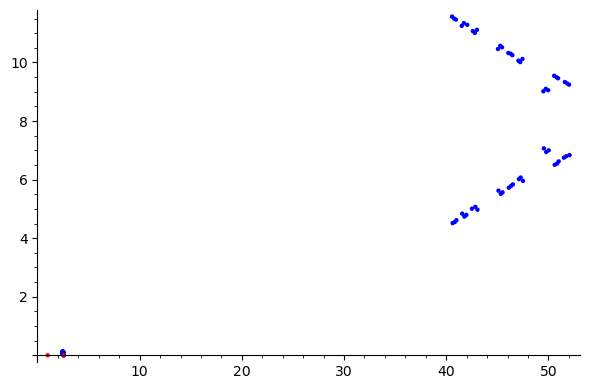}
	\includegraphics[width=0.32\linewidth]{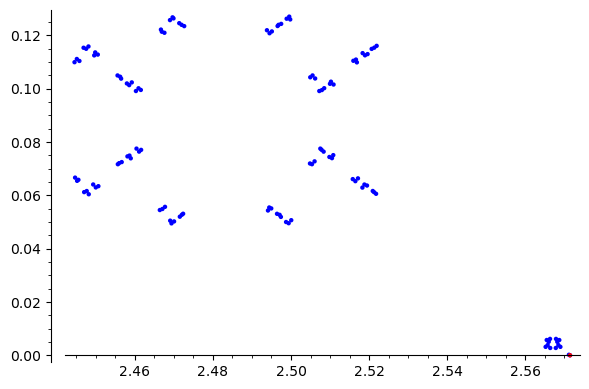}
	\caption{Three views of the critical points of the Jaynes-Cummings model for $p=3$. The blue curve (it is really a $p$-adic curve) corresponds to rank $1$ critical points. The two red dots correspond to rank $0$ critical points: the leftmost one, isolated, is the focus-focus point, the other the elliptic point. The rest of the figure corresponds to regular points. The blue curve extends indefinitely repeating the pattern in the first figure.}
	\label{fig:critical3}
\end{figure}

See Figures \ref{fig:critical2}, \ref{fig:critical3} and \ref{fig:critical5} for representations of the critical points for $p=2,3,5$.

\begin{proof}
	First, $x,y,z,u,v$ are analytic functions in $\mathrm{S}^2_p\times (\Qp)^2$: either they are the variables in the charts themselves, or they are related to them by analytic expressions of the form $z=\sqrt{1-x^2-y^2}$. This implies that $J$ and $H$ are analytic, because they are polynomials in these variables, and the $2$-form $\omega$ is also analytic.
	
	Next we see that
	\[\left\{\begin{aligned}
		\dd J &=u\dd u+v\dd v+\dd z; \\
		\dd H &=\frac{u\dd x+x\dd u+v\dd y+y\dd v}{2}.
	\end{aligned}\right.\]
	Using \eqref{eq:hamilton}, we get
	\[\left\{\begin{aligned}
		X_J &=-u\frac{\partial}{\partial v}+v\frac{\partial}{\partial u}+\frac{\partial}{\partial \theta}=-u\frac{\partial}{\partial v}+v \frac{\partial}{\partial u}-x\frac{\partial}{\partial y}+y \frac{\partial}{\partial x}; \\
		X_H &=\frac{1}{2}\left(-uz\frac{\partial}{\partial y}-x\frac{\partial}{\partial v}+vz\frac{\partial}{\partial x}+y\frac{\partial}{\partial u}\right).
	\end{aligned}\right.\]
	This concludes part (1) of the theorem, because the first vector field corresponds exactly to the rotation action in the planes $(x,y)$ and $(u,v)$. We also get
	\begin{align*}
		\{J,H\} & =\omega(X_J,X_H) \\
		& =-\frac{-uyz+vxz}{2z}+\frac{vx-uy}{2}=0.
	\end{align*}
	Now we have to find where $\dd J$ and $\dd H$ are collinear. If $z=0$, $\dd x$ and $\dd y$ are proportional in $\mathrm{S}^2_p$, and $\{\dd x,\dd z,\dd u,\dd v\}$ form a basis of the cotangent space. So the only way the two forms can be proportional is if $\dd H=0$, but that implies $x=y=0$, which is not possible. Hence, $z\ne 0$.
	
	Now the problem reduces to find collinearity of
	\[\left\{\begin{aligned}
		z\dd J &=uz\dd u+vz\dd v-x\dd x-y\dd y; \\
		2\dd H &=x\dd u+y\dd v+u\dd x+v\dd y.
	\end{aligned}\right.\]
	By calling $-a$ the proportionality constant and using that $\{\dd x,\dd y,\dd u,\dd v\}$ is a basis, this boils down to
	\[\left\{\begin{aligned}
		uz &=-ax; \\
		vz &=-ay; \\
		x & =au; \\
		y & =av.
	\end{aligned}\right.\]
	If $u$ and $v$ are both $0$, we obtain $x=y=0$ and $z=\pm1$, which leads to two points where the rank is $0$. Otherwise
	\[x=au,y=av,z=-a^2\]
	and substituting in $x^2+y^2+z^2=1$, we get \[a^2(u^2+v^2)+a^4=1.\] The description of the critical values follows from substituting the expressions for the critical points in $J$ and $H$. This completes part (2). Also, these points form a set with measure $0$ (it has only $2$ degrees of freedom), so, together with $\{J,H\}=0$, we obtain that this is an integrable system.
	
	Now we turn to study the fibers. Let $(j,h)\in(\Qp)^2$. Our goal is to find $(x,y,z,u,v)$ such that
	\[\left\{\begin{aligned}
		\frac{u^2+v^2}{2}+z & =j \\
		ux+vy & =2h \\
		x^2+y^2+z^2 & =1
	\end{aligned}\right.\]
	Define $b=uy-vx$. These equations imply that
	\[u^2+v^2=2(j-z)\]
	and
	\begin{align*}
		b^2 & =(uy-vx)^2 \\
		& =(u^2+v^2)(x^2+y^2)-(ux+vy)^2 \\
		& =2(j-z)(1-z^2)-4h^2
	\end{align*}
	so $(z,b)\in V_{j,h}$. Once we have $z$ and $b$, the next step is to choose suitable $(u,v)$. Finally, for $(x,y)$ we have a linear system
	\[\left\{\begin{aligned}
		ux+vy & =2h; \\
		uy-vx & =b.
	\end{aligned}\right.\]
	The solution is unique if $u^2+v^2\ne 0$, that is, if $z\ne j$, and leads to
	\[\left\{\begin{aligned}
		x & =\frac{2hu-bv}{u^2+v^2}; \\
		y & =\frac{2hv+bu}{u^2+v^2}
	\end{aligned}\right.\]
	as we wanted.
	This finishes part (3).
\end{proof}

\begin{figure}
	\includegraphics[width=0.32\linewidth]{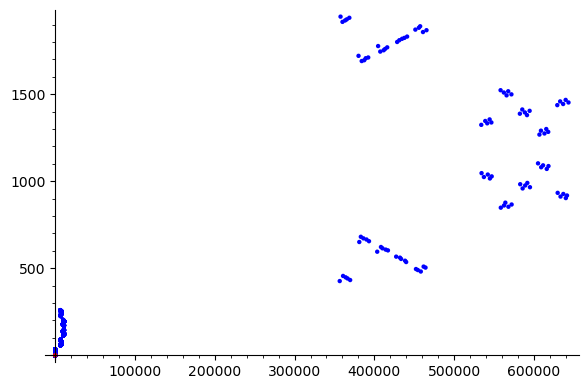}
	\includegraphics[width=0.32\linewidth]{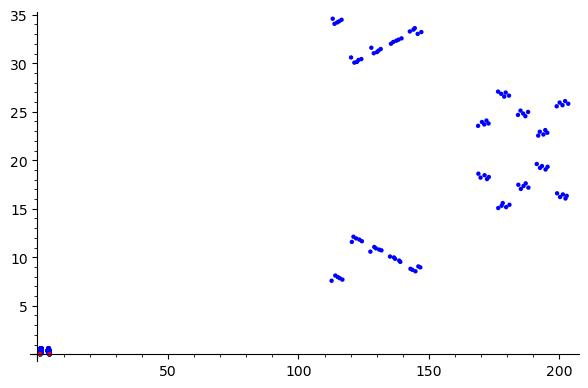}
	\includegraphics[width=0.32\linewidth]{crit-5}
	\caption{Three views of the critical points of the Jaynes-Cummings model for $p=5$. The blue curve (it is really a $p$-adic curve) corresponds to rank $1$ critical points. The two red dots correspond to rank $0$ critical points: the leftmost one is the focus-focus point (in this case it is not isolated), the other the elliptic point. The rest of the figure corresponds to regular points. The blue curve extends indefinitely repeating the pattern in the first figure.}
	\label{fig:critical5}
\end{figure}

\begin{corollary} \label{cor:products}
	\letpprime. Let $F:\mathrm{S}_p^2\times(\Qp)^2\to(\Qp)^2$ be the $p$-adic analytic Jaynes-Cummings model given by \eqref{eq:JC}. Let $V_{j,h}$ and $U_{j,h}(z,b)$ as in Definition \ref{def:UV}. Then, for any $j,h\in\Qp$ and $z,b\in V_{j,h}$, $U_{j,h}(z,b)$ is homeomorphic to $\mathrm{S}^1_p$ except if $z=j$, in which case the following statements hold:
	\begin{enumerate}
		\item If $p\not\equiv 1\mod 4$ and $(j,h)\notin\{(-1,0),(1,0)\}$, then $U_{j,h}(z,b)$ is still homeomorphic to $\mathrm{S}^1_p$.
		\item If $p\not\equiv 1\mod 4$ and $(j,h)\in\{(-1,0),(1,0)\}$, then $U_{j,h}(z,b)$ is a point.
		\item If $p\equiv 1\mod 4$ and $(j,h)\notin\{(-1,0),(1,0)\}$, then $U_{j,h}(z,b)$ is homeomorphic to $\Qp^*$.
		\item If $p\equiv 1\mod 4$ and $(j,h)\in\{(-1,0),(1,0)\}$, then $U_{j,h}(z,b)$ is the union of two $2$-planes.
	\end{enumerate}
\end{corollary}

\begin{proof}
	If $z\ne j$, then $U_{j,h}(z,b)$ is homeomorphic to $\mathrm{S}^1_p$. Now we must see that, if $z=j$, $U_{j,h}(z,b)$ is homeomorphic to $\mathrm{S}^1_p$, a point, $\Qp^*$, or two $2$-planes, in the different cases. First note that $z=j$ implies $u^2+v^2=0$.
	\begin{enumerate}
		\item If $p\not\equiv 1\mod 4$, we must have $u=v=0$. If $h\ne 0$, the conditions become incompatible, so it must be $h=0$, and $x$ and $y$ can be chosen freely such that $x^2+y^2=1-z^2=1-j^2$. This leads to a circle if $1-j^2$ is sum of two squares, a single point if $j=\pm 1$, and nothing otherwise.
		\item If $p\equiv 1\mod 4$, we have that $b^2=-4h^2$, and it is always a square in this case, and $b=\pm 2\ii h$ for $\ii$ such that $\ii^2=-1$. Once we choose a solution, the solutions to $u^2+v^2=0$ are of the form $(u,\ii u)$ and $(u,-\ii u)$. So we have $v=\epsilon \ii u$, for $\epsilon=\pm 1$, and
		\[2h=ux+vy=ux+\epsilon \ii uy,\]
		as well as that
		\[b=uy-vx=uy-\epsilon \ii ux=-\epsilon \ii (ux+\epsilon \ii uy)=-2\epsilon \ii h,\]
		hence the value of $b$ decides $\epsilon$. $u$ can be freely chosen, and it forces $v$. We must now make three cases.
		\begin{enumerate}
			\item If $u=v=0$, we must have $h=0$. We are again in case (1), but with two lines instead of a point for $j=\pm 1$ and a circle for any other $j$.
			\item If $u,v\ne 0$ and $h=0$, we have $x+\epsilon \ii y=0$. This implies $x^2+y^2=0$, hence $j=\pm 1$. In this case $U_{j,h}(z,b)$ is formed by the points $(x,\epsilon \ii x,j,u,\epsilon \ii u)$, which form two planes. The previous case is already included here.
			\item If $u,v\ne 0$ and $h\ne 0$, we have that
			\[x=\frac{2h}{u}-\epsilon \ii y.\]
			Substituting in $x^2+y^2+z^2=1$,
			\begin{align*}
			1 & =\frac{4h^2}{u^2}-\frac{4\epsilon \ii hy}{u}-y^2+y^2+j^2 \\
			& =-\frac{4h^2}{v^2}+\frac{4hy}{v}+j^2.
			\end{align*}
			Solving in $y$,
			\[y=\frac{v}{4h}(1-j^2)+\frac{h}{v},\]
			and analogously
			\[x=\frac{u}{4h}(1-j^2)+\frac{h}{u}.\]
			Hence, in this case $u$ determines $v$, $x$ and $y$. Every value $u\ne 0$ is valid, so we have $U_{j,h}(z,b)$ homeomorphic to $\Qp^*$. \qedhere
		\end{enumerate}
	\end{enumerate}
\end{proof}

\begin{remark}
	In the real case, the fiber of $(J,H)$ is a point at $(-1,0)$, homeomorphic to a circle at the points in two curves, homeomorphic to a pinched torus at $(1,0)$ and homeomorphic to a torus otherwise (Figure \ref{fig:real-JC-image}). This seems different to the $p$-adic case, but it actually reproduces parts (1) and (2) of Corollary \ref{cor:products}: $V_{j,h}$ in the real case is topologically a circle, except in the rank $1$ critical points and $(-1,0)$, where it degenerates to a point. Multiplying $V_{j,h}$ by $\mathrm{S}^1_p$ and pinching one $\mathrm{S}^1_p$ at the two points $(-1,0)$ and $(1,0)$ leaves exactly the fibers we know. Parts (3) and (4) have no real equivalent, and correspond to special properties of the $p$-adic fields.
\end{remark}

\subsection{Fibers and image for $p>2$}

To complete the study of the fibers and image of the Jaynes-Cummings model, we only need to describe $V_{j,h}$ for the different values of $j$ and $h$. As it will turn out, the results are different for $p=2$ respect to $p>2$. (This is not strange, after having seen that $\mathrm{S}^2_p$ is compact only for $p=2$.) We start with $p>2$. This will actually be divided in two cases, depending on $p$ modulo $4$.

\begin{definition}\label{def:classes}
Given $j\in\Qp$, we classify the $p$-adic numbers in three classes. If $p\equiv 1\mod 4$, the classes are as follows:
\begin{itemize}
	\item \textit{first class:} $z$ such that $2(z^2-1)(z-j)$ is a square;
	\item \textit{second class:} $z$ such that $2(z^2-1)(z-j)$ has even order but is not a square;
	\item \textit{third class:} $z$ such that $2(z^2-1)(z-j)$ has odd order.
\end{itemize}
If $p\equiv 3\mod 4$, the classes need to be slightly modified:
\begin{itemize}
	\item \textit{first class:} $z$ such that $2(z^2-1)(z-j)$ is a square and the two parts $z^2-1$ and $2(z-j)$ separately have even order.
	\item \textit{second class:} $z$ such that $z^2-1$ and $2(z-j)$ have even order, but their product is not a square.
	\item \textit{third class:} $z$ such that at least one of $z^2-1$ and $2(z-j)$ has odd order.
\end{itemize}
Here, $1$, $-1$ and $j$ are considered to be in the first class.
\end{definition}

In what follows, for a fixed $j\in\Qp$, we call
\[\ord(2(z^2-1)(z-j))=\ord(z+1)+\ord(z-1)+\ord(z-j)\]
the \textit{potential} at $z\in\Qp$. This value increases as we move toward $1$, $-1$ and $j$: the ``equipotential contours'' are initially balls containing these three numbers, then one around a number and another around the other two, and finally a different ball around each number.

\begin{theorem}\label{thm:JC-z}
	Let $V_{j,h}$ be as given in Definition \ref{def:UV}. The projection of $V_{j,h}$ to the coordinate $z$ consists of:
	\begin{enumerate}
		\item all numbers in the first class according to Definition \ref{def:classes} with potential less than $2\ord(h)+2$,
		\item some numbers in the first and second classes according to Definition \ref{def:classes} with potential exactly $2\ord(h)+2$ (which points depends on the concrete value of $h$), and
		\item if $p\equiv 1\mod 4$, all numbers with potential greater than $2\ord(h)+2$ (independently of the class).
	\end{enumerate}
\end{theorem}

\begin{proof}
	This follows from a case analysis of
	\[b^2=2(z^2-1)(z-j)-4h^2.\]
	If the potential is less than $2\ord(h)+2$, the part $2(z^2-1)(z-j)$ wins: its order and leading digit determine those of $b^2$. This must be a square if we want $b^2$ to be a square. Also, if $p\equiv 3\mod 4$, the two factors separately must have even order.
	
	If the potential is exactly $2\ord(h)+2$, we are subtracting two things with the same order, so the result may or may not be a square.
	
	Finally, if the potential is greater than $2\ord(h)+2$, the order of the result is $2\ord(h)+2$. This is even, and the leading digit will be that of $-4h^2$. This is a square if and only if $p\equiv 1\mod 4$.
\end{proof}

\begin{corollary}\label{cor:JC-surj}
	\letpprime. Let $F:\mathrm{S}_p^2\times(\Qp)^2\to(\Qp)^2$ be the $p$-adic analytic Jaynes-Cummings model given by \eqref{eq:JC}. If $p>2$, the image $F(\mathrm{S}_p^2\times(\Qp)^2)$ of $F$ (i.e. the classical spectrum of the system) is the whole $(\Qp)^2$.
\end{corollary}

\begin{proof}
	We must prove that $V_{j,h}$ is never empty. This is a consequence of the previous result: for all $z$ with low enough order,
	\[\ord(2(z^2-1)(z-j))=3\ord(z)\]
	and $\ord(z-j)=\ord(z)$. The number $2(z^2-1)(z-j)$ will be a square if and only if $2z^3$ is a square, that is, if and only if $2z$ is a square. So any $z$ with low enough order such that $2z$ is a square will work, because $\ord(z-j)$ will automatically be even.
\end{proof}

We now know the values of $z$ that appear in $V_{j,h}$. The next step is to determine which ones lead to one value of $b$ (which is zero) and which ones lead to two, but this is easy: there is only one $b$ if and only if
\[2(z^2-1)(z-j)=4h^2.\]
Hence, there are at most three values of $z$ with this condition, and all of them are in the first class or in the third. Only the $z$ in the first class (i.e. those with $\ord(z^2-1)$ and $\ord(z-j)$ even, if $p\equiv 3\mod 4$, and all of them if $p\equiv 1\mod 4$) are involved here. This completes the characterization of $V_{j,h}$.

Corollary \ref{cor:products} tells us the form of the ``sub-fiber'' $U_{j,h}(z,b)$, that can be a circle, a point, a curve, or the union of two $2$-planes. With this information we can give a topological characterization of the fibers, analogous to the one for the real case. For this, we need a criterion to decide whether a zero of an analytic function is surrounded by squares, non-squares, or a mixture. (If we interpret ``squares'' as the $p$-adic equivalent of ``positive real numbers'', this is similar to the criterion to classify local extrema of real functions.)

\begin{proposition}\label{prop:square}
	\letpprime. Let
	\[f(x)=\sum_{i=1}^\infty a_ix^i\]
	be a power series in one variable with coefficients in $\Qp$ and without constant term that converges in some open ball. Then the following statements hold.
	\begin{enumerate}
		\item If $a_1\ne 0$, for any $\epsilon>0$, there are $x_1$ and $x_2$ such that $|x_1|<\epsilon$, $|x_2|<\epsilon$, $f(x_1)$ is a square and $f(x_2)$ is not.
		\item If $a_1=0$ and $a_2$ is a square different from $0$, there is $\epsilon>0$ such that, for all $x$ with $|x|<\epsilon$, $f(x)$ is a square.
		\item If $a_1=0$ and $a_2$ is a non-square, there is $\epsilon>0$ such that, for all $x$ with $|x|<\epsilon$, $f(x)$ is a non-square.
	\end{enumerate}
\end{proposition}
\begin{proof}
	The condition of $f(x)$ being a square is only determined by the order of $f(x)$ and its leading digit. For $x$ small enough, the order and leading digit of $f(x)$ are the same as those of the first nonzero term of the series, because the rest of terms have greater order. Hence the problem reduces to the case where the series has only one term.
	
	If that term is $a_1x$, to make this a square we take a square as $x$ if $a_1$ is a square, and a non-square otherwise. To make it a non-square, we do the opposite.
	
	If that term is $a_2x^2$, that is automatically a square if $a_2$ is a square, and a non-square otherwise.
\end{proof}

\begin{corollary}\label{cor:square}
	\letpprime. Let $U\subset \Qp$ be an open set, $x_0\in U$, and $f:U\to\Qp$ analytic such that $f(x_0)=0$. Also let $f'$ be the derivative of $f$ and $f''$ be its second derivative. Then the following statements hold.
	\begin{enumerate}
		\item If $f'(x_0)\ne 0$, then for any $\epsilon>0$, there are $x_1$ and $x_2$ such that $|x_1-x_0|<\epsilon$, $|x_2-x_0|<\epsilon$, $f(x_1)$ is a square and $f(x_2)$ is not.
		\item If $f'(x_0)=0$ and $f''(x_0)/2$ is a square different from $0$, then there is $\epsilon>0$ such that, for all $x$ with $|x-x_0|<\epsilon$, $f(x)$ is a square.
		\item If $f'(x_0)=0$ and $f''(x_0)/2$ is a non-square, then there is $\epsilon>0$ such that, for all $x$ with $|x-x_0|<\epsilon$, $f(x)$ is a non-square.
	\end{enumerate}
\end{corollary}

\begin{proof}
	Apply Proposition \ref{prop:square} to $f(x-x_0)$.
\end{proof}

\begin{lemma}\label{lemma:firstclass}
	\letpprime\ such that $p\equiv 3\mod 4$. If $z_0$ is a first class number, any $z$ sufficiently near $z_0$ such that $2(z^2-1)(z-j)-4h^2$ is a square is also in the first class, except if $j=\pm 1$ and $z_0=j$, in which case some numbers are in the first class and the rest in the third.
\end{lemma}
\begin{proof}
	If $z_0\notin\{1,-1,j\}$, the two factors $z^2-1$ and $z-j$ will not be zero near $z_0$, so their order will be locally constant.
	
	If $z_0\in\{1,-1,j\}$ and it is not true that $z_0=j=\pm 1$, one of the factors will not be zero and its order will be locally constant. In this case, $-4h^2$ is a square, which implies $h=0$, and $2(z^2-1)(z-j)$ is a square, so the order of the other factor is also preserved. In the case $z_0=j=\pm 1$, the order of the two factors is unbounded, but as $z_0+j$ has order $0$, the two orders always coincide. If they are even, $z$ is in the first class, and if they are odd, $z$ is in the third class.
\end{proof}

\newcommand{\fibra}[1]{\includegraphics[width=0.5\linewidth]{fibra-#1}}

\begin{figure}
	\begin{tabular}{c|c}
		\fibra{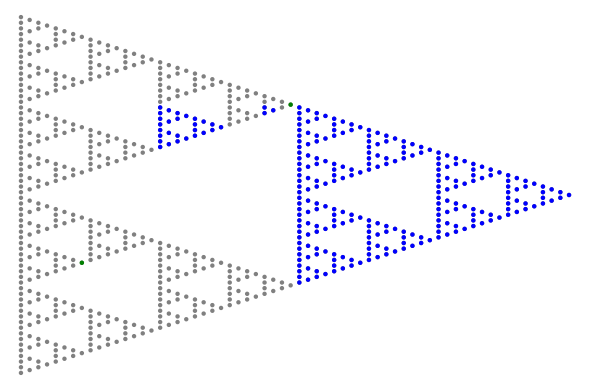} & \fibra{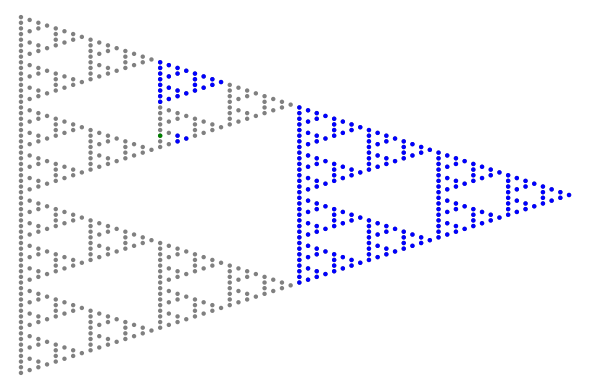} \\
		$(23,0)$ & $(-47/8,-15/4)$ \\ \hline
		\fibra{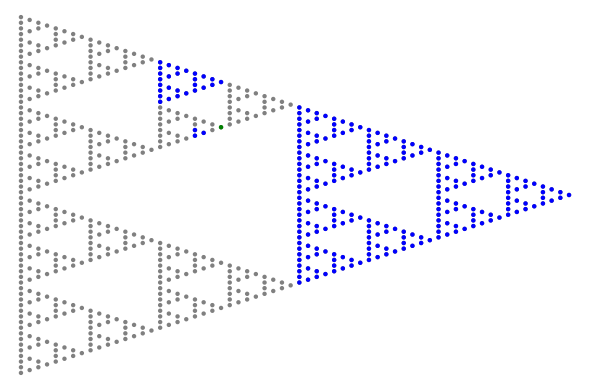} & \fibra{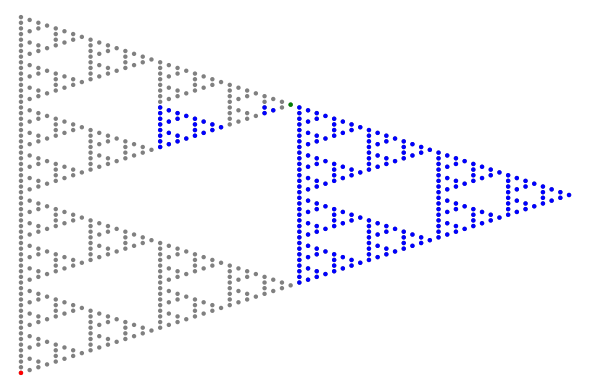} \\
		$(23,3)$ & $(-1,0)$ \\ \hline
		\fibra{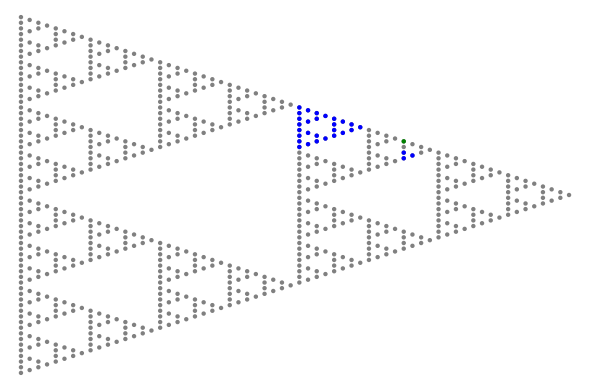} & \fibra{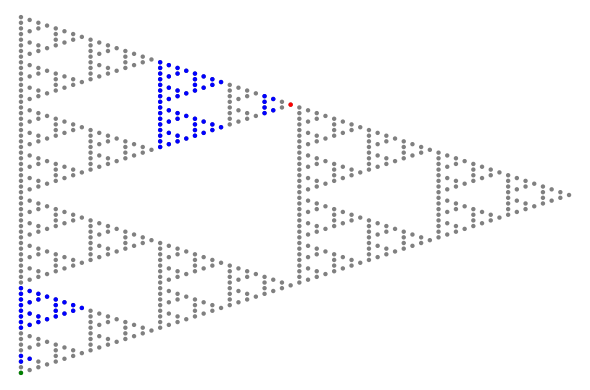} \\
		$(23,1)$ & $(1,0)$
	\end{tabular}
	\caption{Fibers of the Jaynes-Cummings model when $p=3$. The three on the left have rank $2$, the one on the top right has rank $1$, and the other two have rank $0$. The blue points represent values of $z$ for which $x,y,u,v$ form two circles, and at the green points they form only one circle (because $b=0$ at those points). The grey points are values of $z$ that are not in the fiber. In the two figures for rank $0$, there is a red point at $z=j$, which represents a single point in the fiber (isolated in the elliptic case).}
	\label{fig:fibers3}
\end{figure}

\begin{theorem}\label{thm:JC-topology}
	\letpprime\ such that $p\equiv 3\mod 4$. Let $F:\mathrm{S}_p^2\times(\Qp)^2\to(\Qp)^2$ be the $p$-adic analytic Jaynes-Cummings model given by \eqref{eq:JC}.
	\begin{enumerate}
		\item If $(j,h)=(-1,0)$, then the fiber $F^{-1}(\{(j,h)\})$ is the disjoint union of a $2$-dimensional $p$-adic analytic submanifold and an isolated point at $(0,0,-1,0,0)$.
		\item If $(j,h)=(1,0)$, then the fiber $F^{-1}(\{(j,h)\})$ has dimension $2$ and a singularity at $(0,0,1,0,0)$.
		\item If $(j,h)$ is a rank $1$ critical value, that is, $j=(1-3a^4)/2a^2$ and $h=(1-a^4)/2a$ for some $a\in\Qp$ such that $1-a^4$ is the sum of two squares:
		\begin{enumerate}
			\item If $-3a^4-1$ is a non-square modulo $p$, then the fiber $F^{-1}(\{(j,h)\})$ is the disjoint union of a $2$-dimensional $p$-adic analytic submanifold and the circle formed by the critical points for that value:
			\[\Big\{(au,av,-a^2,u,v)\Big| u,v\in\Qp,a^2(u^2+v^2)+a^4=1\Big\}\]
			\item If $-3a^4-1$ is a square modulo $p$, then the fiber $F^{-1}(\{(j,h)\})$ has dimension $2$ and singularities at the critical points.
		\end{enumerate} 
		\item For the rest of values of $(j,h)\in F(\mathrm{S}_p^2\times(\Qp)^2)$, the fiber $F^{-1}(\{(j,h)\})$ is a $2$-dimensional manifold.
	\end{enumerate}
\end{theorem}

\begin{proof}
	We define
	\[f_{j,h}(z):=2(z^2-1)(z-j)-4h^2\]
	By Lemma \ref{lemma:firstclass}, when $z_0\in V_{j,h}$ is a first class number and not a rank $0$ critical point, all $z$ near $z_0$, provided $f_{j,h}(z)$ is a square, are also in $V_{j,h}$. So we only need to analyze whether $f_{j,h}(z)$ is a square or not.
	
	If $b\ne 0$, the sub-fiber $U_{j,h}(z,b)$ is a circle, and $V_{j,h}$ has dimension $1$ because $b$ varies analytically with $z$. So in these points, the fiber has locally dimension $2$ and only the cases with $b=0$ need separate treatment.
	
	To achieve $b=0$ we need $f_{j,h}(z_0)=0$.
	We will call $f'_{j,h}$ the derivative of $f_{j,h}$. If this is nonzero at $z_0$, by Corollary \ref{cor:square}, some $z$ near $z_0$ give two values of $b$ and other $z$ give no possible value. But
	\begin{align*}
		\frac{\partial z}{\partial b}(z_0,0) & =2b\frac{\partial z}{\partial (b^2)}(z_0,0) \\
		& =2b\left(\frac{\partial (b^2)}{\partial z}(z_0,0)\right)^{-1} \\
		& =\frac{2b}{f'_{j,h}(z_0)}=0
	\end{align*}
	so we can parameterize near $(z_0,0)$ by $b$, and $V_{j,h}$ still has dimension $1$ at $(z_0,0)$. If $f'_{j,h}(z_0)=0$, we have two cases by Corollary \ref{cor:square}:
	\begin{itemize}
		\item If $f''_{j,h}(z_0)/2$ is a square, $f_{j,h}$ is a square near $z_0$. Unless we are at a rank $0$ point, this means that all $z$ near $z_0$ give two values of $b$ (while $z_0$ gives only one), so this is a singularity in $V_{j,h}$.
		\item If $f''_{j,h}(z_0)/2$ is a non-square, $f_{j,h}$ is a non-square near $z_0$, which means that no $z$ near $z_0$ appears in $V_{j,h}$, so this is an isolated point.
	\end{itemize}
	
	We have that
	\[f'_{j,h}(z_0)=2(z_0^2-1)+4z_0(z_0-j)=6z_0^2-4jz_0-2\]
	The cases where we have not dimension $1$ are those with $f_{j,h}(z_0)=f'_{j,h}(z_0)=0$, and this leads to
	\[6z_0^2-4jz_0-2=0\Longrightarrow j=\frac{3z_0^2-1}{2z_0}\]
	and
	\begin{align}
	0 & =2(z_0^2-1)(z_0-j)-4h^2\nonumber \\
	& =2(z_0^2-1)\frac{1-z_0^2}{2z_0}-4h^2\nonumber \\
	& =-\frac{(z_0^2-1)^2}{z_0}-4h^2\label{eq:JC-rel-z-h}
	\end{align}
	If $h\ne 0$, this can be rewritten as
	\[z_0=-\left(\frac{z_0^2-1}{2h}\right)^2\]
	Calling $-a$ what is inside the parenthesis, we get $z_0=-(-a)^2=-a^2$ and
	\[(j,h)=\left(\frac{1-3a^4}{2a^2},\frac{1-a^4}{2a}\right)\]
	Hence, this case leads exactly to the rank $1$ critical points. We must now check the second derivative:
	\begin{align*}
	f''_{j,h}(z_0) & =12z_0-4j \\
	& =-12a^2-\frac{2-6a^4}{a^2} \\
	& =-\frac{6a^4+2}{a^2}.
	\end{align*}
	This divided by $2$ is a square if and only if $-3a^4-1$ is a square.
	
	Now suppose $h=0$. The equation \eqref{eq:JC-rel-z-h} now gives
	\[-\frac{(z_0^2-1)^2}{z_0}=0,\]
	which implies $z_0=\pm 1$, and
	\[j=\frac{3z_0^2-1}{2z_0}=\pm 1.\]
	These are the rank $0$ critical points. The second derivative is
	\[f''_{j,h}(z_0)=12z_0-4j=\pm 8\]
	$4$ is always a square, and for the current primes $-4$ is a non-square (because $-1$ is). By Corollary \ref{cor:square} we have an isolated point for $(j,h)=(-1,0)$. If $(j,h)=(1,0)$ we cannot use as before Lemma \ref{lemma:firstclass} to ensure that the points near the critical one are in $V_{j,h}$, because we are at a rank $0$ point, but the same lemma tells us that there are points in $V_{j,h}$ arbitrarily close to the critical point, so we have a singularity. This completes the proof.
\end{proof}

For the other primes, we have:

\begin{figure}
	\begin{tabular}{c|c}
		\fibra{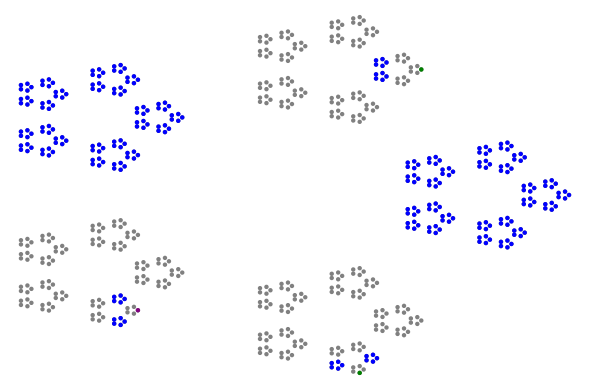} & \fibra{5-r1} \\
		$(23,0)$ & $(-47/8,-15/4)$ \\ \hline
		& \\[-13pt]
		\fibra{5-r2-h5} & \fibra{5-elip} \\
		$(23,5)$ & $(-1,0)$ \\ \hline
		& \\[-13pt]
		\fibra{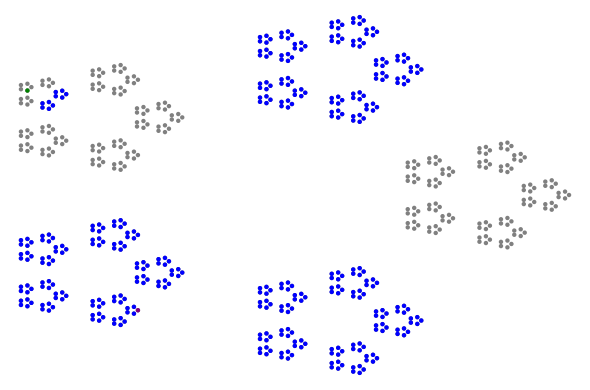} & \fibra{5-foco} \\
		$(23,1)$ & $(1,0)$
	\end{tabular}
	\caption{Fibers of the Jaynes-Cummings model for $p=5$. The three fibers on the left have rank $2$, the one on the top right has rank $1$, and the other two have rank $0$. The blue points represent values of $z$ for which $x,y,u,v$ form two circles, at the green points they form only one circle (because $b=0$ at those points), and $z=j$ is marked by a purple point if $x,y,u,v$ have dimension $1$ and by a yellow point if they have dimension $2$ (two planes). The grey points are values of $z$ that are not in the fiber.}
	\label{fig:fibers1}
\end{figure}

\begin{theorem}\label{thm:JC-topology1}
	\letpprime\ such that $p\equiv 1\mod 4$. Let $F:\mathrm{S}_p^2\times(\Qp)^2\to(\Qp)^2$ be the $p$-adic analytic Jaynes-Cummings model given by \eqref{eq:JC}, $\ii\in\Qp$ such that $\ii^2=-1$ and
	\[L_{-1}=\Big\{\left(\delta u,\delta\epsilon \ii u,-1,u,\epsilon \ii u\right)\,\,\Big|\,\, u\in\Qp,\epsilon=\pm 1,\delta=\pm 1\Big\},\]
	\[L_1=\Big\{\left(-\delta\epsilon \ii u,\delta u,1,u,\epsilon \ii u\right)\,\,\Big|\,\, u\in\Qp,\epsilon=\pm 1,\delta=\pm 1\Big\}.\]
	\begin{enumerate}
		\item If $(j,h)=(\pm 1,0)$ is a rank $0$ critical value, then the fiber $F^{-1}(\{(j,h)\})$ has dimension $2$ and a singularity at every point of $L_j$.
		\item If $(j,h)$ is a rank $1$ critical value, that is, $j=(1-3a^4)/2a^2$ and $h=(1-a^4)/2a$ for some $a\in\Qp$:
		\begin{enumerate}
			\item If $-3a^4-1$ is a non-square modulo $p$, then the fiber $F^{-1}(\{(j,h)\})$ is the disjoint union of a $2$-dimensional $p$-adic analytic submanifold and the circle formed by the critical points for that value:
			\[\Big\{(au,av,-a^2,u,v)\,\,\Big|\,\, u,v\in\Qp,a^2(u^2+v^2)+a^4=1\Big\}\]
			\item If $-3a^4-1$ is a square modulo $p$, it has dimension $2$ and singularities at the critical points.
		\end{enumerate} 
		\item For the rest of values of $(j,h)$, the fiber $F^{-1}(\{(j,h)\})$ is a $2$-dimensional $p$-adic analytic manifold.
	\end{enumerate}
\end{theorem}
\begin{proof}
	This is similar to the case $p\equiv 3\mod 4$, so we will focus on the differences between the two proofs.
	\begin{itemize}
		\item Lemma \ref{lemma:firstclass} is not needed, because for these primes there is no condition in the factors $z^2-1$ and $z-j$.
		\item For $b\ne 0$, the sub-fiber may not be a circle: this happens if $z=j$. But even in this case the sub-fiber $U_{j,h}(z,b)$ has dimension $1$, so this does not alter the situation.
		\item The cases with $b=0$ and $h=0$ lead again to the two rank $0$ points, but in this case their sub-fiber has dimension $2$ instead of $0$. This creates a singularity at the points in $U_{j,0}(j,0)$, for $j=\pm 1$, which are in the limit of the dimension $1$ sub-fibers when $z$ tends to $j$. So we must calculate this limit. Fixing $u$, we know (Corollary \ref{cor:products}) that $v$ tends to $\epsilon \ii u$ and $b$ tends to $0$, and
		\begin{align*}
			\lim_{z\to j} \left(\frac{2hu-bv}{u^2+v^2},\frac{2hv+bu}{u^2+v^2}\right) & =\lim_{z\to j} \left(\frac{-bv}{2(j-z)},\frac{bu}{2(j-z)}\right) \\
			& =(-\epsilon \ii u,u)\lim_{z\to j}\frac{b}{2(j-z)}
		\end{align*}
		Squaring the limit
		\begin{align*}
			\lim_{z\to j}\frac{b^2}{4(j-z)^2} & =\lim_{z\to j}\frac{2(j-z)(1-z^2)}{4(j-z)^2} \\
			& =\lim_{z\to j}\frac{2(j-z)(j^2-z^2)}{4(j-z)^2} \\
			& =\lim_{z\to j}\frac{j+z}{2}=j
		\end{align*}
		so the original limit gives $\delta(u,\epsilon \ii u)$, where $\delta=\pm 1$, if $j=-1$, and $\delta(-\epsilon \ii u,u)$ if $j=1$.
		The second derivatives, $4$ and $-4$, are now both perfect squares, so the part of the fiber with $z\ne j$ has points arbitrarily close to the two planes for $z=j$.
	\end{itemize}
\end{proof}

\begin{remark}
	$L_{-1}$ consists of four lines of rank $1$ points with $a=\delta$ that meet at the rank $0$ point $(0,0,-1,0,0)$. Analogously, $L_1$ consists of four lines of rank $1$ points with $a=-\delta\epsilon \ii$ that meet at $(0,0,1,0,0)$.
\end{remark}

\subsection{Fibers and image if $p=2$}
In what follows, we will talk about the $2$-adic expansions of numbers. We say that a $2$-adic number \emph{ends in $a_1a_2\ldots a_k$}, where each $a_i$ is $0$ or $1$, if the first $k$ digits (the rightmost ones) are $a_1a_2\ldots a_k$, not counting the infinitely many zeros at the right. With this definition, all numbers end in $1$, the multiples of $4$ plus $1$ end in $01$ (as well as these numbers multiplied by a power of two), the multiples of $4$ plus $3$ in $11$, and so on.

We know that $x$ is a square if and only if $\ord(x)$ is even and $x$ ends in $001$, and $x$ is sum of two squares if and only if $x$ ends in $01$ (Corollary \ref{cor:image}). We have now four classes of $z$, for a fixed $j$:
\begin{enumerate}
	\item $j-z$ and $1-z^2$ end in $01$, $\ord(2(j-z)(1-z^2))$ is even and $2(j-z)(1-z^2)$ ends in $001$.
	\item $j-z$ and $1-z^2$ end in $01$, $\ord(2(j-z)(1-z^2))$ is even and $2(j-z)(1-z^2)$ ends in $101$.
	\item $j-z$ and $1-z^2$ end in $01$, and $\ord(2(j-z)(1-z^2))$ is odd. (Note that $2(j-z)(1-z^2)$ necessarily ends in $01$.)
	\item At least one of $j-z$ and $1-z^2$ ends in $11$.
\end{enumerate}
Again, $1$, $-1$ and $j$ are considered to be in the first class.

However, we know that $(x,y,z)\in \mathrm{S}^2_p$, and for $p=2$ this implies that $z$ has one of the forms
\[2u,\quad 5+16u,\quad 11+16u,\quad 1+2^m(3+4u),\quad -1+2^m(1+4u)\]
in addition to the special cases $1$ and $-1$ (basically, for all $z$ not included here $1-z^2$ ends in $11$ and they are automatically in the last class). Concretely, $z$ must be a $p$-adic integer.

The analog of Theorem \ref{thm:JC-z} for $p=2$ is the following:

\begin{theorem}\label{thm:JC-z2}
	Let $V_{j,h}$ be as in Definition \ref{def:UV}. The projection of $V_{j,h}$ to the coordinate $z$ consists of
	\begin{enumerate}
		\item all numbers in the first class with potential less than $2\ord(h)$,
		\item all numbers in the second class with potential $2\ord(h)$,
		\item some numbers in the first and second classes with potential $2\ord(h)+2$ (which points depends on the concrete value of $h$), and
		\item all numbers in the third class with potential $2\ord(h)+3$.
	\end{enumerate}
\end{theorem}

\begin{proof}
	If the potential of $z$ is less than $2\ord(h)$, subtracting $4h^2$ from $2(j-z)(1-z^2)$ does not change the order nor the first three digits from the right, so the result is a square if and only if $2(j-z)(1-z^2)$ is a square.
	
	If the potential is exactly $2\ord(h)$, subtracting $4h^2$ keeps the order and also the second digit, but inverts the third, so we must have $101$ before subtracting in order to reach a square at the end.
	
	If the potential is $2\ord(h)+1$, the result will have that order, and it will not be a square.
	
	If the potential is $2\ord(h)+2$, subtracting $4h^2$ will increase the order at least to $2\ord(h)+4$ (the first two digits are $01$ in both numbers). The new order and the new first three digits cannot be known in advance.
	
	If the potential is $2\ord(h)+3$, we need the initial value to end in $01$ so that subtracting $001$ at order $2\ord(h)+2$ results in $001$, but that is automatic.
	
	Finally, if the potential is greater than $2\ord(h)+3$, subtracting $4h^2$ cannot leave another square because the new first two digits will be $11$ instead of $01$.
\end{proof}

Calculating the image is significantly more complicated here, because of the compactness of $\mathrm{S}^2_p$, that forbids taking $z$ ``big enough'' as we did with $p\equiv 3\mod 4$. Actually, $F$ is not surjective, as the following necessary condition shows.

\begin{proposition}\label{prop:JC-image-nec}
	Let $(j,h)$ be in the image $F(\mathrm{S}_p^2\times(\Qp)^2)$ of the Jaynes-Cummings model $F:\mathrm{S}_p^2\times(\Qp)^2\to(\Qp)^2$ given by \eqref{eq:JC}. The following statements hold.
	\begin{enumerate}
		\item If $\ord(j)\ge 1$, then $\ord(h)\ge 0$ (that is, $h$ is integer).
		\item If $\ord(j)=0$, then $\ord(h)\ge -1$ and $\ord(h)\ne 0$.
		\item If $\ord(j)<0$ is even, then $\ord(h)=\ord(j)/2-1$.
		\item If $\ord(j)<0$ is odd, then $\ord(h)\ge (\ord(j)-1)/2$.
	\end{enumerate}
\end{proposition}
\begin{proof}
	If $(j,h)$ is in the image, there is $z$ in the situation of Theorem \ref{thm:JC-z2}. The potential of $z$ must be at most $2\ord(h)+3$, so
	\begin{align*}
	2\ord(h)+3 & \ge \ord(2(j-z)(1-z^2)) \\
	& =1+\ord(j-z)+\ord(1-z^2).
	\end{align*}
	There are two important aspects in this inequality:
	\begin{itemize}
		\item If the right-hand side is odd, equality must be attained. This is because the only case with odd potential has it equal to $2\ord(h)+3$.
		\item If $\ord(z)=0$, $z^2$ ends in $001$ and $\ord(1-z^2)\ge 3$.
	\end{itemize}
	
	With this into account, now we distinguish the possible cases. If $\ord(j)\ge 1$ and $\ord(z)\ge 1$, $2\ord(h)+3\ge 1+1+1$ and $\ord(h)\ge 0$. If $\ord(z)=0$, $2\ord(h)+3\ge 1+0+3$ and $\ord(h)\ge 1$.
	
	If $\ord(j)=0$ and $\ord(z)\ge 1$, $\ord(j-z)=0$ and $\ord(1-z^2)=0$, so it must be equal, $2\ord(h)+3=1$, and $\ord(h)=-1$. If $\ord(z)=0$, $2\ord(h)+3\ge 1+0+3$ and $\ord(h)\ge 1$.
	
	If $\ord(j)<0$ is even, $\ord(j-z)=\ord(j)$ and $\ord(1-z^2)=0$, so it must be equal again, $2\ord(h)+3=1+\ord(j)$, and $\ord(h)=\ord(j)/2-1$.
	
	Finally, if $\ord(j)<0$ is odd, $2\ord(h)+3\ge 1+\ord(j)+0$ and $\ord(h)\ge (\ord(j)-2)/2$, which implies $\ord(h)\ge (\ord(j)-1)/2$.
\end{proof}

Another necessary condition that applies only to $\ord(j)\le -2$ is the following:

\begin{proposition}\label{prop:JC-image-01}
	Let $F:\mathrm{S}_p^2\times(\Qp)^2\to(\Qp)^2$ be the Jaynes-Cummings model given by \eqref{eq:JC}.
	If $(j,h)\in F(\mathrm{S}_p^2\times(\Qp)^2)$ and $\ord(j)\le -2$, then $j$ ends in $01$.
\end{proposition}
\begin{proof}
	This is an immediate consequence of the condition that $j-z$ ends in $01$ together with $z$ being integer.
\end{proof}

We can also formulate some sufficient conditions for a point to be in the image:

\begin{proposition}\label{prop:JC-image-suf}
	Let $(j,h)\in(\Qp)^2$ and let $F:\mathrm{S}_p^2\times(\Qp)^2\to(\Qp)^2$ be the Jaynes-Cummings model given by \eqref{eq:JC}. Then the following statements hold.
	\begin{enumerate}
		\item If $\ord(j)\ge 1$ and $\ord(h)\ge 0$, $(j,h)$ is in the image of $F$.
		\item If $\ord(j)\le 0$ is even, $j$ ends in $01$ and $\ord(h)=\ord(j)/2-1$, $(j,h)$ is in the image.
		\item If $\ord(j)\le -3$ is odd, $j$ ends in $01$ and $\ord(h)\ge (\ord(j)+1)/2$, $(j,h)$ is in the image.
	\end{enumerate}
\end{proposition}
\begin{proof}
	For part (1), we will prove that there is a $z$ in the third class with potential $2\ord(h)+3$. To do that, we take $z$ such that $\ord(j-z)=2\ord(h)+2$ and $j-z$ ends in $01$. This $z$ has positive order, so $1-z^2$ also ends in $01$ and the same happens with the product. We have also
	\begin{align*}
	\ord(2(j-z)(1-z^2)) & =1+2\ord(h)+2+0 \\
	& =2\ord(h)+3
	\end{align*}
	so this $z$ is in the third class and we are done.
	
	For part (2), we follow the same strategy, but now making $\ord(j-z)=\ord(j)$ and $z$ still with positive order. Then the potential is
	\[1+\ord(j)=2\ord(h)+3,\]
	as we want.
	
	For part (3), we will construct a $z$ in the first class if $\ord(j)<2\ord(h)-1$ and in the second class if $\ord(j)=2\ord(h)-1$. This $z$ will have positive order, so 
	\[\ord(2(j-z)(1-z^2))=1+\ord(j)\]
	which, in any of the two cases, is the correct potential for the selected class. In order to construct $z$, it is only left to force $2(j-z)(1-z^2)$ to end in $001$ if $z$ needs to be in the first class, and in $101$ for the second class. As $z$ has positive order, $j-z$ will end in the same three digits as $j$, that are $001$ or $101$. Now we take $z$ with order $1$ if the current ending is not the desired one (so that multiplying by $1-z^2$ inverts the third digit), and greater order otherwise (so that the ending stays as is).
\end{proof}

This settles the cases $\ord(j)\ge 1$, $\ord(j)$ even $\le -2$, and $\ord(j)$ odd $\le -3$ with \[\ord(h)\ne(\ord(j)-1)/2.\] Determining whether other cases are in the image seems more complicated.

We have only left to discuss the topological structure of the fibers. This is similar to the case $p\equiv 3\mod 4$: for fixed $j$ and $h$, there are at most three $z$ such that $b=0$, and they are in the first and fourth classes. Those in the first class are the ones interesting to us. Then, $U_{j,h}(z,b)$ forms a circle, except in the rank $0$ points where it is a single point.

\begin{lemma}\label{lemma:firstclass2}
	Let $p=2$. If $z_0\in\Q_2$ is a first class number (such that $2(z_0^2-1)(z_0-j)-4h^2$ is a square), any $z\in\Q_2$ sufficiently near $z_0$ such that $2(z^2-1)(z-j)-4h^2$ is a square is also in the first class, except if $j=\pm 1$ and $z_0=j$, in which case some numbers are in the first class and the rest in the fourth.
\end{lemma}
\begin{proof}
	The same as Lemma \ref{lemma:firstclass}, changing even order by $01$ ending.
\end{proof}

\begin{figure}
	\begin{tabular}{c|c}
		\fibra{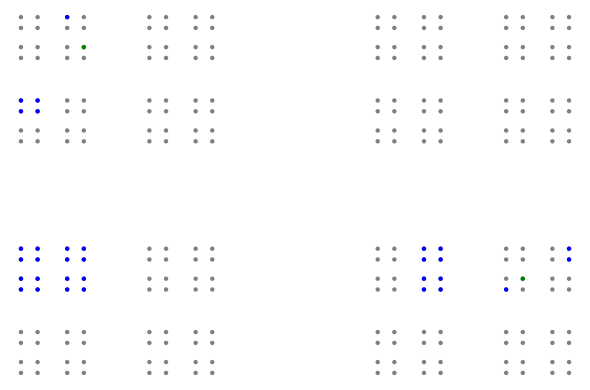} & \fibra{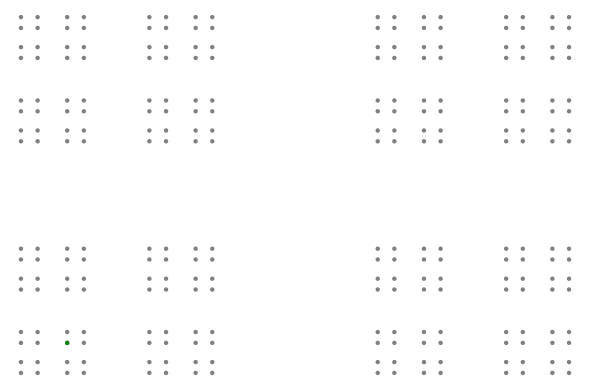} \\
		$(22,0)$ & $(-47/8,-15/4)$ \\ \hline
		\fibra{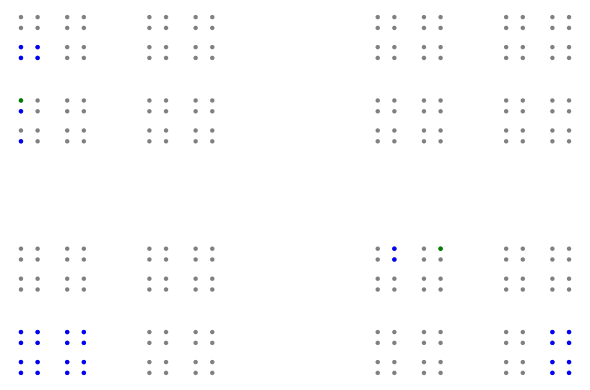} & \fibra{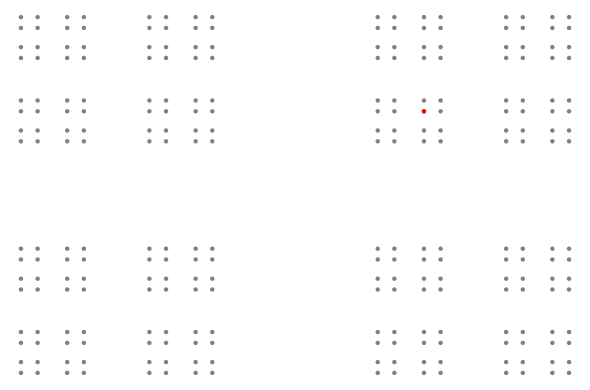} \\
		$(22,2)$ & $(-1,0)$ \\ \hline
		\fibra{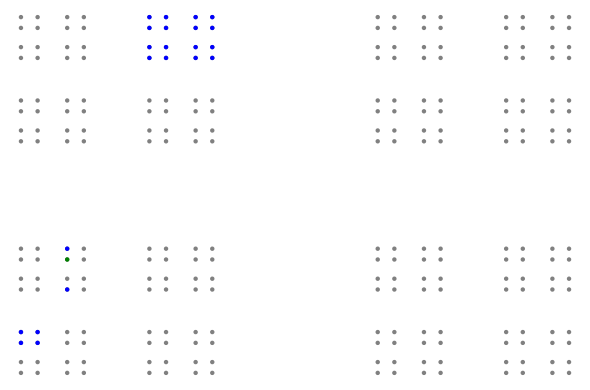} & \fibra{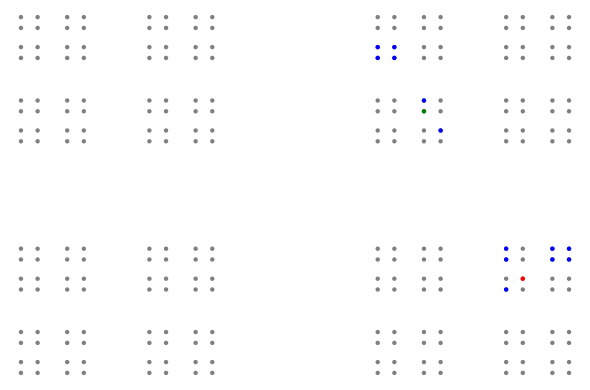} \\
		$(22,1)$ & $(1,0)$
	\end{tabular}
	\caption{Fibers of the Jaynes-Cummings model if $p=2$. The three on the left have rank $2$, the one on the top right has rank $1$, and the other two have rank $0$. The blue points represent values of $z$ for which $x,y,u,v$ form two circles, and at the green points they form only one circle (because $b=0$ at those points). The grey points are values of $z$ that are not in the fiber. In the two figures for rank $0$, there is a red point at $z=j$, which represents a single point in the fiber (the unique point in the fiber in the elliptic case).}
	\label{fig:fibers2}
\end{figure}

For an illustration of the following result see Figure \ref{fig:fibers2}.

\begin{theorem}\label{thm:JC-topology2}
	Let $p=2$. Let $F:\mathrm{S}_2^2\times(\Q_2)^2\to(\Q_2)^2$ be the $p$-adic analytic Jaynes-Cummings model given by \eqref{eq:JC}. The following statements hold.
	\begin{enumerate}
		\item If $(j,h)=(-1,0)$, then the fiber $F^{-1}(\{(j,h)\})$ consists of a single point at $(0,0,-1,0,0)$.
		\item If $(j,h)=(1,0)$, then the fiber $F^{-1}(\{(j,h)\})$ has dimension $2$ and a singularity at $(0,0,1,0,0)$.
		\item If $(j,h)$ is a rank $1$ critical value, then the fiber $F^{-1}(\{(j,h)\})$ is the disjoint union of a $2$-dimensional $p$-adic analytic submanifold, maybe empty, and the circle formed by the critical points for that value:
		\[\Big\{(au,av,-a^2,u,v)\,\,\Big|\,\, u,v\in\Qp,a^2(u^2+v^2)+a^4=1\Big\}.\]
		\item For the rest of values of $(j,h)\in F(\mathrm{S}_p^2\times(\Qp)^2)$, the fiber $F^{-1}(\{(j,h)\})$ is a $2$-dimensional $p$-adic analytic submanifold.
	\end{enumerate}
\end{theorem}

\begin{proof}
	The same as Theorem \ref{thm:JC-topology}. The changes are:
	
	\begin{itemize}
		\item If $p=2$ the number $-3a^4-1$ is never a square. This happens because $a^4$ ends in $0001$ and has order multiple of $4$, so $-3a^4$ ends in $1101$ and also has order multiple of $4$. No number with this form, after subtracting $1$, ends in $001$ (the ending stays as $1101$ if its order is negative, changes to $11$ if it is zero, or changes to $1111$ if it is positive).
		\item No point other than $(0,0,-1,0,0)$ can have $(-1,0)$ as image. This is because, if it was so,
		\[2(-1-z)(1-z^2)=2(1+z)(z^2-1)\]
		would be a square, so $2(z-1)$ is also a square. Hence, $z-1$ ends in $001$ and has odd order. But $-1-z$ should end in $01$, so $z+1$ ends in $11$. There is no number ending in $001$ and with odd order that ends in $11$ after adding $2$.
	\end{itemize}
\end{proof}

In the case of the rank $1$ critical values, Figure \ref{fig:fibers2} seems to imply that the fiber only consists of the critical points. We have not been able to deduce this from the formula, and it may well happen that the figure is not taking into consideration enough points to pick one in the fiber.

\section{Non-degeneracy and normal forms of the critical points of $F$}\label{sec:nondeg}

In this section we verify the non-degeneracy of the critical points of the $p$-adic Jaynes-Cummings model $F:\mathrm{S}_p^2\times(\Qp)^2\to (\Qp)^2$ given by \eqref{eq:JC}, and obtain a normal form near each critical point. Because we are dealing with $(\Qp)^2$-valued maps, the calculations have to be done, even if they are analogous to the case of $\R^2$-valued maps. Also, in the $\R^2$ case the calculations are mathematical folklore among experts and we did not see them explicitly written elsewhere so that we could cite them here (the statement is given in \cite[Proposition 2.1, paragraph 3]{PelVuN}), and since the example we are presenting here is foundational for the theory of $(\Qp)^2$-valued integrable systems, we want to present the calculations explicitly.

Also, it should be noted that Williamson's full classification \cite{Williamson} is not (yet) available in the $p$-adic case, so all conclusions concerning the example have to be done by hand. (In the real case there are some simplifications as in \cite[Section 5]{Williamson}).

The following definition is analogous to the one for the real case, \cite[Definition 3.1]{PRV}. In \cite{BolFom,DulPel,LeFPel} there are criteria for deciding whether a point is non-degenerate, for rank $0$ and $1$ points.

\begin{definition}
	Let $(M,\omega)$ be a $p$-adic analytic symplectic four-manifold. Let $F=(f_1,f_2)$ be an integrable system on $(M,\omega)$ and let $m\in M$ be a critical point of $F$. If $\dd F(m)=0$, the $m$ is called \textit{non-degenerate} if the Hessians $\dd^2 f_j(m)$ span a Cartan subalgebra of the symplectic Lie algebra of quadratic forms on the tangent space $(\mathrm{T}_mM,\omega_M)$. If $\rank(\dd F(m))=1$ one can assume that $\dd f_1(m)\ne 0$. Let $\imath:S\to M$ be an embedded local $2$-dimensional symplectic submanifold through $m$ such that $\mathrm{T}_m S\subset\ker(\dd f_1(m))$ and $\mathrm{T}_m S$ is transversal to the Hamiltonian vector field $X_{f_1}$ defined by the function $f_1$. The critical point $m$ of $F$ is called \textit{(transversally) non-degenerate} if $\dd^2(\imath^*f_2)(m)$ is a non-degenerate symmetric bilinear form on $\mathrm{T}_m S$.
\end{definition}

First we treat the case of rank $1$ points.

\begin{proposition}\label{prop:nf-rank1}
	\letpprime. The critical points in the preimage under the Jaynes-Cummings model $F:\mathrm{S}_p^2\times(\Qp)^2\to(\Qp)^2$ given by \eqref{eq:JC} of the curves:
	\[\left\{\left(\frac{1-3a^4}{2a^2},\frac{1-a^4}{2a}\right)\,\,\middle|\,\, a\in\Qp\right\}\]
	are non-degenerate of rank $1$, and of ``elliptic-regular'' type: for a point $q=(au,av,-a^2,u,v)$, where $u,v\in\Qp$ with $a^2(u^2+v^2)+a^4=1$, consider the map
	\[\phi:\mathrm{T}_{(0,0,0,0)}(\Qp)^4\to \mathrm{T}_q(\mathrm{S}^2\times(\Qp)^2)\]
	given by
	\[\phi(x,\xi,y,\eta)=CD\begin{pmatrix}
		x \\
		\xi \\
		y \\
		\eta
	\end{pmatrix}\]
	where
	\[C=\begin{pmatrix}
		v & au & av & u \\
		-u & av & -au & v \\
		-av & -u & v & au \\
		au & -v & -u & av
	\end{pmatrix},
	D=\begin{pmatrix}
		1 & 0 & 0 & 0 \\
		0 & 1 & 0 & \frac{2a^6-a^4-1}{a(3a^4+1)} \\
		\frac{a(a^2-1)^2}{3a^4+1} & 0 & 1 & 0 \\
		0 & 0 & 0 & 1
	\end{pmatrix}.\]
	The map $\phi$ is a linear symplectomorphism, i.e. an automorphism such that $\phi^*\Omega=\omega_1$, where
	\[\omega_1=\frac{(1-a^4)(a^2+1)}{a^3}(\dd x\wedge\dd\xi+\dd y\wedge\dd\eta)\]
	and $\Omega$ is the Jaynes-Cummings symplectic linear form at $q$. In addition, $\phi$ satisfies the equation
	\[\tilde{F}\circ\phi=\left(\eta+\ocal(\eta^2),\alpha(a)x^2+\beta(a)\xi^2+\gamma(a)\eta^2+\ocal((x,\xi,\eta)^3)\right)\]
	where
	\[\left\{
	\begin{aligned}
	\alpha(a) & =\frac{a^2(a^2+1)^2(3a^4+1)}{2}; \\
	\beta(a) & =\frac{(3a^4+1)^2}{2}; \\
	\gamma(a) & =\frac{a^2(1-a^4)(a^2+1)^2}{2};
	\end{aligned}
	\right.\]
	and $\tilde{F}=B\circ(F-F(q))$ with
	\[B=\begin{pmatrix}
		a & 0 \\
		\frac{a^6(3a^4+1)}{1-a^4} & \frac{-2a^5(3a^4+1)}{1-a^4}
	\end{pmatrix}.\]
	In particular, if $\varphi:(\Qp)^2\to(\Qp)^2$ is given by
	\[\varphi(s,t)=\left(s,\frac{a^2(1-a^4)(a^2+1)^2s^2+(3a^4+1)t}{2}\right)\]
	then
	\[\tilde{F}\circ\phi=\varphi\Big(\eta+\ocal(\eta^2),a^2(a^2+1)^2x^2+(3a^4+1)\xi^2+\ocal((x,\xi)^3)\Big).\]
\end{proposition}

\begin{proof}
	The point $q$ is of rank $1$ by Theorem \ref{thm:JC-general}. First, we want $\tilde{F}_2$ to be quadratic near $q$. This means that $\dd \tilde{F}_2(q)=0$. Let $\tilde{F}_2=\lambda J+\mu H$.
	\begin{align*}
		0=\dd \tilde{F}_2(q) & =\lambda\dd J(q)+\mu\dd H(q) \\
		& =\lambda\left(u\dd u+v\dd v-\frac{x}{z}\dd x-\frac{y}{z}\dd y\right)+\mu\frac{u\dd x+v\dd y+x\dd u+y\dd v}{2} \\
		& =\lambda\left(u\dd u+v\dd v+\frac{u}{a}\dd x+\frac{v}{a}\dd y\right)+\mu\frac{u\dd x+v\dd y+au\dd u+av\dd v}{2} 
	\end{align*}
	which is true for $(\lambda,\mu)$ proportional to $(a,-2)$. We will take for the moment this combination, and the proportionality constant will be determined later.
	
	In the coordinates $(x,y,u,v)$ we have:
	\[\dd^2 J=\begin{pmatrix}
		-\frac{1}{z}-\frac{x^2}{z^3} & -\frac{xy}{z^3} & 0 & 0 \\
		-\frac{xy}{z^3} & -\frac{1}{z}-\frac{y^2}{z^3} & 0 & 0 \\
		0 & 0 & 1 & 0 \\
		0 & 0 & 0 & 1
	\end{pmatrix}
	=\begin{pmatrix}
		\frac{1}{a^2}+\frac{u^2}{a^4} & \frac{uv}{a^4} & 0 & 0 \\
		\frac{uv}{a^4} & \frac{1}{a^2}+\frac{v^2}{a^4} & 0 & 0 \\
		0 & 0 & 1 & 0 \\
		0 & 0 & 0 & 1
	\end{pmatrix},\]
	\[\dd^2 H=\frac{1}{2}\begin{pmatrix}
		0 & 0 & 1 & 0 \\
		0 & 0 & 0 & 1 \\
		1 & 0 & 0 & 0 \\
		0 & 1 & 0 & 0
	\end{pmatrix}.\]
	The combination
	\[a\dd^2 J-2\dd^2 H\]
	gives
	\[A=\begin{pmatrix}
		\frac{1}{a}+\frac{u^2}{a^3} & \frac{uv}{a^3} & -1 & 0 \\
		\frac{uv}{a^3} & \frac{1}{a}+\frac{v^2}{a^3} & 0 & -1 \\
		-1 & 0 & a & 0 \\
		0 & -1 & 0 & a
	\end{pmatrix}.\]
	
	Choosing the basis given by the columns of $C$ and changing, $A$ becomes
	\[A_1=C^TAC
	=\frac{1-a^4}{a^7}\begin{pmatrix}
		a^4(a^2+1)^2 & 0 & 0 & 0 \\
		0 & a^2(3a^4+1) & 0 & -2a^7+a^5+a \\
		0 & 0 & 0 & 0 \\
		0 & -2a^7+a^5+a & 0 & a^8-2a^6+1
	\end{pmatrix}.\]
	
	Let $(x_1,\xi_1,y_1,\eta_1)$ be the coordinates after this change. Note that $y_1$ does not appear in $A_1$, which is expected because
	\[\frac{\partial}{\partial y_1}=av\frac{\partial}{\partial x}-au\frac{\partial}{\partial y}+v\frac{\partial}{\partial u}-u\frac{\partial}{\partial v}=X_J\]
	and both Hessians of $J$ and $H$ should vanish at $X_J$, which is proportional to $X_H$.
	
	Now we have that
	\[\dd J(q)=\frac{u}{a}\dd x+\frac{v}{a}\dd y+u\dd u+v\dd v=\frac{1}{a}\dd\eta_1(q),\]
	so the linear term in $aJ$ is $\eta_1$, and
	\[aJ-2H=\frac{1}{2}\begin{pmatrix}
		x_1 & \xi_1 & y_1 & \eta_1
	\end{pmatrix}
	A_1
	\begin{pmatrix}
		x_1 \\ \xi_1 \\ y_1 \\ \eta_1
	\end{pmatrix}+\ocal(3)\]
	\[=\frac{1-a^4}{2a^7}[a^4(a^2+1)^2x_1^2+a^2(3a^4+1)\xi_1^2+(-4a^7+2a^5+a)\xi_1\eta_1+(a^8-2a^6+1)\eta_1^2]+\ocal(3)\]
	and
	\begin{align*}
		\Omega & =\frac{1}{a^2}\dd x\wedge\dd y+\dd u\wedge\dd v \\
		& =\frac{(1-a^4)(a^2+1)}{a^3}\left[\dd x_1\wedge\dd\xi_1+\dd y_1\wedge\dd\eta_1+\left(\frac{1}{a}-a\right)\dd x_1\wedge\dd\eta_1\right]
	\end{align*}
	
	We want to remove the terms in $\xi_1\eta_1$ and $\dd x_1\wedge\dd\eta_1$ in these expressions. In order to do this, we first change the coordinates $\xi$ and $\eta$:
	\[(\xi,\eta)=\left(\xi_1+\frac{-2a^6+a^4+1}{a(3a^4+1)}\eta_1,\eta_1\right)\]
	The result is that
	\[(aJ,aJ-2H)=\]
	\[\left(\eta+\ocal(2),\frac{1-a^4}{2a^5(3a^4+1)}[a^2(a^2+1)^2(3a^4+1)x_1^2+(3a^4+1)^2\xi^2+a^2(1-a^4)(a^2+1)^2\eta^2]+\ocal(3)\right)\]
	which is almost the expression we want. Now we change $x$ and $y$:
	\[(x,y)=\left(x_1,y_1-\frac{a(a^2-1)^2}{3a^4+1}x_1\right)\]
	This does not affect $(aJ,aJ-2H)$, because they do not use $y$, so taking the matrix $B$ as above, we have the desired result for $\tilde{F}$.
	
	Making the changes in $\Omega$,
	\begin{align*}
		\Omega & =\frac{(1-a^4)(a^2+1)}{a^3}\left[\dd x_1\wedge\dd\xi_1+\dd y_1\wedge\dd\eta_1+\left(\frac{1}{a}-a\right)\dd x_1\wedge\dd\eta_1\right] \\
		& =\frac{(1-a^4)(a^2+1)}{a^3}\left[\dd x_1\wedge\dd\xi_1+\dd y_1\wedge\dd\eta_1+\left(\frac{-2a^6+a^4+1}{a(3a^4+1)}-\frac{a(a^2-1)^2}{3a^4+1}\right)\dd x_1\wedge\dd\eta_1\right] \\
		& =\frac{(1-a^4)(a^2+1)}{a^3}\left[\dd x_1\wedge\left(\dd\xi_1+\frac{-2a^6+a^4+1}{a(3a^4+1)}\dd\eta_1\right)+\left(\dd y_1-\frac{a(a^2-1)^2}{3a^4+1}\dd x_1\right)\wedge\dd\eta_1\right] \\
		&=\frac{(1-a^4)(a^2+1)}{a^3}(\dd x\wedge\dd\xi+\dd y\wedge\dd\eta)
	\end{align*}
	and we are done.
\end{proof}

The following result is a direct consequence of the proof of Proposition \ref{prop:nf-rank1}, which can be simplified in the case of the real Jaynes-Cummings model.

\begin{corollary}\label{cor:JC-simplified}
	In the case of the real Jaynes-Cummings model $F:\mathrm{S}^2\times\R^2\to\R^2$ given in Proposition \ref{prop:JC-real}, we have that:
	\begin{equation}\label{eq:JC-simplified}
		\tilde{F}\circ\phi=\left(\eta+\ocal(\eta^2),\frac{x^2+\xi^2+\eta^2}{2}+\ocal((x,\xi)^3)\right)
	\end{equation}
	where $\tilde{F}=B\circ(F-F(q))$ with
	\[B=\begin{pmatrix}
		\frac{(1-a^4)(a^2+1)}{(3a^4+1)^{3/4}} & 0 \\
		\frac{a^2}{\sqrt{3a^4+1}} & \frac{-2a}{\sqrt{3a^4+1}}
	\end{pmatrix}\]
	and $(x,\xi,y,\eta)$ are local coordinates around $q$ such that \[\omega=\dd x\wedge\dd\xi+\dd y\wedge\dd\eta.\] In particular, if
	\[\varphi(s,t)=\left(s,\frac{s^2+t}{2}\right)\]
	then
	\[\tilde{F}\circ\phi=\varphi\Big(\eta+\ocal(\eta^2),x^2+\xi^2+\ocal(x,\xi)^3\Big).\]
\end{corollary}

\begin{proof}
	If we apply the proof of Proposition \ref{prop:nf-rank1} word by word, we end up with the expressions
	\[\tilde{F}\circ\phi=\left(\eta+\ocal(\eta^2),\frac{C_1x^2+C_2\xi^2+C_3\eta^2}{2}+\ocal((x,\xi)^3)\right)\]
	and
	\[\omega=C_4(\dd x\wedge\dd\xi+\dd y\wedge\dd\eta),\]
	where the constants $C_i$ depend on $a$. At the critical points $1-a^4$ must be sum of two squares, which in the real case means $-1<a<1$. (The endpoints of this interval correspond to a rank $0$ point.) This implies that $C_i>0$ for all $i$.
	Now we can make further simplifications:
	\begin{enumerate}
		\item Multiply the coordinates $x$ and $\xi$ by $\sqrt{C_4}$, so that
		\[\tilde{F}\circ\phi=\left(\eta+\ocal(\eta^2),\frac{C_1x^2+C_2\xi^2}{2C_4}+\frac{C_3\eta^2}{2}+\ocal((x,\xi)^3)\right)\]
		and
		\[\omega=\dd x\wedge\dd\xi+C_4\dd y\wedge\dd \eta.\]
		\item Multiply $x$ by $\sqrt[4]{C_1/C_2}$ and divide $\xi$ by the same amount. This does not alter $\omega$ and makes
		\[\tilde{F}\circ\phi=\left(\eta+\ocal(\eta^2),\frac{\sqrt{C_1C_2}(x^2+\xi^2)}{2C_4}+\frac{C_3\eta^2}{2}+\ocal((x,\xi)^3)\right).\]
		\item Multiply $\eta$ by $\sqrt{C_3C_4}/\sqrt[4]{C_1C_2}$ (in what follows we call this factor $C_5$). After this step,
		\[\tilde{F}\circ\phi=\left(\frac{\eta}{C_5}+\ocal(\eta^2),\frac{\sqrt{C_1C_2}(x^2+\xi^2+\eta^2)}{2C_4}+\ocal((x,\xi)^3)\right)\]
		and
		\[\omega=\dd x\wedge\dd\xi+\frac{C_4}{C_5}\dd y\wedge\dd \eta.\]
		\item Multiply $y$ by $C_4/C_5$, so that $\omega$ takes its final form and $\tilde{F}$ is not altered.
		\item Multiply the first row of $B$ by $C_5$ and the second by $C_4/\sqrt{C_1C_2}$. The simplification is complete.
	\end{enumerate}
\end{proof}

\begin{remark}
	In the $p$-adic case the simplifications in the proof of Corollary \ref{cor:JC-simplified} are in general not possible: we need to take roots of the constants $C_i$, which is only possible for some values of $a$ and $p$.
\end{remark}

\begin{remark}
	We do not know how/if some form of Eliasson and Vey's Theorem \cite{Eliasson,Eliasson-thesis,Russmann,Vey,VuNWac} holds in the $p$-adic case (the analytic case of this theorem is due to R\"u\ss mann \cite{Russmann} for two degrees of freedom and Vey \cite{Vey} in arbitrary dimension). In the real case Eliasson's Theorem (assuming that there are no hyperbolic components) says that there is a local diffeomorphism $\varphi$ and symplectic coordinates $\phi^{-1}=(x,\xi,y,\eta)$ such that $F\circ\phi=\varphi(q_1,q_2)$, where $q_i$ is one of the elliptic, real or focus-focus models. In the real-elliptic case, if we derive this expression twice we obtain the linear statement we have proved, in the simplified form (\ref{eq:JC-simplified}). The term $\eta^2$ is not in the elliptic model $q_2=(x^2+\xi^2)/2$, but it appears when deriving (it comes from a second derivative of $\varphi$).
\end{remark}

Now the rank $0$ points. We know that there are two: $(0,0,1,0,0)$ (whose image is $(1,0)$) and $(0,0,-1,0,0)$ (whose image is $(-1,0)$). Both have a singularity, however, the two singularities have different types.

\begin{proposition}\label{prop:nf-elliptic}
	\letpprime. Let $F:\mathrm{S}_p^2\times(\Qp)^2\to(\Qp)^2$ be the $p$-adic Jaynes-Cummings model given by \eqref{eq:JC}. The point $q=(0,0,-1,0,0)\in\mathrm{S}_p^2\times(\Qp)^2$ is non-degenerate of rank $0$, and of ``elliptic-elliptic'' type:
	consider the map
	\[\phi:\mathrm{T}_{(0,0,0,0)}(\Qp)^4\to \mathrm{T}_q(\mathrm{S}_p^2\times(\Qp)^2)\]
	given by
	\[\phi(x,\xi,y,\eta)=\frac{1}{2}(x+y,\xi+\eta,x-y,\xi-\eta)\]
	The map $\phi$ is a linear symplectomorphism, i.e. an automorphism such that $\phi^*\Omega=\omega_1$, where
	\[\omega_1=\frac{1}{2}(\dd x\wedge\dd\xi+\dd y\wedge\dd\eta)\]
	and $\Omega$ is the Jaynes-Cummings symplectic linear form at $q$. In addition, $\phi$ satisfies the equation
	\[\tilde{F}\circ\phi=\frac{1}{2}(x^2+\xi^2,y^2+\eta^2)+\ocal((x,\xi,y,\eta)^3)\]
	where $\tilde{F}=B\circ(F-F(q))$ with
	\[B=\begin{pmatrix}
		1 & 2 \\
		1 & -2
	\end{pmatrix}.\]
\end{proposition}

\begin{proof}
	The point has rank $0$ because $\dd J=\dd H=0$. The Hessians are
	\[\dd^2 J=\begin{pmatrix}
		1 & 0 & 0 & 0 \\
		0 & 1 & 0 & 0 \\
		0 & 0 & 1 & 0 \\
		0 & 0 & 0 & 1
	\end{pmatrix},
	\dd^2 H=\frac{1}{2}\begin{pmatrix}
		0 & 0 & 1 & 0 \\
		0 & 0 & 0 & 1 \\
		1 & 0 & 0 & 0 \\
		0 & 1 & 0 & 0
	\end{pmatrix}\]
	To see that the point is non-degenerate, we take a linear combination of these matrices and multiply it by $\omega_q^{-1}$:
	\[\omega_q^{-1}(\lambda\dd^2 J+\mu\dd^2 H)=\begin{pmatrix}
		0 & 1 & 0 & 0 \\
		-1 & 0 & 0 & 0 \\
		0 & 0 & 0 & 1 \\
		0 & 0 & -1 & 0
	\end{pmatrix}^{-1}
	\begin{pmatrix}
		\lambda & 0 & \mu & 0 \\
		0 & \lambda & 0 & \mu \\
		\mu & 0 & \lambda & 0 \\
		0 & \mu & 0 & \lambda
	\end{pmatrix}=
	\begin{pmatrix}
		0 & -\lambda & 0 & -\mu \\
		\lambda & 0 & \mu & 0 \\
		0 & -\mu & 0 & -\lambda \\
		\mu & 0 & \lambda & 0
	\end{pmatrix}\]
	The characteristic polynomial of this matrix is
	\begin{align*}
		t^4+2\lambda^2t^2+2\mu^2t^2+\lambda^4+\mu^4-2\lambda^2\mu^2 & =(t^2+\lambda^2+\mu^2)^2-4\lambda^2\mu^2 \\
		& =(t^2+\lambda^2+\mu^2+2\lambda\mu)(t^2+\lambda^2+\mu^2-2\lambda\mu) \\
		& =[t^2+(\lambda+\mu)^2][t^2+(\lambda-\mu)^2]
	\end{align*}
	that has, in general, four different roots (in $\Qp$ if $p\equiv 1\mod 4$, or in a degree $2$ extension otherwise). Hence, the point is non-degenerate.
	
	To show the local expression for $F$, we have that
	\[\dd^2(J+2H)=\begin{pmatrix}
		1 & 0 & 1 & 0 \\
		0 & 1 & 0 & 1 \\
		1 & 0 & 1 & 0 \\
		0 & 1 & 0 & 1 
	\end{pmatrix}=C^TA_1C\]
	and
	\[\dd^2(J-2H)=\begin{pmatrix}
		1 & 0 & -1 & 0 \\
		0 & 1 & 0 & -1 \\
		-1 & 0 & 1 & 0 \\
		0 & -1 & 0 & 1
	\end{pmatrix}=C^TA_2C,\]
	where
	\[A_1=\begin{pmatrix}
		1 & 0 & 0 & 0 \\
		0 & 1 & 0 & 0 \\
		0 & 0 & 0 & 0 \\
		0 & 0 & 0 & 0
	\end{pmatrix},
	A_2=\begin{pmatrix}
		0 & 0 & 0 & 0 \\
		0 & 0 & 0 & 0 \\
		0 & 0 & 1 & 0 \\
		0 & 0 & 0 & 1
	\end{pmatrix},
	C=\begin{pmatrix}
		1 & 0 & 1 & 0 \\
		0 & 1 & 0 & 1 \\
		1 & 0 & -1 & 0 \\
		0 & 1 & 0 & -1
	\end{pmatrix}.\]
	In the basis formed by the columns of $C^{-1}$, the Hessians of $J+2H$ and $J-2H$ have the forms of the matrices $A_1$ and $A_2$, which gives the formula we want. We also have that
	\[\omega_q=C^T\frac{1}{2}\begin{pmatrix}
		0 & 1 & 0 & 0 \\
		-1 & 0 & 0 & 0 \\
		0 & 0 & 0 & 1 \\
		0 & 0 & -1 & 0
	\end{pmatrix}C\]
	hence $\omega_q$ in the new coordinates is $(\dd x\wedge\dd\xi+\dd y\wedge\dd\eta)/2$.
\end{proof}

\begin{remark}\label{rem:elliptic}
	Identifying a Hessian with its quadratic form, the result in Proposition \ref{prop:nf-elliptic} can be written
	\[\dd^2\tilde{F}\circ\phi=\frac{1}{2}(x^2+\xi^2,y^2+\eta^2).\]
	
	In the real case the $1/2$ can be eliminated from the expression of $\omega$ for the rank $0$ points in a similar way to Corollary \ref{cor:JC-simplified}: first divide all coordinates by $\sqrt{2}$ and then adjust $B$ to recover the form of $\tilde{F}$. In the $p$-adic case, this can be done only if $2$ is a square modulo $p$ (that is, if $p\equiv \pm 1\mod 8$).
\end{remark}

\begin{proposition}\label{prop:nf-focus}
	\letpprime. Let $F:\mathrm{S}_p^2\times(\Qp)^2\to(\Qp)^2$ be the $p$-adic Jaynes-Cummings model given by \eqref{eq:JC}. The point $q=(0,0,1,0,0)\in\mathrm{S}_p^2\times(\Qp)^2$ is non-degenerate of rank $0$, and of ``focus-focus'' type:
	consider the map
	\[\phi:\mathrm{T}_{(0,0,0,0)}(\Qp)^4\to \mathrm{T}_q(\mathrm{S}^2\times(\Qp)^2)\]
	given by
	\[\phi(x,\xi,y,\eta)=\frac{1}{2}(\eta-x,y+\xi,y-\xi,\eta+x)\]
	The map $\phi$ is a linear symplectomorphism, i.e. an automorphism such that $\phi^*\Omega=\omega_1$, where
	\[\omega=\frac{1}{2}(\dd x\wedge\dd\xi+\dd y\wedge\dd\eta)\]
	and $\Omega$ is the Jaynes-Cummings symplectic linear form at $q$. In addition, $\phi$ satisfies the equation
	\[\tilde{F}\circ\phi=(x\eta-y\xi,x\xi+y\eta)+\ocal((x,\xi,y,\eta)^3)\]
	where $\tilde{F}=B\circ(F-F(q))$ with
	\[B=\begin{pmatrix}
		2 & 0 \\
		0 & 4
	\end{pmatrix}.\]
\end{proposition}

\begin{proof}
	Again, the rank is $0$ because at this point $\dd J=\dd H=0$. Now the Hessians are
	\[\dd^2 J=\begin{pmatrix}
		-1 & 0 & 0 & 0 \\
		0 & -1 & 0 & 0 \\
		0 & 0 & 1 & 0 \\
		0 & 0 & 0 & 1
	\end{pmatrix},
	\dd^2 H=\frac{1}{2}\begin{pmatrix}
		0 & 0 & 1 & 0 \\
		0 & 0 & 0 & 1 \\
		1 & 0 & 0 & 0 \\
		0 & 1 & 0 & 0
	\end{pmatrix}\]
	Taking a linear combination and multiplying by the inverse of the matrix of the symplectic form,
	\[\omega_q^{-1}(\lambda\dd^2 J+\mu\dd^2 H)=\begin{pmatrix}
		0 & -1 & 0 & 0 \\
		1 & 0 & 0 & 0 \\
		0 & 0 & 0 & 1 \\
		0 & 0 & -1 & 0
	\end{pmatrix}^{-1}
	\begin{pmatrix}
		-\lambda & 0 & \mu & 0 \\
		0 & -\lambda & 0 & \mu \\
		\mu & 0 & \lambda & 0 \\
		0 & \mu & 0 & \lambda
	\end{pmatrix}=
	\begin{pmatrix}
		0 & -\lambda & 0 & \mu \\
		\lambda & 0 & -\mu & 0 \\
		0 & -\mu & 0 & -\lambda \\
		\mu & 0 & \lambda & 0
	\end{pmatrix}.\]
	The characteristic polynomial is now
	\begin{align*}
		t^4+2\lambda^2t^2-2\mu^2t^2+\lambda^4+\mu^4+2\lambda^2\mu^2 & =(t^2+\lambda^2+\mu^2)^2-4\mu^2t^2 \\
		& =(t^2+\lambda^2+\mu^2+2\mu t)(t^2+\lambda^2+\mu^2-2\mu t) \\
		& =[(t+\mu)^2+\lambda^2][(t-\mu)^2+\lambda^2]
	\end{align*}
	that has again four different roots (maybe not in $\Qp$, but in a degree $2$ extension). Hence, the point is non-degenerate. It is only left to find the matrix $C$ which makes the following equalities hold:
	\[2\dd^2 J=C^T\begin{pmatrix}
		0 & 0 & 0 & 1 \\
		0 & 0 & -1 & 0 \\
		0 & -1 & 0 & 0 \\
		1 & 0 & 0 & 0
	\end{pmatrix}C\]
	and
	\[4\dd^2 H=C^T\begin{pmatrix}
		0 & 1 & 0 & 0 \\
		1 & 0 & 0 & 0 \\
		0 & 0 & 0 & 1 \\
		0 & 0 & 1 & 0
	\end{pmatrix}C.\]
	The matrix we are looking for is
	\[C=\begin{pmatrix}
		-1 & 0 & 0 & 1 \\
		0 & 1 & -1 & 0 \\
		0 & 1 & 1 & 0 \\
		1 & 0 & 0 & 1
	\end{pmatrix}\]
	and we can see that $\omega_q$ in the new basis becomes $(\dd x\wedge\dd\xi+\dd y\wedge\dd\eta)/2$.
\end{proof}

\begin{remark}
	Again, identifying in Proposition \ref{prop:nf-focus} a Hessian with its quadratic form,
	\[\dd^2\tilde{F}\circ\phi=(x\eta-y\xi,x\xi+y\eta).\]
	In the real case the $1/2$ can be eliminated from $\omega$, as in Remark \ref{rem:elliptic}.
\end{remark}

\appendix
\section{Basic $p$-adic theory}\label{app:prelim}

In this appendix we review the basic $p$-adic theory we need in the main part of the paper and derive some results (such as Proposition \ref{prop:initial}), the statements of which we did not find explicitly written elsewhere and which we need in the main part of the paper.

\subsection{Properties of the $p$-adic numbers}

The field of real numbers $\R$ is defined as a completion of $\Q$ with respect to the normal absolute value on $\Q$. Analogously, the field of $p$-adic numbers $\Qp$ can be defined as a completion of $\Q$ with respect to a non-archimedean absolute value. \textit{Throughout this section we fix a prime number $p\in\Z$.}

Following \cite[Definitions 2.1.2 and 2.1.4]{Gouvea}, the \textit{$p$-adic valuation} on $\Z$ is the function
\[\ord_p:\Z\setminus\{0\}\to\Z\]
defined as follows: for each integer $n\in\Z$, $n\ne 0$, let $\ord_p(n)$ be the unique positive integer satisfying
\[n=p^{\ord_p(n)}n',\quad\text{with } p\not|n'.\]
We extend $\ord_p$ to the field of rational numbers as follows: if $x=a/b\in\Q\setminus\{0\}$, then
\[\ord_p(x)=\ord_p(a)-\ord_p(b).\]
Also, for any $x\in\Q$, we define the \textit{$p$-adic absolute value} of $x$ by
\[|x|_p=p^{-\ord_p(x)}\]
if $x\ne 0$, and we set $|0|_p=0$.

One can check that $|\cdot|_p$ is a non-archimedean absolute value:
\begin{itemize}
	\item $|x|_p>0$ for all $x\ne 0$,
	\item $|x+y|_p\le\max\{|x|_p,|y|_p\}$ for all $x,y\in\Q$,
	\item $|xy|_p=|x|_p\,|y|_p$ for all $x,y\in\Q$.
\end{itemize}

\begin{theorem}[{\cite[Theorem 3.2.13]{Gouvea}}]\label{thm:p-adics}
	There exists a field $\Qp$ with a non-archimedean absolute value $|\cdot|_p$, such that the following statements hold.
	\begin{enumerate}
		\item There exists an inclusion $\Q\hookrightarrow\Qp$, and the absolute value induced by $|\cdot|_p$ on $\Q$ via this inclusion is the $p$-adic absolute value.
		\item The image of $\Q$ under this inclusion is dense in $\Qp$ with respect to the absolute value $|\cdot|_p$.
		\item $\Qp$ is complete with respect to the absolute value $|\cdot|_p$.
	\end{enumerate}
	The field $\Qp$ satisfying (1), (2) and (3) is unique up to isomorphism of fields preserving the absolute values.
\end{theorem}

Following \cite[Definition 3.3.3]{Gouvea},
the ring of $p$-adic integers $\Zp$ is defined by:
\[\Zp=\{x\in\Qp\mid|x|_p\le 1\}.\]

\begin{proposition}[{\cite[Proposition 3.3.4]{Gouvea}}]\label{prop:exp}
	For any $x\in\Zp$, there exists a Cauchy sequence $\alpha_n$ converging to $x$, of the following type:
	\begin{itemize}
		\item $\alpha_n\in\Z$ satisfies $0\le\alpha_n\le p^n-1$;
		\item for every $n$ we have $\alpha_n\equiv\alpha_{n-1}\mod p^{n-1}$.
	\end{itemize}
	The sequence $(\alpha_n)$ with these properties is unique.
\end{proposition}

Proposition \ref{prop:exp} implies that any $p$-adic number $a$ can be written uniquely as
$a=\sum_{n=n_0}^{\infty} a_np^n$
where $0\le a_n\le p-1$ and $a_{n_0}>0$, which is called \textit{$p$-adic expansion of $a$}. We have that the absolute value defined in Theorem \ref{thm:p-adics}, $|a|_p$, coincides with $p^{-n_0}$. This motivates to define $\ord_p(a):=n_0$ and call it \textit{order} of $a$.

In the following we will drop the subindex $p$ in $\ord_p$ and $|\cdot|_p$.

The topology of the $p$-adic field is very different from the reals, despite both being completions of the rationals with different metrics.
\begin{theorem}[{\cite[Corollaries 3.3.6 and 3.3.7]{Gouvea}}]\label{thm:topology}
	The following statements hold.
	\begin{itemize}
		\item The $p$-adic metric on $\Qp$ given by $d_p(a,b)=|a-b|$ satisfies the inequality
		$d_p(a,c)\le\max\{d_p(a,b),d_p(b,c)\}.$
		This makes $\Qp$ an \emph{ultrametric space}.
		\item $\Qp$ is totally disconnected, that is, all sets with more than one element are disconnected.
		\item The balls in the ultrametric space $\Qp$ are given by
		\[\mathrm{B}_\epsilon(x_0)=\{x\in\Qp\mid|x-x_0|\le \epsilon\}.\]
		Replacing $x_0$ by any other point in the ball does not change the ball.
		\item All balls are compact and open (in particular, $\Zp$ is compact and open). $\Qp$ is locally compact.
	\end{itemize}
\end{theorem}

\begin{corollary}\label{cor:disjoint}
	An open subset of $\Qp$ is a disjoint union of balls.
\end{corollary}
\begin{proof}
	This is a consequence of the previous theorem and \cite[Lemma 1.4]{Schneider}.
\end{proof}

The following is an important theorem in $p$-adic algebra. For a polynomial $f(x)$, we denote by $f'(x)$ the formal derivative of $f(x)$.

\begin{theorem}[{Hensel's lifting, \cite[Theorem 3.4.1 and Problem 112]{Gouvea}}]\label{thm:hensel}
	Let $f$ be a polynomial in $\Zp[x]$. Let $\alpha_1$ be a $p$-adic integer, $r=\ord(f(\alpha_1))$ and $s=\ord(f'(\alpha_1))$. If $r>2s$, there exists $\alpha\in\Zp$ such that $\ord(\alpha-\alpha_1)\ge r-s$ and $f(\alpha)=0$.
\end{theorem}

\begin{corollary}\label{cor:hensel}
	If $a,b\in\Zp\setminus p\Zp$ and $k\in\N,k\ge 2$, $a^2\equiv b^2\mod p^k$ if and only if $a\equiv \pm b\mod p^k$ and $p\ne 2$, or $p=2$ and $a\equiv \pm b\mod 2^{k-1}$.
\end{corollary}
\begin{proof}
	The implication to the left is obvious if $p>2$. If $p=2$, we have that $b=a+2^{k-1}t$ for some $t$, so
	\[b^2=a^2+2^kat+2^{2k-2}t^2.\]
	
	To prove the other implication, we apply Hensel's lifting to $f(x)=x^2-a^2$ and $\alpha_1=b$. If $p\ne 2$, we have $r=\ord(b^2-a^2)\ge k$ and $s=\ord(2b)=0$, so there is $\alpha$ with $\ord(\alpha-b)\ge k$ and $\alpha^2-a^2=0$. This implies $\alpha=\pm a$, so $\ord(\pm a-b)\ge k$, as we wanted. The case $p=2$ is similar but with $s=1$.
\end{proof}

The following result is a consequence of the absolute value being non-archimedean.

\begin{proposition}[{\cite[Corollary 4.1.2]{Gouvea}}]\label{prop:sum}
	A series in $\Qp$ converges if and only if the sequence of its terms converges to zero.
\end{proposition}

Now we define some concepts we need concerning the topology of $(\Qp)^n$.
\begin{itemize}
	\item For any $n\in\N$, we define the \textit{$p$-adic norm} on $(\Qp)^n$ by
	\[\|v\|=\max_{1\le i\le n} |v_i|.\]
	\item The \textit{balls} in $(\Qp)^n$ are defined with this norm:
	\[\mathrm{B}_\epsilon(x_0)=\{x\in(\Qp)^n\mid\|x-x_0\|\le \epsilon\}.\]
	The resulting topology in $(\Qp)^n$ is the $n$-th product of the topology in $\Qp$.
	\item For any $n,m\in\N$, the \textit{limit} of a function $f:U\to(\Qp)^m$, where $U$ is an open set in $(\Qp)^n$, at a point $x_0\in \overline{U}$, is equal to $y_0$ if and only if, for any $\epsilon>0$, there is $\delta>0$ such that $f(\mathrm{B}_\delta(x_0)\cap U)\subset \mathrm{B}_\epsilon(y_0)$; we denote this by $\lim_{x\to x_0} f(x)=y_0$.
	\item $f$ is \textit{continuous} at $x_0\in U$ if $\lim_{x\rightarrow x_0} f(x)=f(x_0)$.
	\item $f$ is \textit{continuous} at $U$ if it is continuous in each $x_0\in U$ (this is equivalent to the standard definition of continuous function between two topological spaces).
\end{itemize}

Because of Theorem \ref{thm:topology}, continuous functions look very different from their real counterparts. For example, the functions $x\mapsto \ord(x)$ and $x\mapsto |x|$ are both continuous in $\Qp\setminus\{0\}$, despite having discrete images.

$p$-adic differentiation is defined in analogy to the real case. Let $U\subset(\Qp)^n$ be an open set. (Actually, by Corollary \ref{cor:disjoint}, we can take $U$ to be a ball.)
A function $f:U\to(\Qp)^m$ is \textit{differentiable at $x\in U$} if there is a linear map $\dd f(x):(\Qp)^n\to(\Qp)^m$ such that
\[\lim_{v\rightarrow 0} \frac{\|f(x+v)-f(x)-\dd f(x)(v)\|}{\|v\|}=0.\]

It is easy to check that
If $f:U\to(\Qp)^m$ is differentiable at $x$, then the limit
\[\frac{\partial f}{\partial x_i}(x):=\lim_{t\rightarrow 0}\frac{f(x+t\e_i)-f(x)}{t}\]
exists and
$\dd f(x)(v)=\sum_{i=1}^n\frac{\partial f}{\partial x_i}(x)v_i.$
The derivatives of elementary functions give the same result in the real and $p$-adic cases. For example,
$\frac{\dd}{\dd x}x^n=nx^{n-1}$
and
$\frac{\dd}{\dd x}\sqrt{x}=\frac{1}{2\sqrt{x}}.$
The easiest way of seeing this is just taking the limits:
\[\lim_{t\rightarrow 0}\frac{(x+t)^n-x^n}{t}=\lim_{t\rightarrow 0}(nx^{n-1}+\binom{n}{2}x^{n-2}t+\ldots)=nx^{n-1};\]
\[\lim_{t\rightarrow 0}\frac{\sqrt{x+t}-\sqrt{x}}{t}=\lim_{t\rightarrow 0}\frac{x+t-x}{t(\sqrt{x+t}+\sqrt{x})}=\frac{1}{2\sqrt{x}}.\]

The previous results convey, at a formal level that there is a high degree of similarity between the real and $p$-adic cases. However, upon closer analysis, one realizes that this is not necessarily the case. Indeed, consider the function $f:\Qp\to\Qp$ given by
$f(x)=\sum_{n=\ord(x)}^\infty a_np^{2n}$
where
$x=\sum_{n=\ord(x)}^\infty a_np^n$
is the $p$-adic expansion of $x$. We can check that
$|f(x+t)-f(x)|=p^{-\ord(f(x+t)-f(x))}=p^{-2\ord(t)}=|t|^2$
which implies that the function is continuous, and also that
the function has zero derivative everywhere. In the real case, such a function would necessarily be constant. However, $f$ is not only non-constant, but it is actually injective.

\subsection{$p$-adic initial value problems}

It is not a good idea, at least in principle, to use differentiable functions in general in the context of $p$-adic symplectic geometry: for any differential equation, the solution will not be unique, not even locally, because we could add an injective function with zero derivative to the solution and we will have another solution. The workaround is to restrict to analytic functions.

A \textit{power series} in $(\Qp)^n$ is given by
$f(x)=\sum_{I\in\N^n}a_I(x-x_0)^I$
where $x^I$ means $x_1^{i_1}\ldots x_n^{i_n}$ and $a_I$ are coefficients in $\Qp$. The following result is well-known and will be useful to us later.

\begin{proposition}[{\cite[Proposition 4.2.1]{Gouvea}}]\label{prop:convergence}
	Consider a power series $f$ in one variable in $\Qp$. The convergence radius of the series is given by
	\[\rho=\frac{1}{\limsup\sqrt[i]{|a_i|}}=p^{-r}\quad
	\text{where the convergence order is}\quad
	r=-\liminf\frac{\ord(a_i)}{i}\]
	Then:
	\begin{itemize}
		\item If $\rho=0$ (that is, $r=\infty$), then $f(x)$ converges only when $x=x_0$.
		\item If $\rho=\infty$ (that is, $r=-\infty$), then $f(x)$ converges for every $x\in\Qp$.
		\item If $0<\rho<\infty$ and $\lim_{i\rightarrow\infty}|a_i|\rho^i=0$ (that is, $\lim_{i\rightarrow\infty}\ord(a_i)+ir=\infty$), then $f(x)$ converges if and only if $|x|\le\rho$ (that is, $\ord(x)\ge r$).
		\item If $0<\rho<\infty$ and $|a_i|\rho^i$ does not tend to zero, then $f(x)$ converges if and only if $|x|<\rho$ (that is, $\ord(x)>r$).
	\end{itemize}
\end{proposition}

Let $U\subset (\Qp)^n$ be an open set. A function $f:U\to\Qp$ is \textit{analytic} \cite[page 38]{Schneider} if $U$ can be expressed as
$U=\bigcup_{i\in I}U_i$
where $U_i=x_i+p^{r_i}(\Zp)^n$, for some $x_i\in(\Qp)^n$ and $r_i\in\Z$, and there is a power series $f_i$ converging in $U_i$ such that $f(x)=f_i(x)$ for every $x\in U_i$.

\begin{proposition}[$p$-adic analytic initial value problem]\label{prop:initial}
	Let $U,V$ be open subsets of $\Qp$. An initial value problem, of the form
	\[\left\{\begin{aligned}
		&\frac{\dd y}{\dd x}=f(x,y) \\
		&y(x_0)=y_0
	\end{aligned}\right.\]
	where $f:U\times V\to \Qp$ is analytic, $x_0\in U$ and $y_0\in V$, has an analytic solution in a neighborhood of $x_0$. The solution is locally unique among analytic functions, that is, any other solution coincides with it near $x_0$.
\end{proposition}

\begin{proof}
	We may assume without loss of generality (shrinking $U$ if necessary) that $f$ is given by a power series in $U$, centered at $x_0$. Take
	$y(x)=\sum_{i=0}^\infty a_i(x-x_0)^i.$
	The initial value implies that $a_0=y_0$. The differential equations give
	\[\sum_{i=0}^\infty a_ii(x-x_0)^{i-1}=f(x,\sum_{i=0}^\infty a_i(x-x_0)^i).\]
	The degree $k$ part at the left-hand side gives $(k+1)a_{k+1}$ and the right-hand side gives a polynomial in $a_0,\ldots,a_k$. Hence $a_{k+1}$ is uniquely determined from the previous ones. The resulting $y(x)$ is a solution in a neighborhood of the origin (the intersection of $U$ with the convergence domain of the series), and it is locally unique because any other analytic solution would have the same power series around $x_0$ and so it coincides with $y$ near this point.
\end{proof}

\begin{remark}
	Proposition \ref{prop:initial} implies that we can now speak of ``the solution'' of an analytic differential equation, maybe not in the sense of ``the unique solution'', but in the sense of ``the germ of every solution''.
\end{remark}

\subsection{Properties of trigonometric and exponential series on the $p$\--adic field}\label{app:analytic}
In the $p$-adic setting, one defines the following analogs of the well known sine, cosine, and exponential series in the real case.
\[\exp(x)=\sum_{i=0}^\infty \frac{x^i}{i!};\quad
\cos(x)=\sum_{i=0}^\infty \frac{(-1)^ix^{2i}}{(2i)!};\quad
\sin(x)=\sum_{i=0}^\infty \frac{(-1)^ix^{2i+1}}{(2i+1)!}.\]

\begin{figure}
	\includegraphics[scale=0.5]{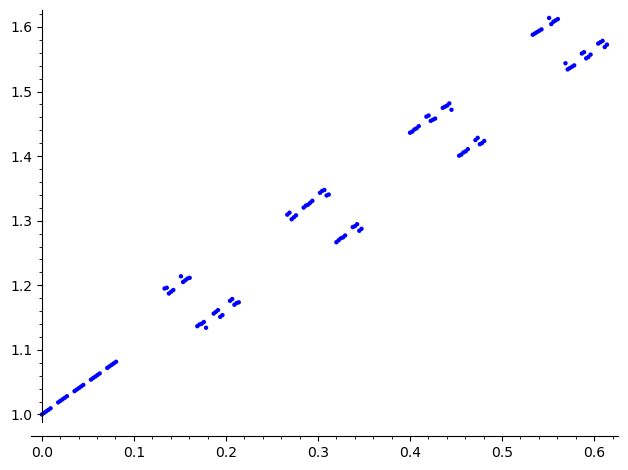}
	\caption{Graphical representation of the exponential series in $\Q_5$. The blue points represent $(x,\exp(x))$ for $x\in 5\Z_5$.}
	\label{fig:exp}
\end{figure}

These expressions, when adequately restricted, define functions which have the following properties (quite different, again, from the real counterparts), which we have not found as explicitly written elsewhere but that for us were essential in Section \ref{sec:circles}, so we prove all of them.

A study of the $p$-adic exponential series may be found in Conrad \cite[Section 4]{Conrad}. The first statement below corresponds to \cite[page 13 and Examples 4.2 and 8.15]{Conrad}. The sine series in the second statement below is mentioned in \cite[page 21]{Conrad} where the author also indicates that the proof that the exponential series converges on $p\mathbb{Z}_p$ for $p= 2$ and on $4\mathbb{Z}_2$ for $p= 2$ also applies to the case of the p-adic sine series; this corresponds to the first part of the second statement below. We have not found the second part of the second statement below or the third statement below explicitly written here or elsewhere.  We do not have a complete explicit description of the image of the cosine series, and hence the claim for this part (third part of statement) refers to an inclusion into a set, not an equality. The proof method we follow below to determine the domain and image of the exponential series is essentially self-contained and the calculations are carried out in a slightly different form (but using essentially the same technique) than how they are presented in \cite{Conrad}.

\begin{proposition}\label{prop:functions} Let $d=2$ if $p=2$, otherwise $d=1$.
	\begin{enumerate}
		\item The domain of the exponential series $\exp$ is $p^d\Zp$. Its image is $1+p^d\Zp$.
		\item The domain of the sine series $\sin$ is $p^d\Zp$. Its image is $p^d\Zp$.
		\item The domain of the cosine series $\cos$ is $p^d\Zp$. Its image is contained in $1+p^{2d-1}\Zp$.
	\end{enumerate}
\end{proposition}
\begin{proof}
	In order to determine the domains of definition of these functions, it suffices to apply Proposition \ref{prop:convergence}. We prove the claim for the exponential and the other two cases are analogous:
	\[r=-\liminf\frac{\ord(a_i)}{i}=\limsup\frac{\ord(i!)}{i}=\limsup\frac{1}{i}\sum_{j=1}^{\infty}\left\lfloor\frac{i}{p^j}\right\rfloor\]\[\le\limsup\frac{1}{i}\sum_{j=1}^{\infty}\frac{i}{p^j}=\limsup\frac{1}{i}\frac{i}{p-1}=\frac{1}{p-1}.\]
	
	Let $r_i=\lfloor\log_p i\rfloor$, so that the sum in $\ord(i!)$ actually ends in $r_i$.
	\[r=\limsup\frac{1}{i}\sum_{j=1}^{r_i}\left\lfloor\frac{i}{p^j}\right\rfloor\ge\limsup\frac{1}{i}\sum_{j=1}^{r_i}\left(\frac{i}{p^j}-1\right)=\limsup\frac{1}{i}\left(\frac{i}{p-1}-r_i\right)=\frac{1}{p-1}.\]
	
	So we have $r=1/(p-1)$. For $p>2$, this implies that the convergence domain of the series is $p\Zp$. For $p=2$, we have $r=1$, and
	$-\ord(i!)+i=i-\sum_{j=1}^{\infty}\left\lfloor\frac{i}{2^j}\right\rfloor.$
	This equals $1$ for an $i$ that is a power of $2$, so it does not tend to infinity, and the domain is $4\Z_2$.
	
	Now we turn to the image. Let
	$x=\sum_{n=\ord(x)}^\infty x_np^n\in\Qp, 0\le x_n\le p-1$
	be the $p$-adic expansion for $x$, and the same for $y$. If $y=\exp (x)$, we have that $\ord(x)\ge d$ and the series implies that $\ord(y)\ge 0$. Now
	\begin{equation}
		\sum_{n=0}^\infty y_np^n =	\sum_{i=0}^\infty\frac{1}{i!}\left(\sum_{n=d}^\infty x_np^n\right)^i\label{eq:exp}
	\end{equation}
	\[
	= \sum_{i=0}^\infty\frac{1}{i!}\sum_{n_1,\ldots,n_i\ge d} x_{n_1}\ldots x_{n_i}p^{n_1+\ldots+n_i}= \sum_{n=0}^\infty p^n\sum_{\substack{n_1,\ldots,n_i\ge d \\ n_1+\ldots+n_i=n}}\frac{1}{i!}x_{n_1}\ldots x_{n_i}.
	\]
	Taking modulo $p$ at both sides, the only term with $n=0$ in the right corresponds to $i=0$, and it is an empty product that equals $1$, and the rest of terms have $n\ge i>\ord(i!)$, which makes them multiples of $p$. So we have $y_0=1$. If $p=2$, we take modulo $4$. The only terms left on the left-hand side are $1+2y_1$. On the right all terms with $n>0$ have automatically $n\ge 2$ and
	$\ord(i!)\le i-1\le n/2-1\le n-2$,
	so they are multiples of $p^2$ and the right-hand side is still $1$. This implies $y_1=0$. In any case, we have proved that the image is contained in $1+p^d\Zp$.
	
	Conversely, suppose that $y\in 1+p^d\Zp$. Then the construction of $x$ with $y=\exp (x)$ can be done inductively on $n$. Supposing that we already have $x_k$ for $k<n$, we will deduce $x_n$. In order to do this, we take mod $p^{n+1}$ on both sides of \eqref{eq:exp} and get in this way that
	$\sum_{k=0}^n y_kp^k\equiv \sum_{k=0}^np^k\sum_{\substack{n_1,\ldots,n_i\ge d \\ n_1+\ldots+n_i=k}}\frac{1}{i!}x_{n_1}\ldots x_{n_i}.$
	By induction hypothesis, the components with $k<n$ on the left coincide modulo $p^n$ with those on the right, so the expression simplifies to
	\[y_np^n\equiv tp^n+\sum_{\substack{n_1,\ldots,n_i\ge d \\ n_1+\ldots+n_i=n}}\frac{1}{i!}x_{n_1}\ldots x_{n_i}p^n\]
	for some $t$. One term in the right-hand side is exactly $x_np^n$, and the rest only depend on $x_k$ for $k<n$, which we already know. This means that we can obtain a unique value for $x_n$, and the induction step is complete.
	
	The case of the sine series is similar to the exponential, but with $y_0=0$, because there is no term of degree $0$ in the sine series.
	
	For the cosine series, we start in the same way, noticing that
	\[\sum_{n=0}^\infty y_np^n=\sum_{n=0}^\infty p^n\sum_{\substack{n_1,\ldots,n_{2i}\ge d \\ n_1+\ldots+n_{2i}=n}}\frac{(-1)^i}{(2i)!}x_{n_1}\ldots x_{n_{2i}}.\]
	For $p\ne 2$, we take modulo $p$, and obtain again $y_0=1$, as we want. For $p=2$, we take modulo $8$. At the right, all terms with $i\ge 1$ have $n\ge 4$ and
	$\ord(2i!)\le 2i-1\le n/2-1\le n-3.$
	This means only $1$ survives from this sum, and $y_0=1,y_1=y_2=0$, as we wanted.
\end{proof}

\section{$p$-adic analytic functions, vector fields and forms}\label{app:vector}
Throughout this section $p$ is a fixed prime number. The content of this section is directly analogous to the real case, we include it here for completeness and also because it gives us the chance to discuss some peculiarities of the $p$-adic case which do not appear in the real case.

Given a $p$-adic analytic manifold $M$, a function $f:M\to\Qp$ is ($p$-adic) \textit{analytic} \cite[page 49]{Schneider} if it is analytic as a map between manifolds, that is, for the charts $\phi$ of $M$, $f|_{U_\phi}\circ \phi^{-1}$ is analytic on $\phi(U_\phi)$. Let $\Omega^0(M)$ be the space of analytic maps $M\to\Qp$.

A \textit{tangent vector} to $q\in M$ is a linear map $v:\Omega^0(M)\to\Qp$ such that
$v(fg)=v(f)g(q)+f(q)v(g)$
for all $f,g\in\Omega^0(M)$. Let $\mathrm{T}_qM$ be the space of tangent vectors to $q$. If $M$ is a $p$-adic analytic manifold, then $\mathrm{T}M$ has also the structure of an analytic manifold. An \textit{analytic vector field} on $M$ is an analytic map $M\to \mathrm{T}M$ that assigns a tangent vector to each point, or equivalently, a linear map $X:\Omega^0(M)\to\Omega^0(M)$ such that
$X(fg)=X(f)g+fX(g)$
for all $f,g\in\Omega^0(M)$. Let $\X(M)$ be the space of vector fields in $M$. An \textit{analytic $k$-form} in $M$ is a linear antisymmetric map
$\alpha:\X(M)^k\to\Omega^0(M).$
Let $\Omega^k(M)$ be the space of $k$-forms in $M$.

The \textit{pullback} $F^*(f)$ of $f\in\Omega^0(N)$ by $F$ and the \textit{push-forward} $F_*(v)$ of a vector $v\in \mathrm{T}_qM$ are defined exactly as in the real case. Similarly, if $F$ is bi-analytic, the \textit{push-forward} of a vector field $X\in \X(M)$ is defined as in the real case and denoted by $F_*(X)\in\X(N)$. Similarly for the \textit{pullback} of a form $\alpha\in\Omega^k(N)$, denoted as usual by $F^*(\alpha)\in\Omega^k(M)$.

In the linear case where $M$ is an open subset of $(\Qp)^n$, as $(\Qp)^n$ is paracompact, analytic functions can be given by a family of power series, each one converging in an element $U$ of a partition of $M$ in open sets (actually, balls):
\[f(x_1,\ldots,x_n)=\sum_{I\in\N^n}a_I(x-x_0)^I\]
for any $x_0\in U$ and some coefficients $a_I$. This allows us to define the vector field $\partial/\partial x_i\in\X(M)$, for $M$ open in $(\Qp)^n$ is given by
\[\frac{\partial}{\partial x_i}(f)=\frac{\partial f}{\partial x_i}=\sum_{I\in\N^n}a_Ii_j(x-x_0)^{I_j}\]
for $f\in\Omega^0(M)$, where $I=(i_1,\ldots,i_n)$ and $I_j$ is defined as $(i_1,\ldots,i_j-1,\ldots,i_n)$.
It follows that for any function $f\in\Omega^0(M)$ and $x_0\in M$, there are functions $g_i\in\Omega^0(M)$ such that
\begin{equation}\label{eq:decomp}
	f(x)=f(x_0)+\sum_{i=1}^n(x_i-x_{0i})g(x)
\end{equation}
and $g(x_0)=\frac{\partial f}{\partial x}(x_0)$. Also,
for any $X\in\X(M)$ and $f\in\Omega^0(M)$, we have that
$X(f)=\sum_{i=1}^n X(x_i)\frac{\partial f}{\partial x_i}$. This says that the vector fields $\partial/\partial x_i$ form a basis of the space $\X(M)$, or locally, that the vectors $(\partial/\partial x_i)_q$ form a basis of the vector space $\mathrm{T}_q(M)$.

In the real case, the proof of the formula \eqref{eq:decomp} is usually done by a method involving integrals (as for example in \cite[Theorem 3.4]{GamRuiz}). One has to be careful if one wants to derive it for $p$-adic smooth functions (which we have not defined, since we are restricting to the analytic case, but the definition is the natural one) because it is more delicate to work with integrals. However, the formula can still be derived without integrals, by using induction on $n$. Supposing it is true for $n$, to prove it for $n+1$ we only need to find $g_{n+1}$ such that
\[f(x_1,\ldots,x_n,x_{n+1})=f(x_1,\ldots,x_n,x_{0,n+1})+(x_{n+1}-x_{0,n+1})g_{n+1}(x_1,\ldots,x_n,x_{n+1}).\]
The formula already gives the value of $g_{n+1}$ for $x_{n+1}\ne x_{0,n+1}$. For $x_{n+1}=x_{0,n+1}$ we take $g_{n+1}$ to be the partial derivative of $f$ with respect to $x_{n+1}$ in those points. By smoothness of $f$, this function is also continuous on those points, and in fact smooth.

Finally, if $M\subset (\Qp)^n$ and $f\in\Omega^0(M)$, the differential form $\dd f$ is defined as the map sending a vector field $X$ to $X(f)$. In particular, if
\[X=\sum_{i=1}^n f_i\frac{\partial}{\partial x_i},\quad f_i:M\to\Qp,\]
we have $\dd x_i(X)=X(x_i)=f_i$. Hence, the $1$-forms $\dd x_i$ are a basis of $\Omega^1(M)$, dual to the basis of $\X(M)$.

All of the previous definitions generalize to the context of $p$-adic analytic manifolds. Indeed, let $M$ be a $p$-adic analytic manifold and $U=U_\phi$ an open set of $M$. The vector field $\partial/\partial x_i\in\X(U)$ is defined as the push-forward by $\phi$ of the corresponding vector field in $(\Qp)^n$, and the $1$-forms $\dd x_i\in\Omega^1(U)$ are defined as the pullback of the corresponding forms in $(\Qp)^n$. (With this definition, the vector fields $\partial/\partial x_i$ and the $1$-forms $\dd x_i$ are again bases of their respective spaces.)

The wedge operation $\wedge$ is defined as usual, and similarly for the differential operator
\[\dd(f\cdot \dd x_I):=\dd f\wedge\dd x_I,\]
extending linearly to all forms, where $\dd x_I$ is shorthand for $\dd x_{i_1}\wedge\ldots\wedge\dd x_{i_k}$. A form is \textit{closed} if its differential is $0$. Also, given a $k$-form $\omega$ and a vector field $X$, $\imath(X)\omega$ is defined as usual.

\section{$p$-adic hamiltonian actions}\label{app:actions}
In Sections \ref{sec:circ-oscillator} and \ref{sec:spin} of the paper we discuss Hamiltonian $(\mathrm{S}_p^1)$-actions, which are essential to understand the Jaynes-Cummings model. In this appendix we briefly recall the concept of $p$-adic Hamiltonian action.

\begin{definition}
	A \textit{$p$-adic analytic Lie group} is a group endowed with a structure of $p$-adic analytic manifold that makes the group operations analytic. The Lie algebra associated to a group is given, as usual, by the left-invariant vector fields, identified by tangent vectors at the identity.
\end{definition}

The first example is $\Qp$, whose Lie algebra is $\Qp$ itself with the Lie bracket equal to $0$. Unlike real Lie groups, we can restrict this: $\Zp$ is still a group, and it contains all the elements near the identity of $\Qp$, hence it has the same Lie algebra, and the same applies to $p^r\Zp$. (Note that the Lie algebra still characterizes locally the group.)

It is a known result about $p$-adic Lie groups that every $p$-adic Lie group is paracompact \cite[Corollary 18.8]{Schneider}. Hence (Theorem \ref{thm:paracompact}) it is an ultrametric space.

In what follows let $G$ be a $p$-adic Lie group with Lie algebra $\g$ and $(M,\omega)$ a $p$-adic symplectic manifold. We have the canonical actions of $G$ in itself:
\[\mathrm{L}_g(h):=gh, \mathrm{R}_g(h):=hg\]
By slight abuse of notation, we call (as usual when speaking of Lie algebras)
\[g\xi:=(\mathrm{L}_g)_*(\xi), \xi g:=(\mathrm{R}_g)_*(\xi)\]
for $g\in G$ and $\xi\in\g$. (These two are tangent vectors at $g$.)

\begin{definition}
	Let $G$ be a $p$-adic analytic Lie group with Lie algebra $\g$ and $(M,\omega)$ a $p$-adic analytic symplectic manifold. A $G$-action $G\times M\to M$ is \textit{analytic} if it is analytic as a map between the manifolds $G\times M$ and $M$. Given such an action, for $g\in G$, we denote by $\psi_g$ the map from $M$ to $M$ sending $q$ to $g\cdot q$. The action is \textit{symplectic} if, for any $g\in G$, $\psi_g$ is a symplectomorphism.
	
	Given a symplectic action of $G$ over $M$, and a vector $\xi\in\g$, we define a vector field $X_\xi$ assigning to each point $q\in M$ the push-forward of $\xi$ by the application $g\mapsto g\cdot q$, that is,
	\[X_\xi(q)(f)=\xi(g\mapsto f(g\cdot q))\]
	for $f\in\Omega^0(M)$.
\end{definition}

\begin{definition}\label{def:momentum-map}
	Let $G$ be a $p$-adic analytic Lie group with Lie algebra $\g$ and $(M,\omega)$ a $p$-adic analytic symplectic manifold. We say that the action is \textit{weakly Hamiltonian} if, for any $\xi\in\g$, the induced vector field $X_\xi$ is Hamiltonian with Hamiltonian function $\mu_\xi:M\to\R$. The \textit{momentum map} of a weakly Hamiltonian action is the map $\mu:M\to\g^*$ defined by $\mu_\xi(q)=\langle\mu(q), \xi\rangle$ for all $\xi\in\g$. The action is called \textit{Hamiltonian} if, for all $\xi\in\g$ and $g\in G$,
	\[\mu_\xi\circ\psi_g=\mu_{g^{-1}\xi g}.\]
	(Note that if $\xi$ is a tangent vector at the identity, $\xi g$ is a tangent vector at $g$ and $g^{-1}\xi g$ is again a tangent vector at the identity.)
\end{definition}

\begin{remark}
	In terms of momentum map, the condition for the action to be Hamiltonian is written as $\mu(g\cdot q)=g\mu(q)g^{-1}$.
	In the well studied case of $G$ being a compact connected Abelian Lie group (a torus) acting on a (real) symplectic manifold, a weakly Hamiltonian action is Hamiltonian if and only if $\mu$ is $G$-invariant. If this torus acts on a compact $M$, which is the setting of the well known convexity theorem of Atiyah \cite{Atiyah} and Guillemin-Sternberg \cite{GuiSte}, then being weakly Hamiltonian is equivalent to being Hamiltonian.
\end{remark}

\end{document}